\documentclass[11pt, reqno]{article}
\usepackage[T1]{fontenc}

\newcommand{\papernumber}{II}
\renewcommand{\thesection}{\papernumber.\arabic{section}}
\usepackage{amsmath,amssymb,amsfonts,amsthm}
\usepackage{color}
\usepackage{xr-hyper}
\usepackage[colorlinks=true,linkcolor=blue,citecolor=blue,urlcolor=blue,pdfborder={0 0 0}]{hyperref}
\usepackage{cleveref}

\usepackage{caption}
\usepackage{thmtools}

\usepackage{mathrsfs}

\crefname{theorem}{Theorem}{Theorems}
\crefname{thm}{Theorem}{Theorems}
\crefname{lemma}{Lemma}{Lemmas}
\crefname{lem}{Lemma}{Lemmas}
\crefname{remark}{Remark}{Remarks}
\crefname{prop}{Proposition}{Propositions}
\crefname{defn}{Definition}{Definitions}
\crefname{corollary}{Corollary}{Corollaries}
\crefname{conjecture}{Conjecture}{Conjectures}
\crefname{question}{Question}{Questions}
\crefname{chapter}{Chapter}{Chapters}
\crefname{section}{Section}{Sections}
\crefname{figure}{Figure}{Figures}
\crefname{example}{Example}{Examples}

\theoremstyle{plain}
\newtheorem{thm}{Theorem}[section]

\newtheorem{lemma}[thm]{Lemma}
\newtheorem{theorem}[thm]{Theorem}

\newtheorem{corollary}[thm]{Corollary}

\newtheorem{prop}[thm]{Proposition}

\theoremstyle{definition}
\newtheorem{defn}[thm]{Definition}

\theoremstyle{remark}
\newtheorem{remark}[thm]{Remark}

\numberwithin{equation}{section}

\renewcommand{\P}{\mathbb P}
\newcommand{\E}{\mathbb E}

\newcommand{\R}{\mathbb R}
\newcommand{\Z}{\mathbb Z}

\newcommand{\cE}{\mathcal E}
\newcommand{\cF}{\mathcal F}
\newcommand{\cG}{\mathcal G}

\def\P{\mathbb{P}}

\DeclareMathSymbol{\leqslant}{\mathalpha}{AMSa}{"36} 
\DeclareMathSymbol{\geqslant}{\mathalpha}{AMSa}{"3E} 
\DeclareMathSymbol{\eset}{\mathalpha}{AMSb}{"3F}     

\newcommand{\Mobius}{\mathrm{M\ddot{o}}}

\makeatletter
\newcommand{\hathat}[1]{%
\begingroup%
  \let\macc@kerna\z@%
  \let\macc@kernb\z@%
  \let\macc@nucleus\@empty%
  \hat{\mathchoice%
    {\raisebox{.2ex}{\vphantom{\ensuremath{\displaystyle #1}}}}%
    {\raisebox{.2ex}{\vphantom{\ensuremath{\textstyle #1}}}}%
    {\raisebox{.16ex}{\vphantom{\ensuremath{\scriptstyle #1}}}}%
    {\raisebox{.14ex}{\vphantom{\ensuremath{\scriptscriptstyle #1}}}}%
    \smash{\hat{#1}}}%
\endgroup%
}
\makeatother

\renewcommand{\epsilon}{\varepsilon}

\usepackage[margin=1in]{geometry}
\usepackage{tikz}
\usepackage{pgfplots}
\usepackage{extarrows}
\usepackage{titling}
\usepackage{commath}
\usepackage{wasysym}
\usepackage{mathtools}
\usepackage{titlesec}
\newcommand{\eps}{\varepsilon}
\usepackage{bbm}
\usepackage{setspace}
\usepackage{enumitem}
\usepackage{booktabs}

\usepackage[font=footnotesize,labelfont=bf]{caption}

\usepackage[textsize=tiny]{todonotes}

\usepackage{cite}

\usepackage{tocloft}
\makeatletter
\pgfdeclareshape{slash underlined}
{
  \inheritsavedanchors[from=rectangle] 
  \inheritanchorborder[from=rectangle]
  \inheritanchor[from=rectangle]{north}
  \inheritanchor[from=rectangle]{north west}
  \inheritanchor[from=rectangle]{north east}
  \inheritanchor[from=rectangle]{center}
  \inheritanchor[from=rectangle]{west}
  \inheritanchor[from=rectangle]{east}
  \inheritanchor[from=rectangle]{mid}
  \inheritanchor[from=rectangle]{mid west}
  \inheritanchor[from=rectangle]{mid east}
  \inheritanchor[from=rectangle]{base}
  \inheritanchor[from=rectangle]{base west}
  \inheritanchor[from=rectangle]{base east}
  \inheritanchor[from=rectangle]{south}
  \inheritanchor[from=rectangle]{south west}
  \inheritanchor[from=rectangle]{south east}
  \inheritanchorborder[from=rectangle]
  \foregroundpath{
    \southwest \pgf@xa=\pgf@x \pgf@ya=\pgf@y
    \northeast \pgf@xb=\pgf@x \pgf@yb=\pgf@y
    \pgf@xc=\pgf@xa
    \advance\pgf@xc by .5\pgf@xb
    \pgf@yc=\pgf@ya
    \advance\pgf@xc by -1.3pt
    \advance\pgf@yc by -1.8pt
    \pgfpathmoveto{\pgfqpoint{\pgf@xc}{\pgf@yc}}
    \advance\pgf@xc by  2.6pt
    \advance\pgf@yc by  3.6pt
    \pgfpathlineto{\pgfqpoint{\pgf@xc}{\pgf@yc}}
    \pgfpathmoveto{\pgfqpoint{\pgf@xa}{\pgf@ya}}
    \pgfpathlineto{\pgfqpoint{\pgf@xb}{\pgf@ya}}
 }
}
\tikzset{nomorepostaction/.code=\let\tikz@postactions\pgfutil@empty}

\makeatother

\pretitle{\begin{flushleft}\Large}
\posttitle{\par\end{flushleft}\vskip 0.5em
}
\preauthor{\begin{flushleft}}
\postauthor{
\\
\vspace{0.5em}
\footnotesize{The Division of Physics, Mathematics and Astronomy, California Institute of Technology\\ Email: 
\href{mailto:t.hutchcroft@caltech.edu}{t.hutchcroft@caltech.edu}}
\end{flushleft}}
\predate{\begin{flushleft}}
\postdate{\par\end{flushleft}}
\title{{\bf Critical long-range percolation II: Low effective dimension}}

\renewenvironment{abstract}
 {\par\noindent\textbf{\abstractname.}\ \ignorespaces}
 {\par\medskip}

\titleformat{\section}
  {\normalfont\large \bf}{\thesection}{0.4em}{}

\author{{\bf Tom Hutchcroft}}

\makeatletter
\newcommand\mytag[2][]{%
  \def\@currentlabel{#2}%
  (#2)\label{#1} 
}
\makeatother

\begin{document}

\date{\small{\today}}

\maketitle

\begin{abstract}
In long-range Bernoulli bond percolation on the $d$-dimensional lattice $\Z^d$, each pair of points $x$ and $y$ are connected by an edge with probability $1-\exp(-\beta\|x-y\|^{-d-\alpha})$, where $\alpha>0$ is fixed, $\beta \geq 0$ is the parameter that is varied to induce a phase transition, and $\|\cdot\|$ is a norm. As $d$ and $\alpha$ are varied, the model is conjectured to exhibit eight qualitatively different forms of second-order critical behaviour, with a transition between a mean-field regime and a low-dimensional regime satisfying the  hyperscaling relations when $d=\min\{6,3\alpha\}$, a transition between effectively long- and short-range regimes at a crossover value $\alpha=\alpha_c(d)$, and with various logarithmic corrections to these behaviours occurring at the boundaries between these regimes. 

This is the second of a series\footnote{This paper is intended to be readable independently of the rest of the series. An expository account of the series over 12 hours of lectures is available on the YouTube channel of the Fondation Hadamard, see \url{https://www.youtube.com/playlist?list=PLbq-TeAWSXhPQt8MA9_GAdNgtEbvswsfS}.}
 of three papers developing a rigorous and detailed theory of the model's critical behaviour in five of these eight regimes, including all effectively long-range (LR) regimes and all effectively high-dimensional (HD) regimes.
In this paper we focus on the LR-LD regime $d/3<\alpha<\alpha_c(d)$, in which the model is effectively long-range but \emph{below} its upper critical dimension. Since computing $\alpha_c(d)$ for $2<d<6$ appears to be outside the scope of current techniques, we instead give an axiomatic definition of what it means to belong to the effectively long-range regime, which we prove holds when $\alpha <1$. With this definition in hand, 
we prove up-to-constants estimates on the critical and slightly subcritical two-point function throughout the LR regime and up-to-constants estimates on the critical volume tail and $k$-point function throughout the LR-LD regime.
%
We deduce in particular that the critical exponents $\eta$, $\gamma$, $\Delta$, $d_f$, and $\nu$ satisfy the identities
\[
 \eta = 2-\alpha, \qquad \gamma = (2-\eta)\nu, \qquad \text{ and } \qquad \Delta = \nu d_f
\]
throughout the LR regime (assuming that at least one of $\gamma$, $\nu$, or $\Delta$ is well-defined) and that the critical exponents $\delta$ and $d_f$ are given by the hyperscaling identities
\[
   \delta = \frac{d+\alpha}{d-\alpha} \qquad \text{ and } \qquad d_f = \frac{d+\alpha}{2}
\]
throughout the LR-LD regime. 
Our results also give a first glimpse of \emph{conformal invariance} in the LR-LD  regime, yielding in particular that the critical $k$-point function in this regime coincides up-to-constants with an explicit M\"obius-covariant function.
\end{abstract}

\newpage
\setcounter{tocdepth}{2}

\tableofcontents

\setstretch{1.1}

\newpage

\section{Introduction}
\label{sec:introduction}

In \textbf{long-range percolation} on $\Z^d$, 
 we define a random graph with vertex set $\Z^d$ by setting any two distinct vertices $x,y\in \Z^d$ to be connected by an edge $\{x,y\}$ with probability $1-e^{-\beta J(x,y)}$ independently of all other pairs, where $J:\Z^d\times \Z^d \to [0,\infty)$ is a symmetric, translation invariant\footnote{Meaning that $J(x,y)=J(y,x)=J(0,y-x)$ for every distinct $x,y\in \Z^d$.} kernel and $\beta\geq 0$ is a parameter that is varied to induce a phase transition. We will 
 be interested in the case that the kernel $J$ has power-law decay of the form
\begin{equation}
\label{eq:kernel_simple_intro}
  J(x,y) \sim \text{const.}\|x-y\|^{-d-\alpha}
\end{equation}
for some $\alpha>0$ and norm $\|\cdot\|$, and in fact will work under the slightly stronger\footnote{In fact for the effectively low-dimensional results of this paper it suffices that $|J'(r)|\sim C r^{-d-\alpha-1}$, without any assumptions on the rate of convergence. We retain the stronger conditions used in the other papers of the series to allow for a uniform treatment of our results concerning the entire effectively long-range regime, some of which rely on the high and critical-dimensional results proven in the other papers of the series.} assumption that $J(x,y)= \mathbbm{1}(x\neq y) J(\|x-y\|)$ for some decreasing, differentiable function $J(r)$ satisfying
\begin{equation}
  |J'(r)| = C(1+\delta_r) r^{-d-\alpha-1}
\end{equation}
for some constant $C>0$ and measurable error function $\delta_r$ satisfying $\delta_r\to 0$ as $r\to \infty$ and $\int_{r_0}^\infty \frac{|\delta_r|}{r}\dif r<\infty$ for some $r_0<\infty$.
 The \textbf{critical point} $\beta_c=\beta_c(J)$ is defined by
\[
  \beta_c = \inf\{\beta \geq 0: \P_\beta(\text{an infinite cluster exists})>0\},
\]
where we write $\P_\beta$ and $\E_\beta$ for probabilities and expectations taken with respect to the law of the model with parameter $\beta$. 
For kernels of the form \eqref{eq:kernel_simple_intro}, this critical point is non-trivial in the sense that $0<\beta_c<\infty$ if and only if $d\geq 2$ and $\alpha>0$ or $d=1$ and $0<\alpha \leq 1$ \cite{newman1986one,schulman1983long}. 
In this paper we will be interested exclusively in the study of the model \emph{at and near its critical point} $\beta=\beta_c$,
referring the reader to e.g.\ \cite{biskup2021arithmetic,ding2023uniqueness,baumler2023distances,jorritsma2024cluster,berger2024scaling,aoun2021sharp} for results on other aspects of the model.

One interesting feature of this model (and indeed long-range models more broadly) is that varying $\alpha$ is in some ways analogous to varying the dimension $d$. This phenomenon is captured in part\footnote{We caution the reader not to rely too strongly on the effective dimension: It gives the correct prediction that the model is mean-field when $d_\mathrm{eff}>d_c=6$ and satisfies the hyperscaling relations when $d_\mathrm{eff}<d_c=6$, but is not predicted to fully determine the critical behaviour of the model in the effectively low-dimensional regime. In particular, for $d<6$ the crossover from effectively long-range to effectively short-range behaviour is \emph{not} predicted to occur at the point $\alpha=2$ where the equality $d=d_\mathrm{eff}$ begins to hold.} by the \textbf{effective dimension} 
\[d_\mathrm{eff} := \max\left\{d, \frac{2d}{\alpha}\right\},\]
which is the spectral dimension of the random walk with jump kernel proportional to $\|x-y\|^{-d-\alpha}$. 
In particular, the model is predicted to exhibit different critical behaviours according to whether $d_\mathrm{eff}$ is greater than, smaller than, or equal to the \textbf{upper critical dimension} $d_c=6$: When $d_\mathrm{eff}>d_c$ the model is predicted to exhibit relatively simple, \emph{mean-field} critical behaviour, while when
$d_\mathrm{eff}<d_c$ the model is predicted to have more complex critical behaviour characterised by the validity of the \emph{hyperscaling relations} (see \cref{subsec:hyperscaling}). Besides this transition between the effectively low-dimensional and high-dimensional regimes, long-range percolation is also predicted to exhibit a transition between an \emph{effectively long-range} (small $\alpha$) regime and an \emph{effectively short-range} (large $\alpha$) regime, demarcated by a \emph{crossover value} $\alpha_c(d)$ \cite{sak1973recursion,brezin2014crossover,paulos2016conformal}.  One also expects further, subtly different critical behaviours to occur at the boundaries of these regimes, leading to at least eight qualitatively different forms of critical behaviour in total (\cref{fig:cartoon}).

\begin{figure}[t]
\centering
\includegraphics{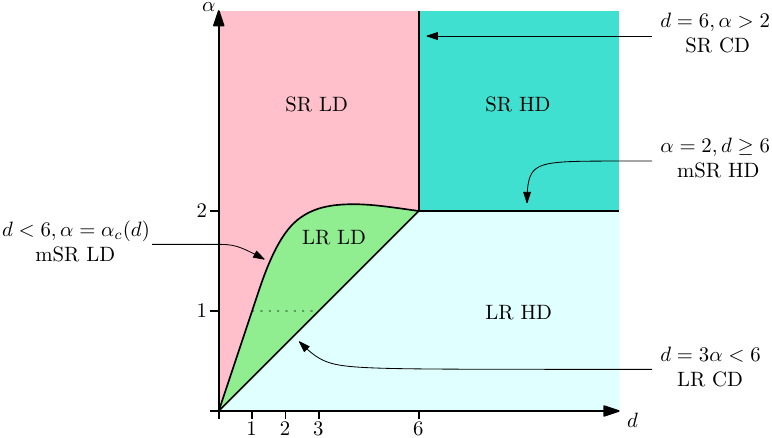}\hspace{1cm}
\caption{Schematic illustration of the different regimes of critical behaviour for long-range percolation. LR, SR, HD, LD, and CD stand for ``Long Range'', ``Short Range'', ``High Dimensional'', ``Low Dimensional'', and ``Critical Dimensional'' respectively, while mSR stands for ``marginally Short Range''. Here we ignore the special behaviours occuring when $d=1$ (where there is either no phase transition when $\alpha>1$ or a discontinuous phase transition when $\alpha=1$ \cite{MR868738,duminil2020long}) for clarity. In this paper we primarily study critical behaviour in the effectively \emph{long-range, low-dimensional} regime LR LD, although we also prove several results applying to the entire effectively long-range regime LR. Our results are proven to apply to the region $d<3\alpha <3$, i.e., the part of the regime LR LD below the dotted line; for other values of $d$ and $\alpha$ we \emph{define} what it means to belong to the regime LR via the condition \eqref{CL} in \cref{def:CL} and prove our results subject to this condition.
 We do not expect the breakdown of our proof method to correspond to any true change in behaviour of the critical model when $\alpha=1$ for $d \geq 2$.}
\label{fig:cartoon}
\end{figure}



This paper is the second in a series of three papers building a detailed and rigorous theory of critical phenomena for long-range percolation in all regimes other than the effectively short-range low-dimensional regime and its boundaries.  This work builds on the earlier works \cite{hutchcroft2020power,hutchcroft2022sharp,hutchcroft2024pointwise,baumler2022isoperimetric,hutchcrofthierarchical} and closely parallels our paper \cite{hutchcroft2022critical} which developed a similar theory for \emph{hierarchical percolation} (which we do not assume familiarity with); working in the Euclidean setting entails both significant additional technical challenges and new phenomena that need to be understood, since hierarchical models are always ``effectively long-range'' and do not exhibit the transition to ``effectively short-range'' behaviour seen in Euclidean models. The focus of this paper is on the effectively long-range, \emph{low-dimensional} regime (where the model is \emph{below} its upper critical dimension) while the first and third papers focus on the effectively high-dimensional and critical-dimensional regimes respectively 
 \cite{LRPpaper1,LRPpaper3}. Despite this focus, several of the results proven in this paper also have important consequences for models of high and critical effective dimension,
  including the pointwise two-point function estimates of \cref{thm:CL_Sak} and the slightly subcritical scaling theory developed in \cref{subsec:subcritical_intro}. Although our papers rely on a form of real-space \emph{renormalization group} (RG) analysis (as explained in detail in \cref{I-subsec:RG}), we stress that all our results  are \emph{non-perturbative} (i.e., do not require any ``small parameter assumptions'') in contrast to both traditional approaches to \emph{high-dimensional} percolation via the lace expansion \cite{MR1043524,MR2430773} and the traditional RG approach to low dimensional critical phenomena in which one expands perturbatively around the upper critical dimension \cite{wilson1972critical,MR3772040,MR3723429,abdesselam2013rigorous,MR709462,bleher2010critical}.
 

\medskip\noindent \textbf{Asymptotic notation.}
We write $\asymp$, $\preceq$, and $\succeq$ for equalities and inequalities that hold to within positive multiplicative constants depending on $d\geq 1$, $\alpha>0$, and the kernel $J$
but not on any other parameters. If implicit constants depend on an additional parameter (such as the index of a moment), this will be indicated using a subscript. Landau's asymptotic notation is used similarly, so that if $f$ is a non-negative function then ``$f(n)=O(n)$ for every $n\geq 1$'' and ``$f(n) \preceq n$ for every $n\geq 1$'' both mean that there exists a positive constant $C$ such that $f(n)\leq C n$ for every $n\geq 1$. We also write $f(n)=o(g(n))$ to mean that $f(n)/g(n)\to 0$ as $n\to\infty$ and write $f(n)\sim g(n)$ to mean that $f(n)/g(n)\to 1$ as $n\to\infty$. The quantities we represent implicitly using big-$O$ and little-$o$ notation are always taken to be non-negative, and we denote quantities of uncertain sign using $\pm O$ or $\pm o$ as appropriate. 





\medskip

\noindent
\textbf{Sak's prediction and its consequences.}
Understanding critical phenomena in intermediate dimensions is a notoriously difficult problem. While critical nearest-neighbour percolation is now well-understood in high dimensions \cite{MR1043524,MR2748397,MR762034,MR1127713,MR1959796,MR3306002,MR4032873,fitzner2015nearest} (see also \cite{MR2430773,MR3306002,MR4032873,liu2025high} for related results for long-range percolation) and is reasonably well-understood in two dimensions\footnote{The main caveat being that results concerning scaling limits and the computation of critical exponents for two-dimensional nearest-neighbour percolation remain limited to  \emph{site percolation on the triangular lattice} and universality remains a major challenge.} \cite{smirnov2001critical2,smirnov2001critical,lawler2002one,MR879034,duminil2020rotational,camia2024conformal}, progress in intermediate dimensions $2<d<6$ and the critical dimension $d=6$ has been extremely limited. For intermediate dimensions the situation is similar throughout mathematical physics more broadly,
 with e.g.\ three-dimensional critical phenomena only\footnote{One possible counterexample to this claim is the work of Bauerschmidt, Lohmann, and Slade \cite{bauerschmidt2020three} regarding \emph{tricritical points} in three dimensions. In contrast to the other models discussed above, these tricritical models continue to behave in a ``marginally trivial'' fashion in three dimensions and can be understood using techniques typically used to analyze models at their upper critical dimension. Another important example is the three-dimensional loop-erased random walk and uniform spanning tree, where we know that scaling limits and critical exponents are well-defined but not how to compute or describe them \cite{shiraishi2018growth,angel2021scaling,kozma2007scaling}.} understood in a handful of ``exactly solvable'' examples such as branched polymers \cite{brydges2003branched,kenyon2009branched} and GFF level-set percolation \cite{drewitz2023critical,werner2021clusters,lupu2016loop}. 

One of several underlying difficulties in the study of intermediate-dimensional critical phenomena for nearest neighbour percolation is that no exact values are conjectured for critical exponents, and indeed there seems to be no reason to believe that closed-form expressions for these exponents should exist. One of the most appealing features of \emph{long-range} percolation (and long-range models more broadly) is that some (but not all) of their critical exponents are predicted to have simple explicit values throughout the effectively long range, low dimensional regime. Indeed, it is a prediction of Sak \cite{sak1973recursion} (who worked in the context of the $O(n)$ model and built on the earlier works \cite{fisher1972critical,suzuki1972wilson}) that the two-point function critical exponent $\eta$, defined (if it exists) via
  $\P_{\beta_c}(x\leftrightarrow y) = \|x-y\|^{-d+2-\eta \pm o(1)}$,
satisfies
\begin{equation}
\label{eq:Sak_Stick}
2-\eta=\alpha\end{equation}
throughout the entire effectively long-range regime. (This prediction is sometimes interpreted as the two-point function ``sticking'' to its mean-field scaling beyond the mean-field regime.) It is also predicted that the exponent $\eta$ coincides with its nearest neighbour-value for $\eta_\mathrm{SR}=\eta_\mathrm{SR}(d)$ throughout the effectively short-range regime, so that
\begin{equation}
\label{eq:Sak}
  2-\eta = \begin{cases} \alpha & \alpha < \alpha_c(d) \qquad \qquad \text{(LR)}
  \\
  2-\eta_{\mathrm{SR}} & \alpha \geq \alpha_c(d) \qquad \qquad \text{(SR and mSR)}
  \end{cases}
\end{equation}
where the crossover value $\alpha_c(d)=2-\eta_\mathrm{SR}$ is the unique value making this function continuous. 
(For $d=1$ there is no crossover to effectively short-range behaviour and the prediction $2-\eta=\alpha$ applies to the entire interval $\alpha\in (0,1)$ in which a continuous phase transition occurs.)
Hiding in the apparent simplicity of this prediction is the complicated value of $\alpha_c(d)=2-\eta_\mathrm{SR}$, which is predicted to be $43/24$ when $d=2$ \cite{smirnov2001critical} and has numerical values\footnote{See \url{https://en.wikipedia.org/wiki/Percolation_critical_exponents} for a summary of numerical and non-rigorous analytic estimates for $\eta_{\mathrm{SR}}(d)$.} slightly above $2$ in dimensions $3$, $4$, and $5$; the rigorous computation of $\alpha_c(d)$ appears to be out of reach of current techniques even for $d=2$, and is morally related to very strong forms of universality for short-range models.
See \cite{brezin2014crossover,luijten1997interaction,gori2017one,behan2017scaling} for further discussion of Sak's prediction in the physics literature (where it has been the subject of some controversy \cite{van1982instability,picco2012critical,brezin2014crossover}) and \cite{MR3772040,MR3723429} for rigorous partial progress on the spin $O(n)$ model.

Recently, significant partial progress has been made on Sak's prediction for long-range percolation \cite{hutchcroft2022sharp,hutchcroft2024pointwise,baumler2022isoperimetric}, with the pointwise estimate 
\begin{equation}
\label{eq:two_point_pointwise}
\P_{\beta_c}(x\leftrightarrow y) \asymp \|x-y\|^{-d+\alpha}
\end{equation}
proven to hold whenever $\alpha<1$ \cite[Theorem 1.4]{hutchcroft2024pointwise}. This completely verifies Sak's prediction for $d=1$ and treats the range of low-effective-dimensional models with $\alpha\in (2/3,1)$  for $d=2$. Moreover, the two-point function estimate of Sak's prediction has been shown in \cite{hutchcroft2022sharp} to always hold as an \emph{upper bound} in the spatially averaged sense
\begin{equation}
\label{eq:Sak_upper_restate}
  \frac{1}{r^d}\sum_{x\in [-r,r]^d}\P_{\beta_c}(0\leftrightarrow x) \preceq r^{-d+\alpha},
\end{equation}
this bound holding vacuously for $\alpha \geq d$. (See \cite{baumler2022isoperimetric} for the current best results regarding unconditional \emph{lower bounds} for general $d$ and $\alpha$.)

  Unfortunately, both the methods and results of the papers \cite{hutchcroft2022sharp,hutchcroft2024pointwise,baumler2022isoperimetric,hutchcroft2020power} (along with Sak's prediction itself) are insensitive to the distinction between low and high effective dimensions and seem unsuitable for the study of other aspects of critical long-range percolation 
   where this distinction is important. Two important examples of exponents that are sensitive to this distinction are the \emph{volume tail exponent} $\delta$ and \emph{fractal dimension exponent} $d_f$ defined (if they exist) by 
   \[
     \P_{\beta_c}(|K|\geq n) = n^{-1/\delta \pm o(1)} \qquad \text{ and } \qquad 
  \frac{\E_{\beta_c}|K\cap B_r|^2}{\E_{\beta_c}|K \cap B_r|} = r^{d_f\pm o(1)}
   \]
as $n,r\to\infty$, where we write $B_r=\{x\in \Z^d: \|x\|\leq r\}$. (This is just one of several possible precise definitions of $d_f$; we will discuss others in \cref{subsec:hyperscaling}.)
Using Sak's prediction \eqref{eq:Sak}, standard heuristic hyperscaling theory (discussed in detail in \cref{subsec:hyperscaling}), and the mean-field values of $\delta$ and $d_f$ discussed in more detail in \cref{I-sec:introduction}, one obtains the prediction that if $1<d<6$ then
\begin{equation}
  (\delta, d_f)  
  \arraycolsep=10pt\def\arraystretch{1.8}
   = \left\{\!\!\!\!\!\!\begin{array}{lll} \Bigl(2,2\alpha\Bigr) & \alpha \leq d/3
   &\text{(LR HD and LR CD)}\\
  \Bigl(\frac{d+\alpha}{d-\alpha},\frac{d+\alpha}{2}\Bigr) & d/3 < \alpha < \alpha_c(d) &\text{(LR LD)} \\
 \Bigl(\delta_\mathrm{SR}, d_{f,\mathrm{SR}}\Bigr) & \alpha \geq \alpha_c(d), &\text{(SR LD and mSR LD)}
  \end{array}\right.
  \label{eq:Sak_delta}
\end{equation}
where $\delta_\mathrm{SR}=\delta_\mathrm{SR}(d)$ and $d_{f,\mathrm{SR}}=d_{f,\mathrm{SR}}(d)$ are the corresponding exponents for the nearest-neighbour model. (See \cref{fig:2d}.) For $d=1$ the prediction is similar except that there is no phase transition for $\alpha >\alpha_c(1)=1$ so that the curve has only two pieces. (The $\alpha\leq d/3$ part of this prediction is verified in Theorems \ref{I-thm:hd_moments_main}, \ref{I-thm:critical_dim_moments_main_slowly_varying}, and \ref{III-thm:critical_dim_moments_main}, including precise logarithmic corrections to scaling when $d=3\alpha<6$.) A further related problem is the computation of the $k$-point function $\tau_{\beta_c}(x_1,\ldots,x_k) := \P_{\beta_c}(x_1,\ldots,x_k$ all connected$)$ which, as we explain in detail in \cref{subsec:glimpse_of_conformal_invariance_,remark:HD_spread}, is expected to have different asymptotic forms above, at, and below the upper critical dimension when $k\geq 3$. 

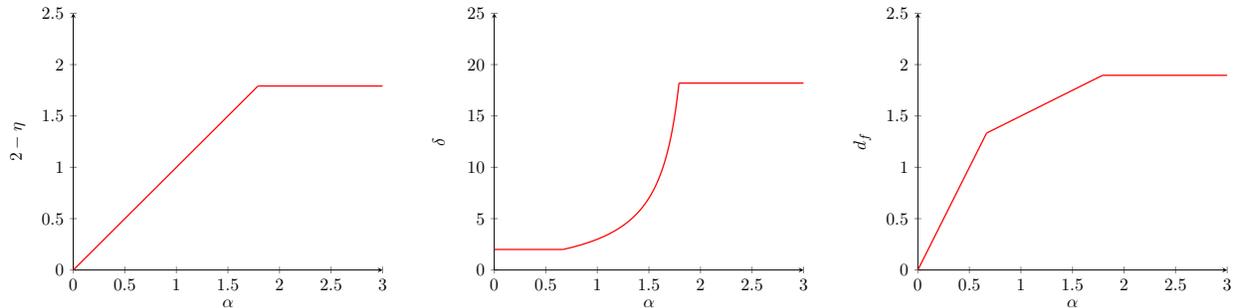
\begin{figure}[t]
\centering
\begin{tikzpicture}[scale=0.6]
\begin{axis}[
    axis lines = left,
    xlabel = $\alpha$,
    ylabel = {$2-\eta$},
    xmin=0,xmax=3,
    ymax=2.5
]
\addplot [
    domain=0:43/24, 
    samples=100, 
    color=red,
    thick
]
{x};
\addplot [
    domain=43/24:3, 
    samples=100, 
    color=red,
    thick
]
{43/24};
\end{axis}

\end{tikzpicture}
    \hspace{0.2cm}
    \begin{tikzpicture}[scale=0.6]
\begin{axis}[
    axis lines = left,
    xlabel = $\alpha$,
    ylabel = {$\delta$},
    xmin=0, xmax=3,
    ymin=0, ymax=25
]
\addplot [
    domain=0:2/3, 
    samples=100, 
    color=red,
    thick
]
{2};
\addplot [
    domain=2/3:43/24, 
    samples=100, 
    color=red,
    thick
]
{(2+x)/(2-x)};
\addplot [
    domain=43/24:3, 
    samples=100, 
    color=red,
    thick
]
{91/5};
\end{axis} 
\end{tikzpicture}
\hspace{0.2cm}
    \begin{tikzpicture}[scale=0.6]
\begin{axis}[
    axis lines = left,
    xlabel = $\alpha$,
    ylabel = {$d_f$},
    xmin=0, xmax=3,
    ymin=0, ymax=2.5
]
\addplot [
    domain=0:2/3, 
    samples=100, 
    color=red,
    thick
]
{2*x};
\addplot [
    domain=2/3:43/24, 
    samples=100, 
    color=red,
    thick
]
{(2+x)/2};
\addplot [
    domain=43/24:3, 
    samples=100, 
    color=red,
    thick
]
{(2+43/24)/2};
\end{axis} 
\end{tikzpicture}
\hspace{0.1cm}
    \caption{Sak's predicted value of $2-\eta$ (left) for $d=2$ and its consequences for $\delta$ (center) and $d_f$ (right) assuming the validity of the hyperscaling relations for $d_\mathrm{eff}<6$ (right). The predicted crossover value $\alpha_c(2)=43/24$ arises as $2-\eta_{\mathrm{SR}}$ with $\eta_{\mathrm{SR}}=5/24$, $\delta_\mathrm{SR}=91/5$, and $d_{f,\mathrm{SR}}=91/48$.}
    \label{fig:2d}
    \vspace{-1em}
\end{figure}


\medskip

\noindent 
\textbf{Results in the case $\alpha<1$.}
Two of the main results of this paper establish up-to-constants estimates on the volume tail $\P_{\beta_c}(|K|\geq n)$ and the $k$-point function $\tau_{\beta_c}(x_1,\ldots,x_k)$ throughout the effectively long-range, low-dimensional regime (LR LD of \cref{fig:cartoon}). In contrast to the two-point function, both quantities exhibit different asymptotic behaviour in this regime than they do at or above the upper critical dimension. This
difference is already apparent in the simplest versions of our results, restricting to the case $\alpha<1$ and to the three-point function rather than general $k$-point functions, which we now state. Further discussion and motivation of these theorems is given throughout the remainder of the introduction.

\begin{thm}
\label{thm:volume_tail_simple}
If $\alpha<1$ and $d<3\alpha$ then 
\[\P_{\beta_c}(|K|\geq n) \asymp n^{-(d-\alpha)/(d+\alpha)} \qquad \text{ and } \qquad \E_{\beta_c}|K\cap B_r|^p \asymp_p r^{\alpha+(p-1)\frac{d+\alpha}{2}}
\] for every $n,r\geq 1$ and positive integer $p$. In particular, the critical exponents $\delta$ and $d_f$ are well-defined and given by $\delta=(d+\alpha)/(d-\alpha)$ and $d_f=(d+\alpha)/2$ respectively.
\end{thm}

\begin{thm}
\label{thm:k_point_simple}
If $\alpha<1$ and $d<3\alpha$ then 
\[
  \tau_{\beta_c}(x,y,z) \asymp \sqrt{\|x-y\|^{-d+\alpha}\|y-z\|^{-d+\alpha}\|z-x\|^{-d+\alpha}}
\]
for all distinct $x,y,z\in \Z^d$.
\end{thm}

Both of these theorems apply to the entire effectively long-range low-dimensional regime $\alpha\in (1/3,1)$ when $d=1$.
To state theorems concerning the entire regime LR LD for $d=2,3,4,5$ we will first need to give a precise \emph{definition} of what it means to belong to this regime (the rigorous computation of its upper boundary $\alpha_c(d)$ appearing to be beyond the scope of current techniques), which we do in \cref{subsec:defining_the_effectively_long_range_regime}.

\medskip

\noindent \textbf{Organization of the introduction.}
The rest of this introduction is structured as follows.
  In \cref{subsec:defining_the_effectively_long_range_regime} we recall the real-space renormalization group framework introduced in \cref{I-subsec:definitions}, use this framework to formulate our axiomatic definition of the effectively long-range regime via the condition \eqref{CL}, and state our results on the values of $\alpha$ and $d$ where the model is proven to be effectively long-range in this sense. 
  In \cref{subsec:main_results_low_dim} we state our main results concerning critical exponents and $k$-point functions in the effectively long-range low-dimensional regime and the validity of Sak's prediction \eqref{eq:CL_Sak} throughout the entire effectively long-range regime (where we obtain new results even for models of high and critical effective dimension). In \cref{subsec:subcritical_intro} we then state our results concerning slightly subcritical two-point functions and volume moments, yielding the scaling relations $\gamma=(2-\eta)\nu$ and $\Delta=d_f\nu$. Finally, in \cref{subsec:glimpse_of_conformal_invariance_} we explain how our $k$-point estimates relate to the conjectured conformal invariance of critical long-range percolation.

\subsection{Defining the effectively long-range regime}
\label{subsec:defining_the_effectively_long_range_regime}

In this section we give an axiomatic definition of what it means to belong to the effectively long-range regime and state the generalization of \cref{thm:volume_tail_simple} to this entire regime. 

We begin by recalling our normalization assumptions on the kernel $J$ and the real-space renormalization group framework introduced in \cref{I-subsec:definitions}.
As explained above, we will assume that our kernel is of the form $J(x,y)=\mathbbm{1}(x\neq y)J(\|x-y\|)$ for some decreasing, differentiable function $J(r)$ satisfying
\[
|J'(r)| = C(1+\delta_r) r^{-d-\alpha-1}
\]
for some constant $C>0$ and measurable function $r\mapsto \delta_r$  such that  $\delta_r\to 0$ as $r\to \infty$ and $\int_{r_0}^\infty |\delta_r| \frac{\dif r}{r}<\infty$ for some $r_0<\infty$.
 \textbf{We may assume without loss of generality that the unit ball in $\|\cdot\|$ has unit Lebesgue measure and that $C=1$, so that \begin{equation}
\label{eq:normalization_conventions}
\tag{$*$}
|J'(r)|\sim r^{-d-\alpha-1} \qquad \text{ and } \qquad |B_r|:=|\{x\in \Z^d:\|x\|\leq r\}|\sim r^d,\end{equation}and will do this throughout the paper.} (Any other norm and kernel of the form we consider can be normalized to satisfy \eqref{eq:normalization_conventions}, which changes the value of $\beta_c$ but does not change other features of the model at criticality.)
%
%
%
%
%
Given a kernel $J$ as above and a parameter $r>0$, we write $J_r$ for the \textbf{cut-off kernel}
\[
J_r(x,y) := \mathbbm{1}(x\neq y, \|x-y\|\leq r) \int_{\|x-y\|}^r |J'(s)| \dif s
\]
and write $\P_{\beta,r},\E_{\beta,r}$ for probabilities and expectations taken with respect to the law of long-range percolation on $\Z^d$ with kernel $J_r$ at parameter $\beta$. We will always write $\beta_c$ for the critical parameter associated to the original kernel $J$, so that the cut-off measure $\P_{\beta_c,r}$ is subcritical at $\beta_c$ when $r<\infty$. (Indeed, for every $r<\infty$ there exists $\eps=\eps(r)>0$ such that $J_r(x,y)\leq (1-\eps)J(x,y)$ for every $x\neq y$, so that $\P_{\beta_c,r}$ is stochastically dominated by the subcritical measure $\P_{(1-\eps)\beta_c}$.)

As explained in detail in \cref{I-subsec:definitions}, our method (which can be thought of as a real-space renormalization group method; see \cref{I-subsec:RG}) is based on understanding the asymptotics of various quantities associated to the measures $\E_{\beta_c,r}$ as $r\to \infty$, such as the moments $\E_{\beta_c,r}|K|^p$ as well as spatially-weighted moments like $\E_{\beta_c,r}\sum_{x\in K}\|x\|_2^2$. In the effectively high- and critical-dimensional regimes treated in \cite{LRPpaper1,LRPpaper3}, this is done by establishing the validity of ``mean-field'' asymptotic ODEs such as
\begin{equation}
\label{eq:volume_moments_ODE_intro}
  \frac{d}{dr}\E_{\beta_c,r}|K|^p \sim \beta_c r^{-\alpha-1} \sum_{\ell=0}^{p-1}\binom{p}{\ell} \E_{\beta_c,r}|K|^{\ell+1}\E_{\beta_c,r}|K|^{p-\ell}.
\end{equation}
In contrast, the low-dimensional analysis of this paper will hinge on the fact that the mean-field relationship \eqref{eq:volume_moments_ODE_intro} \emph{cannot} hold in the effectively long-range, low-dimensional regime, which we will argue is actually sufficient to \emph{determine} several exponents in light of the one-sided estimates established in \cite{hutchcroft2022sharp}. Thus, there is an interesting sense in which our analysis of the effectively \emph{low-dimensional} regime builds directly on our \emph{high-dimensional} analysis, which we apply in the context of a proof by contradiction.


We now define the correlation length condition for effectively long-range critical behaviour \eqref{CL}, the assumption under which our low-effective-dimensional results will hold, beginning with some motivation. In \cref{I-prop:radius_of_gyration,I-prop:displacement_moments} (which apply when $d=3\alpha$ by \cref{III-thm:critical_dim_hydro}), we compute the asymptotics of the \emph{$L^p$ radius of gyration} $\xi_{p}(\beta_c,r)$ (defined in \eqref{eq:radius_of_gyration_intro}) for even integers $p$ throughout the effectively high dimensional regime and the effectively long-range critical dimensional regime $d=3\alpha<6$, finding that
\begin{equation}
\label{eq:radius_of_gyration_intro}
\xi_{p}(\beta_c,r):= \left[\frac{\E_{\beta_c,r}\left[\sum_{x\in K} \|x\|_2^{p} \right]}{\E_{\beta_c,r}|K|}\right]^{1/p} \sim 
C_p\cdot\;\begin{cases}
 r^{\alpha/2} &\; \alpha>2, d>6\phantom{\alpha} \quad \text{(SR HD)}\\
 r \sqrt{\log r} &\; \alpha = 2, d\geq 6\phantom{\alpha} \quad \text{(mSR HD)}\\
r &\; \alpha < 2, d \geq 3\alpha \quad \text{(LR HD\&CD)}
\end{cases}
\end{equation}
as $r\to\infty$ for each even $p\geq 2$, where $C_p$ is a positive constant that may depend on the precise choice of kernel.
The quantity $\xi_{p}(\beta_c,r)$ measures the ``typical distance between two points in a typical large cluster'' (in an $L^p$ sense) under the measure $\P_{\beta_c,r}$.  
It can be thought of as one\footnote{See \cite[Corollary 2]{MR879034} for a proof that this notion coincides up-to-consants with other notions of correlation length for $2d$ nearest neighbour percolation in the slightly subcritical regime. Analogous results for high-dimensional finite-range percolation follow as immediate corollaries of \cite[Theorem 1.1 and 1.2]{hutchcroft2023high}.} precisely defined version of the \textbf{correlation length} $\xi(r)$ of the cut-off measure $\P_{\beta_c,r}$, where the correlation length is defined loosely as the ``scale at which the model begins to behave subcritically''. 
 We see from \eqref{eq:radius_of_gyration_intro} that (in the regimes treated in \cite{LRPpaper1,LRPpaper3}), the effectively long-range regime is characterised by the correlation length $\xi(r)$  associated to the cut-off measure $\P_{\beta_c,r}$ satisfying $\xi(r)=O(r)$ as $r\to \infty$. In other words, the model in the effectively long-range regime feels the effect of the cut-off \emph{at the same scale it is applied}, in contrast to the effectively short-range regime where the effect of the cut-off is not felt (in the sense that the model does not start to ``behave subcritically'') until a much larger scale.
We use this property to \emph{define} the effectively long-range regime, including for models that may be \emph{below} their upper critical dimension:


\begin{defn}[CL]
\label{def:CL}
 We say that the model satisfies the \textbf{correlation length condition for effectively long-range critical behaviour} \mytag[CL]{{\color{blue}CL}}\!\! if $\beta_c<\infty$ and for each $p>0$ there exists a constant $C_p<\infty$ such that $\xi_p(\beta_c,r) \leq C_p r$ for every $r\geq 1$.
\end{defn}

\begin{figure}
\centering
\includegraphics[scale=0.975]{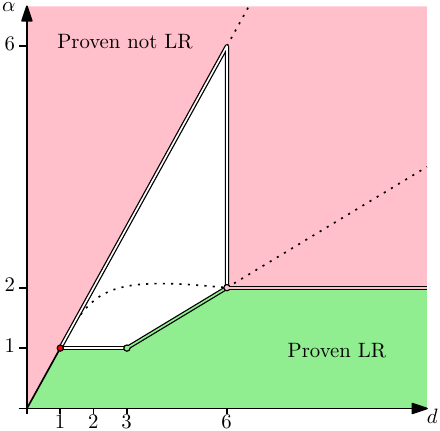}
\caption{Schematic illustration of the values of $\alpha$ and $d$ where the model is proven to satisfy \eqref{CL} (green), proven not to satisfy \eqref{CL} (pink), or neither (white). The coloured ribbons and vertices represent our knowledge on the boundary points of these regions, with the red vertex at $d=\alpha=1$ representing the discontinuous phase transition that is conjectured to occur only at this point \cite{MR868738,duminil2020long}. (The line $d=\alpha<1$ is left uncoloured since it has no real meaning in the model.)
 In the cases $d>6$, $\alpha \geq 2$ and $d=6$, $\alpha=2$ the model is also proven to be effectively short-range in the stronger sense that it has the same super-Brownian scaling limit as the nearest-neighbour model (\cref{I-thm:superprocess_main}). When $d>6$, effectively short-range behaviour is established between the two dotted lines $d=\alpha$ and $d=3\alpha$ only under the numerical assumptions needed to implement the lace expansion. The dotted curve in the white region represents the conjectured curve separating the LR and SR regions.}
\label{fig:CL}
\end{figure}

Noting that $\xi_p(\beta,r)$ is increasing in $p$ by Jensen's inequality, it follows from \cref{I-prop:radius_of_gyration,I-prop:radius_of_gyration_d=3alpha>6,I-prop:displacement_moments} as summarised in \eqref{eq:radius_of_gyration_intro} that \eqref{CL} holds for the models of high and critical effective dimension treated in the other two papers of the series \cite{LRPpaper1,LRPpaper3} if and only if $\alpha<2$. This leads in particular to the following corollary (which in the case $d=3\alpha$ relies on the validity of the hydrodynamic condition as proven in \cref{III-thm:critical_dim_hydro}).

\begin{corollary}
\label{cor:HD_CL}
If $d\geq 3\alpha$ then \eqref{CL} holds if and only if $\alpha<2$.
\end{corollary}

It also follows from \cref{I-prop:radius_of_gyration} that \eqref{CL} does \emph{not} hold for ``spread-out'' models with $d>6$ and $\alpha\geq 2$ whose two-point function can be analyzed using the lace expansion \cite{MR2430773,MR3306002,MR4032873,liu2025high}.  We also prove that \eqref{CL} does not hold when $\alpha >d$ (\cref{cor:large_alpha_not_CL}); see \cref{fig:CL} for a summary.

\medskip

Using the results of \cite{hutchcroft2024pointwise}, we prove that \eqref{CL} is also satisfied whenever $\alpha<1$:

\begin{thm}
\label{thm:alpha<1_CL}
If $\alpha<1$ then \eqref{CL} holds.
\end{thm}

This theorem applies in particular when $d=1$ and $\alpha \in (1/3,1)$ or when $d=2$ and $\alpha\in (2/3,1)$, in which case the model is \emph{below} its upper critical dimension.

\begin{remark}
In light of our work, it is an interesting problem to give equivalent characterisations of when \eqref{CL} holds, ideally leading to ``dichotomy theorems'' separating the effectively long-range and short-range regimes: we make the (imprecise) conjecture that \eqref{CL} coincides with all other natural definitions of the effectively long-range regime, and in particular that the model has the same critical exponents and scaling limit as the nearest-neighbour model whenever $d\geq 2$ and \eqref{CL} does not hold (with possible logarithmic corrections to scaling for $\alpha=\alpha_c(d)$).
With a view to eventually proving such dichotomy theorems, we prove several of our main results (including the volume tail estimate \cref{thm:main_low_dim}) under a much weaker assumption we call the \textbf{weak correlation length condition} \eqref{wCL} that is introduced in \cref{sec:the_maximum_cluster_size}. Let us also stress that we are not yet able to prove that the set of $\alpha$ for which \eqref{CL} holds is an interval for each $d$, or indeed that the validity of the condition \eqref{CL} is determined by the choice of $d$ and $\alpha$ alone (i.e., does not depend on the precise choice of kernel).
\end{remark}

\subsection{Critical exponents in the effectively long-range, low-dimensional regime.}
\label{subsec:main_results_low_dim}

We are now ready to state our results concerning critical behaviour in the effectively long-range low-dimensional regime.  (Note that the theorems in this section are valid even when $\alpha=d$, and show in this case that if \eqref{CL} holds then the model has a discontinuous phase transition.)

\medskip

\noindent \textbf{The two-point function.}
We begin with the following theorem, which establishes the validity of Sak's predicted value  \eqref{eq:Sak} for the two-point exponent $\eta=2-\alpha$ throughout the entire effectively long-range regime (as defined by \eqref{CL}).


\begin{theorem}[LR critical two-point functions]
\label{thm:CL_Sak}
If \eqref{CL} holds then
\begin{equation}
\label{eq:CL_Sak}
\P_{\beta_c}(x\leftrightarrow y)\asymp \|x-y\|^{-d+\alpha}
\end{equation}
for every distinct $x,y\in \Z^d$.  Moreover, \eqref{CL} holds if and only if there exists a decreasing function $h:(0,\infty)\to (0,1]$ decaying faster than any power of $x$ such that
\begin{equation}
  \P_{\beta_c,r}(x\leftrightarrow y)\preceq \|x-y\|^{-d+\alpha} \,h\Biggl(\frac{\|x-y\|}{r}\Biggr)
  \label{eq:CL_h}
\end{equation}
for every $r\geq 1$ and distinct $x,y\in \Z^d$.
\end{theorem}

\begin{remark}
We show in \cref{prop:alpha<1_sCL} that if $\alpha<1$ then \eqref{eq:CL_h} holds with $h$ of the form $h(x)=e^{-cx}$ for some positive constant $c$. The estimates \eqref{eq:CL_Sak} and \eqref{eq:CL_h} together capture a further sense in which the measure $\P_{\beta_c,r}$ has correlation length of order at most $r$ under \eqref{CL}.
\end{remark}

Although our focus here is on the effectively low-dimensional regime, \cref{thm:CL_Sak} also has important consequences in high and critical
 effective dimension. Indeed, the following corollary was only previously known when either $\alpha<1$ \cite{hutchcroft2024pointwise} or for ``spread-out'' kernels with $d>3\alpha$ which can be studied using the lace expansion \cite{MR2430773,MR3306002,MR4032873,liu2025high}. 

\begin{corollary}
\label{cor:high_dimensional_two_point_pointwise}
If $d\geq 3\alpha$ and $\alpha < 2$ then $\P_{\beta_c}(x\leftrightarrow y) \asymp \|x-y\|^{-d+\alpha}$ for every distinct $x,y\in \Z^d$.
\end{corollary}

(The $d=3\alpha<6$ case of this corollary, where the model is at its upper critical dimension, is once again contingent on \cref{III-thm:critical_dim_hydro}, not proven until the third paper of the series, which was used to deduce \cref{cor:HD_CL} from \cref{I-prop:radius_of_gyration,I-prop:radius_of_gyration_d=3alpha>6,I-prop:displacement_moments} in this case.)
See also \cref{III-thm:pointwise_three_point} for pointwise estimates on the \emph{three-point function} in the same regime building on this corollary.

\medskip

\noindent \textbf{The volume tail and fractal dimension.}
We next state our main theorem on the cluster volume distribution in the effectively long-range low-dimensional regime, which verifies the prediction \eqref{eq:Sak_delta} throughout this regime. The proof of this theorem establishes not only the claimed estimates on the volume distribution but establishes precise versions of the heuristic hyperscaling theory used to derive the prediction \eqref{eq:Sak_delta} from \eqref{eq:Sak} as  we explain in detail in \cref{subsec:hyperscaling}.

\begin{theorem}[LR LD critical cluster volumes]
\label{thm:main_low_dim}
If $d<3\alpha$ and 
 \eqref{CL} holds then 
\[
\P_{\beta_c}(|K|\geq n) \asymp 
n^{-(d-\alpha)/(d+\alpha)} 
\qquad \text{ and } \qquad \E_{\beta_c}|K \cap B_r|^p  \asymp_p 
r^{\alpha+(p-1)\frac{d+\alpha}{2}}
\]
for every $n,r\geq 1$ and integer $p\geq 1$.
 In particular, the exponents $\delta$ and $d_f$ are well-defined and given by $\delta=(d+\alpha)/(d-\alpha)$ and $d_f=(d+\alpha)/2$ respectively.
\end{theorem}



\noindent 
\textbf{Higher $k$-point functions.} We now state our results regarding the critical $k$-point function in the effectively long-range low-dimensional regime, generalizing \cref{thm:k_point_simple}. We will state our up-to-constants estimates in terms of the function $S$
 from finite subsets of $\R^d$ with at least two elements to $(0,\infty)$ defined recursively by
\begin{align}
S(\{x,y\})&=\|x-y\|_2^2;
\nonumber\\  S(A) &= \min \sqrt{S(A_1 \cup A_2)S(A_2 \cup A_3)S(A_3 \cup A_1)} \quad \text{ when $|A|\geq 3$}, \label{eq:S_def}
\end{align}
where the minimum is taken over partitions of $A$ into three non-empty subsets $A_1$, $A_2$, and $A_3$; a more transparently geometric function that is equivalent to this quantity up to constants is given in \cref{subsec:higher_gladkov_inequalities_and_the_k_point_function}.
We also write $S(x_1,\ldots,x_k)=S(\{x_1,\ldots,x_k\})$ whenever $x_1,\ldots,x_k\in \R^d$ are distinct. For three points there is only one way to partition into three non-empty subsets, so that
\[S(x,y,z) = \|x-y\|_2\|y-z\|_2\|z-x\|_2.\]
Our next main result yields an up-to-constants relationship between the critical $k$-point function and $S^{-(d-\alpha)/2}$ for models in the effectively long-range low-dimensional regime. 

\begin{theorem}[LR LD critical $k$-point functions]
\label{thm:k_point_S}
If $d<3\alpha$ and 
 \eqref{CL} holds 
 then the critical $k$-point function satisfies
\begin{equation}
\label{eq:k_point_S_intro}
\tau_{\beta_c}(x_1,\ldots,x_k) \asymp_k S(x_1,\ldots,x_k)^{-(d-\alpha)/2}
\end{equation}
for every  $k$-tuple of distinct points $x_1,\ldots,x_k \in \Z^d$.
\end{theorem}

As we explain in \cref{subsec:_k_point_function_hyperscaling}, this theorem is strongly evocative of the \emph{conformal invariance} of the model conjectured to hold when $J \sim C \|x-y\|_2^{-d-\alpha}$ is asyptotically rotationally invariant. We also explain in \cref{subsec:_k_point_function_hyperscaling} that \cref{thm:k_point_S} can be interpreted as a \emph{hyperscaling relation} for the $k$-point function, rather than a special feature of long-range percolation, with a similar relation between the two-point and $k$-point functions also holding for \emph{nearest-neighbour} percolation on planar lattices (\cref{thm:planar_k_point}).

\medskip

%
The proof of the \emph{upper bound} of \cref{thm:k_point_S} relies on a very interesting recent result of Gladkov \cite{gladkov2024percolation} which states that the three-point function for Bernoulli bond percolation on any weighted graph, at any $\beta\geq 0$, satisfies
\begin{equation}
\label{eq:Gladkov_intro}
  \tau(x,y,z) \leq \sqrt{8 \tau(x,y)\tau(y,z)\tau(z,x)}.
\end{equation}
The upper bound of \cref{thm:k_point_S} follows immediately from this inequality and \cref{thm:CL_Sak} in the case $k=3$, while the case $k\geq 4$ follows from a generalization of the Gladkov inequality \eqref{eq:Gladkov_intro} to higher $k$-point functions proven in \cref{subsec:higher_gladkov_inequalities_and_the_k_point_function}. (The \emph{lower bound} requires the full machinery of the paper, and is related to the hyperscaling postulates discussed in \cref{subsec:hyperscaling}.)

 \begin{remark}
While the Gladkov inequality \eqref{eq:Gladkov_intro} is always true, it is sharp for critical percolation only in the effectively low-dimensional regime, and can be thought of as a hyperscaling inequality as we explain in \cref{subsec:_k_point_function_hyperscaling}. Above the critical dimension, it is instead the Aizenman-Newman \cite{MR762034} tree-graph inequality 
\[\tau(x,y,z) \leq \sum_{w\in \Z^d} \tau(x,w)\tau(w,y)\tau(w,z)\]
that is sharp; this change in behaviour of the three-point function is yet another manifestation of the distinction between the high- and low-effective-dimensional regimes; see \cref{remark:HD_spread} for further comparison of $k$-point scaling in the two regimes.
In the critical dimension both inequalities are wasteful by a polylogarithmic factors as established in \cref{III-thm:pointwise_three_point} (see also \cref{III-remark:tree_graph_vs_Gladkov}).
\end{remark}



\subsection{Slightly subcritical scaling relations}
\label{subsec:subcritical_intro}

All the results we have stated so far concern percolation \emph{at the critical point}. We now state some related results concerning \emph{slightly subcritical} percolation. Like \cref{thm:CL_Sak}, the results we state in this section are valid 
for all effectively long-range models (which may be at, above, or below the upper critical dimension), but are arguably most interesting in the effectively low-dimensional regime. As before, we will state several corollaries concerning the effectively long-range critical-dimensional model ($d=3\alpha<6$) that rely on the results proven in the third paper in the series (specifically \cref{III-thm:critical_dim_hydro,III-thm:critical_dim_moments_main}).


The main exponents we will be interested in for the purposes of this section are the \emph{susceptibility exponent} $\gamma$, the \emph{correlation length exponent} $\nu$, and the \emph{gap exponent}~$\Delta$. The susceptibility exponent $\gamma$ is defined (if it exists) by
\[
  \chi(\beta):=\E_{\beta}|K| = |\beta-\beta_c|^{-\gamma \pm o(1)} \qquad \text{ as $\beta\uparrow \beta_c$}.
\]
Meanwhile, the correlation length exponent $\nu$ and gap exponent $\Delta$ are defined imprecisely by
\[
  \xi(\beta) = |\beta-\beta_c|^{-\nu \pm o(1)} \qquad \text{ and } \qquad \zeta(\beta) = |\beta-\beta_c|^{-\Delta \pm o(1)} \qquad \text{ as $\beta \uparrow \beta_c$, }
\]
 where $\xi(\beta)$ denotes the correlation length and $\zeta(\beta)$ denotes the ``typical size of a large cluster'' under the measure $\P_{\beta}$. 
 We give a more precise meaning to the exponent $\nu$ by defining $\nu$ to exist if and only if
\[
  \xi_p(\beta) = |\beta-\beta_c|^{-\nu \pm o(1)} \qquad \text{ as $\beta\uparrow \beta_c$}
\]
for all sufficiently small positive $p$, where $\xi_p(\beta)=\xi_p(\beta,\infty)$ denotes the $L^p$ radius of gyration defined in \eqref{eq:radius_of_gyration_intro}. (For long-range models it is too strong to require that this holds for all $p$; indeed one can easily check that $\xi_p(\beta)=\infty$ for every $p\geq \alpha$ and $\beta>0$.) For $\Delta$, it is a standard assumption of heuristic scaling theory that 
\begin{equation}
  \E_\beta |K|^p \asymp_p \zeta(\beta)^{p-1}\chi(\beta),
\label{eq:scaling_hypothesis_moments}
\end{equation}
and we will \emph{define} $\zeta(\beta)$ via the $p=2$ case of this purported relation, that is, define $\zeta(\beta):=\E_{\beta}|K|^2/\E_{\beta}|K|$. With this definition in hand, the gap exponent becomes precisely defined (if it exists) via $\zeta(\beta)=|\beta-\beta_c|^{-\Delta \pm o(1)}$ as $\beta\uparrow \beta_c$.
 It will be convenient to extend this definition to the measures $\P_{\beta,r}$ in the obvious way, so that $\zeta(\beta,r):=\E_{\beta,r}|K|^2/\E_{\beta,r}|K|$.


In contrast to the exponents $\eta$, $\delta$, and $d_f$ studied in \cref{cor:low_dim_moments}, which all have simple exact expressions throughout the effectively long-range regime, the exponents $\gamma$, $\nu$, and $\Delta$ are \emph{not} conjectured to have closed-form expressions in the effectively long-range low-dimensional (LR LD) regime. As such, rather than compute any of these exponents, we limit ourselves to proving relations between them. These relations are valid throughout the entire effectively long-range regime and also lead to new results in high and critical effective dimension (conditional on the results of the third paper of this series) as we discuss below. The Fisher relation $\gamma=(2-\eta) \nu$ was previously verified under the stronger condition that $\alpha<1$ \cite[Theorem 1.5 and Remark 1.6]{hutchcroft2024pointwise}.

\begin{thm}[LR scaling relations]
\label{thm:scaling_relations}
Suppose that \eqref{CL} holds. If at least one of the exponents $\gamma$, $\nu$, or $\Delta$ is well-defined then all three exponents are well-defined and related by the scaling relations
\[
  \gamma=(2-\eta)\nu \qquad \text{ and } \qquad \Delta = \nu d_f
\]
where $2-\eta=\alpha$ and $d_f = \min\{(d+\alpha)/2,2\alpha\}$.
\end{thm}

We refer the reader to \cite[Chapter 9]{grimmett2010percolation} for a detailed overview of heuristic scaling theory.
The same scaling relations (among various others) were verified for nearest-neighbour percolation in the plane by Kesten \cite{MR879034} (see \cite{duminil2021near} for an alternative approach and \cite{duminil2020planar,MR3940769} for extensions to the random cluster model and Voronoi percolation), and played an important role in the eventual computation of critical exponents for site percolation on the triangular lattice \cite{smirnov2001critical}. They have also been verified for various ``spread-out'' high-dimensional models to which the lace expansion applies; see e.g.\ \cite{HutchcroftTriangle} or \cite{heydenreich2015progress} for an overview. The scaling relations have also been verified in dimensions $d=3,4,5$ conditional on (as yet unproven) \emph{hyperscaling postulates} in \cite{MR1716769}.
Unconditionally and in general dimension we have only \emph{inequalities} including the \emph{mean-field lower bounds} $\gamma \geq 1$ \cite{MR762034}, $\delta \geq 2$ \cite{aizenman1987sharpness}, and $\Delta \geq 2$ \cite{durrett1985thermodynamic}, and the scaling inequalities $\Delta \leq \gamma +1 \leq \delta$ \cite{1901.10363}, and $2(\delta-1) \leq \gamma \delta \leq \Delta (\delta-1)$ \cite{newman1987inequalities,newman1986some}. 
See \cite[Chapter 10]{grimmett2010percolation} and the introduction of \cite{1901.10363} for further discussion.
These inequalities ensure that if any of the exponents $\delta,\gamma$, or $\Delta$ takes its mean-field value then the other exponents do also, but are not sharp outside of the mean-field case.
When $d\geq 3\alpha$, it follows from these inequalities together with the equality $\delta=2$ established in \cref{I-thm:hd_moments_main,III-thm:critical_dim_moments_main} that $\gamma=1$ and $\Delta=2$; \cref{thm:scaling_relations} yields that $\nu=1/\alpha$, which is a new result at this level of generality.


\cref{thm:scaling_relations} will be deduced from the more concrete statements \cref{thm:Fisher_relation,thm:subcritical_volume}, below, which do not require one to assume that any exponents are well-defined. Moreover, these theorems provide \emph{up-to-constants} estimates on all relevant quantities, rather than just exponent relations, and as such can be used to relate not only critical exponents but also the logarithmic corrections to scaling that occur at the critical dimension. (Note that we do not compute the logarithmic corrections to scaling for any of $\chi(\beta)$, $\zeta(\beta)$, or $\xi(\beta)$ at $d=3\alpha<6$ in this series of papers, but plan to do this in future work that will rely on the theorems stated here.)


The first of these results 
 establishes up-to-constants estimates on the two-point function in the slightly subcritical regime, 
  leading to a rigorous verification of the \emph{Fisher relation} $\gamma = (2-\eta)\nu$ \cite[Eq. 9.21]{grimmett2010percolation} assuming that at least one of $\gamma$ or $\nu$ is well-defined. Note that the subcritical two-point function is known to be given asymptotically by
\begin{equation}
\label{eq:two_point_subcritical_asymptotics}
  \P_{\beta}(x\leftrightarrow y) \sim \beta \chi(\beta)^2 J(x,y)
\end{equation}
when $\beta<\beta_c$ is fixed and $x-y\to \infty$. As explained in \cite{aoun2021sharp} (where this result is extended to the random cluster model), one can interpret
this estimate as stating that the most efficient way to connect between distant vertices $x$ and $y$ in the subcritical model is for a \emph{single long edge} to join two
“typical clusters” around $x$ and $y$, with each such cluster having expected size $\chi(\beta)$. The following theorem can be interpreted as stating that the two-point function transitions between critical $\|x-y\|^{-d+\alpha}$ and subcritical $\chi(\beta)^2\|x-y\|^{-d-\alpha}$ scaling in the simplest possible way as $\|x-y\|$ crosses from below to above the correlation length $\xi(\beta) \asymp \chi(\beta)^{1/\alpha}$. A similar theorem was established in \cite[Theorem 1.5]{hutchcroft2024pointwise} under the stronger assumption that $\alpha<1$.

\begin{theorem}[LR slightly subcritical two-point functions]
\label{thm:Fisher_relation}
Suppose that \eqref{CL} holds. If we define $\xi^*(\beta)= \chi(\beta)^{1/\alpha}$ for each $0<\beta < \beta_c$ then
\[
  \P_{\beta}(x \leftrightarrow y) \asymp \|x-y\|^{-d+\alpha} \left(1 \vee \frac{\|x-y\|}{\xi^*(\beta)}\right)^{-2\alpha}
\]
for every $\beta_c/2\leq \beta< \beta_c$ and distinct $x,y \in \Z^d$. As a consequence, $\xi_p(\beta) \asymp_p \xi^*(\beta)$ for all $0<p<\alpha$ and $\beta_c/2\leq \beta<\beta_c$.
\end{theorem}


Our next theorem establishes both the validity of the scaling relation $\Delta=\nu d_f$ and of the scaling hypothesis \eqref{eq:scaling_hypothesis_moments} throughout the effectively long-range regime. Recall that we defined $\zeta(\beta_c,r)=(\E_{\beta_c,r}|K|^2)/(\E_{\beta_c,r}|K|)$.

\begin{theorem}[LR slightly subcritical volume distributions]
\label{thm:subcritical_volume}
If \eqref{CL} holds and we define $\xi^*(\beta)=\chi(\beta)^{1/\alpha}$ and $\zeta^*(\beta) := \zeta(\beta_c,\xi^*(\beta))$
then
\[
\E_{\beta}|K|^p \asymp_p \zeta^*(\beta)^{p-1} \chi(\beta)\]
for every $\beta_c/2 \leq \beta <\beta_c$ and integer $p\geq 1$,  so that $\zeta(\beta) \asymp \zeta^*(\beta)$ for every $\beta_c/2\leq \beta<\beta_c$. Moreover, there exists a positive constant $c$ and a decreasing function $h:(0,\infty)\to (0,1]$ decaying faster than any power such that
\[
\mathbbm{1}(n\leq c\, \zeta^*(\beta)) \preceq
\frac{\P_{\beta}(|K|\geq n)}{\P_{\beta_c}(|K|\geq n)}
  \preceq h\!\left(\frac{n}{\zeta^*(\beta)}\right)
\]
for every $\beta_c/2\leq \beta <\beta_c$ and $n\geq 1$.
\end{theorem}

\begin{remark}It follows from \cref{I-thm:hd_moments_main} (in the case $d>3\alpha$), \cref{III-thm:critical_dim_moments_main} (in the case $d=3\alpha$), and \cref{lem:sCL_moments} (in the case $d<3\alpha$) that if \eqref{CL} holds then
\[
  \zeta(\beta_c,r) := \frac{\E_{\beta_c,r}|K|^2}{\E_{\beta_c,r}|K|} \asymp \begin{cases} r^{2\alpha} &d>3\alpha \qquad \text{(LR HD)}\\ 
   r^{2\alpha} (\log r)^{-1/2} & d=3\alpha \qquad\text{(LR CD)}\\
  r^{(d+\alpha)/2} & d<3\alpha \qquad \text{(LR LD)}.
  \end{cases}
\]
It follows in particular that $\zeta(\beta_c,r) = r^{d_f-o(1)}$ with $d_f=\min\{2\alpha,(d+\alpha)/2\}$, so that \cref{thm:subcritical_volume} does indeed imply that the relation $\Delta=\nu d_f$ holds whenever at least one of $\Delta$ or $\nu$ is well-defined.
\end{remark}





\subsection{A glimpse of conformal invariance.}
\label{subsec:glimpse_of_conformal_invariance_}

In this section we explain how the $k$-point function estimate of \cref{thm:k_point_S} relates to the conjectured conformal invariance of the model, beginning with a general discussion of scale invariance and conformal invariance in $\R^d$; we refer the reader to e.g.\ \cite{MR3874867} for further background and e.g.\ \cite{lawler2008conformally,duminil2012conformal} for an overview of conformal invariance in the context of two-dimensional, short-range models. The standard approach to taking scaling limits in the physics literature is via the rescaled $k$-point functions
\begin{equation}
\label{eq:continuum_k_point}
T(x_1,\ldots,x_k) := \lim_{\lambda\to \infty} \lambda^{\Delta k} \tau_{\beta_c}( [\lambda x_1],\ldots,[\lambda x_k]),
\end{equation}
where for each $x\in \R^d$ we write $[x]$ for the closest lattice point to $x$ (breaking ties arbitrarily). Here $\Delta$ is a \emph{scaling dimension} that should be taken to be $(d-2+\eta)/2$ 
 to get a non-trivial limit in the case $k=2$. (At the critical dimension one may need to replace $\lambda^{\Delta k}$ with $(\ell(\lambda) \lambda^{\Delta})^k$ for some slowly-varying correction $\ell$ to get a non-trivial limit; this does not significantly change the rest of the discussion.)
If this continuum $k$-point function is well-defined for distinct $x_1,\ldots,x_n \in \R^d$ 
 then it must satisfy the scale-covariance relation
\begin{equation}
\label{eq:scale_invariance_k_point}
  T(\lambda x_1,\ldots,\lambda x_n) = \lambda^{-\Delta n} T(x_1,\ldots,x_n)
\end{equation}
for every distinct $x_1,\ldots,x_n \in \R^d$.

If the model is defined with respect to the isometry-invariant kernel of the form $J(x,y)\propto \|x-y\|_2^{-d-\alpha}$ then the continuum limit is also expected to be invariant under Euclidean isometries (e.g.\ rotations).
It is a standard prediction of the physics literature (although still not considered a completely settled problem by physicists \cite{nakayama2015scale}) that ``most'' critical models exhibiting scale and rotation invariance are also \emph{conformally invariant}, i.e., invariant under maps that behave like compositions of (orientation-preserving) Euclidean isometries and dilations \emph{infinitesimally around each point}.
By Liouville's theorem, the only conformal transformations of $\R^d$ with $d\geq 2$ are the M\"obius transformations $\Mobius(\R^d)$, which can be defined\footnote{Strictly speaking this definition includes both the M\"obius and \emph{anti-M\"obius} transformations (such as reflections through hyperplanes), which are orientation-reversing and anti-conformal. This is not a problem since we also expect our scaling limits to be reflection-invariant.} as the group of transformations of $\R^d$ generated by Euclidean isometries, dilations, and the inversion in the unit sphere $x\mapsto x/\|x\|_2^2$; for $d\geq 3$ these are also the only \emph{local} conformal transformations of $\R^d$, in the sense that any conformal map between open connected subsets of $\R^d$ is the restriction of a M\"obius transformation. (In two dimensions there are many local conformal maps that are not M\"obius transformations as provided by e.g.\ the Riemann mapping theorem.)
If it is well-defined, the global continuum $k$-point function $T$ is said to be \textbf{conformally covariant} with \textbf{scaling dimension} $\Delta$ if
\begin{equation}
\label{eq:conformal_transformation_law}
T(\psi(x_1),\ldots,\psi(x_k)) =  T(x_1,\ldots,x_k)\prod_{i=1}^k |\!\det D\psi(x_i)|^{-\Delta/d}
\end{equation}
for every $\psi \in \mathrm{M \ddot{o}}(\R^d)$ and distinct $x_1,\ldots,x_k \in \R^d \setminus \psi^{-1}(\infty)$. Here the $d$th root of the determinant Jacobian $|\!\det D\psi(x_i)|^{1/d}$ should be thought of as a local scaling factor, with the map $\psi(x)=\lambda x$ having $|\!\det D\psi(x)|^{1/d}\equiv \lambda$, so that  \eqref{eq:conformal_transformation_law} is indeed a generalization of \eqref{eq:scale_invariance_k_point}.
 The conformal covariance property \eqref{eq:conformal_transformation_law} was proven for the scaling limit of the $k$-point function for site percolation on the triangular lattice by Camia \cite{camia2024conformal} (see also \cite{camia2025partner,camia2024logarithmic,camia2024conformallyOPE}).

\medskip

Although conformal invariance is a much weaker property in dimensions $d\geq 3$ than in two dimensions, it has long been known in the physics literature that conformal invariance still places surprisingly strong constraints on critical behaviour in these dimensions \cite{rychkov2017epfl}; see also \cite{paulos2016conformal} for a discussion in the setting of long-range models.
One simple example of such a constraint is that the continuum two- and three-point functions should satisfy
\begin{equation}
\label{eq:two_and_three_point_Mobius}
  T(x,y) = C_2 \|x-y\|_2^{-2\Delta} \quad \text{ and } \quad T(x,y,z) = C_3 \|x-y\|_2^{-\Delta} \|y-z\|_2^{-\Delta} \|z-x\|_2^{-\Delta}
\end{equation}
for some constants $C_2$ and $C_3$, this scaling behaviour being forced by M\"obius covariance since the M\"obius transformations act transitively on triples of distinct points in $\R^d$. See \cite{ang2021integrability,camia2024conformal} for the strongest rigorous results in this direction for 2d nearest-neighbour percolation. (The behaviour of higher $k$-point functions is less rigid since M\"obius transformations do not act transitively on sets of size $k\geq 4$.)
Conformal invariance considerations have led to the current best theoretical understanding of the 3d Ising critical point via the (higher dimensional) \emph{conformal bootstrap} \cite{simmons2017conformal,poland2019conformal}, and it remains an important challenge for mathematicians to transport these advances to the rigorous study of lattice models beyond two dimensions; see \cite{MR3874867} for a detailed survey.

\medskip

For long-range percolation models with $d_\mathrm{eff} > d_c=6$, these conformal invariance relations are relatively uninteresting since one expects that the higher $n$-point functions vanish in the continuum limit \eqref{eq:continuum_k_point}: If $d_\mathrm{eff}>d_c=6$ and $\Delta=(d-2+\eta)/2$ where $\eta=\max\{2-\alpha,0\}$ is mean-field value of the two-point function exponent then it follows easily from the tree-graph inequalities \cite{MR762034} that
\begin{equation}
  T(x_1,\ldots,x_n) = \lim_{\lambda\to \infty} \lambda^{\Delta n} \tau_{\beta_c}( [\lambda x_1],\ldots,[\lambda x_n]) \equiv 0 \quad \text{ for $n\geq 3$ and $x_1,\ldots,x_n$ distinct.}
  \label{eq:vacuous_conformal_invariance}
\end{equation}
(In particular, the constant $C_3$ from \eqref{eq:two_and_three_point_Mobius} is zero in this case! This is also true but harder to justify in the critical dimensional case $d=3\alpha<6$, see \cref{III-thm:pointwise_three_point}.) 
 As such, in this setting the only non-trivial conformal invariance relation is the power-law scaling of the two-point function $T(x,y)=A \|x-y\|^{-2\Delta}_2$  that is already forced by scale and rotation invariance alone.
 Conformal invariance, if true, should carry more interesting content in low effective dimensions, where \cref{thm:k_point_S} implies that the limiting higher $k$-point functions are non-trivial (if they are well-defined).

\medskip

While we are not yet able to verify scale, rotation, or conformal invariance of long-range percolation models of low effective dimension, 
our results do show that conformal invariance holds in an ``up-to-constants'' sense.
Indeed, in \cref{subsec:conformal_invariance} we prove that the quantity $S(A)$ defined in \eqref{eq:S_def} is conformally covariant of scaling dimension $-1$ in the sense that if $\psi\in \Mobius(\R^d)$ is a M\"obius transformation then
\begin{equation}
\label{eq:S_Mobius_invariance}
  S(\psi(A))=  S(A) \prod_{a\in A} |\!\det D\psi(a)|^{1/d}
\end{equation}
for every finite set $A\subseteq \R^d \setminus \psi^{-1}(\infty)$. 
As such, it follows from \cref{thm:k_point_S} that the critical $k$-point function satisfies a similar conformal covariance property with scaling dimension $(d-\alpha)/2$ ``up-to-constants''  in the effectively long-range low-dimensional regime:

\begin{corollary}
\label{cor:Mobius}
If $d<3\alpha$ and \eqref{CL} holds then
 \begin{equation}
\label{eq:S_Mobius_invariance}
  \tau_{\beta_c}([\psi(x_1)], \ldots, [\psi(x_k)]) \asymp_k  \tau_{\beta_c}(x_1,\ldots,x_k)\prod_{i=1}^k |\!\det D\psi(x_i)|^{-\frac{d-\alpha}{2d}}
\end{equation}
for every $\psi\in \Mobius(\R^d)$ and 
 $x_1,\ldots,x_k \in \Z^d \setminus \psi^{-1}(\infty)$ with $\|\psi(x_i)-\psi(x_j)\|_\infty > 1$ for every $i\neq j$.
\end{corollary}

\begin{remark}
 The up-to-constants M\"obius-covariance estimate \eqref{eq:S_Mobius_invariance} holds even for kernels satisfying e.g.\ $J(x,y) \propto \|x-y\|_{1}^{-d-\alpha}$ where $d=2$, $\alpha\in (2/3,1)$, and $\|\cdot\|_1$ is the $\ell^1$ norm, which are \emph{not} expected to yield rotationally-invariant scaling limits. As such, one should not be able to improve \eqref{eq:S_Mobius_invariance} to a first-order asymptotic equality without imposing additional assumptions on the model. (Effectively \emph{short-range} models should have rotationally invariant scaling limits assuming only that the kernel is invariant under, say, coordinate permutations: see \cref{I-thm:superprocess_main,I-remark:renormalized_covariance_matrix_symmetry} for related results in high effective dimension.)
\end{remark}

\begin{remark}
In the 2d critical percolation literature one typically thinks of conformal invariance in terms of connection probabilities for macroscopic regions (such as quad-crossing probabilities) rather than $k$-point functions \cite{schramm2011scaling,smirnov2001critical2}. This perspective is not suitable for long-range models. Indeed, if we consider long-range percolation on the rescaled lattice  $\eps \Z^d$ then any two open sets in $\R^d$ are connected by a single edge with probability converging to $1$ as $\eps \downarrow 0$ whenever $\alpha<d$.
\end{remark}


\section{The correlation length condition and its consequences}
\label{subsec:long_edge_relevancy_for_small_alpha}

In this section we prove \cref{thm:alpha<1_CL}, which states that \eqref{CL} holds for $\alpha<1$, along with the following lemma which establishes everything in \cref{thm:CL_Sak} other than the matching \emph{lower bound} for $\P_{\beta_c}(x\leftrightarrow y)$ (which we will return to in \cref{subsec:_k_point_function_hyperscaling}).

\begin{lemma}
\label{lem:CL_Sak_upper}
If \eqref{CL} holds then
$\P_{\beta_c}(x\leftrightarrow y)\preceq \|x-y\|^{-d+\alpha}$
for every distinct $x,y\in \Z^d$.  Moreover, \eqref{CL} holds if and only if there exists a decreasing function $h:(0,\infty)\to (0,1]$ decaying faster than any power of $x$ such that
\begin{equation}
  \P_{\beta_c,r}(x\leftrightarrow y)\preceq \|x-y\|^{-d+\alpha} \,h\Biggl(\frac{\|x-y\|}{r}\Biggr)
\end{equation}
for every $r\geq 1$ and distinct $x,y\in \Z^d$.
\end{lemma}

All of the results in this section apply to the entire effectively-long range regime and are not specific to low-dimensional models.

\subsection{The correlation length condition and the two-point function}
\label{subsec:the_correlation_length_condition_and_the_two_point_function}

In this subsection we prove \cref{lem:CL_Sak_upper}. For some of our intermediate results we will not need the full strength of the assumption \eqref{CL}; the minimal assumptions needed will be clarified through the following definition.

\begin{defn}
\label{def:CLp}
 Given $p>0$, we say that the model satisfies \mytag[CLp]{{\color{blue}CL$^p$}}\!\! if $\beta_c<\infty$ and there exists a constant $C_p<\infty$ such that $\xi_p(\beta_c,r) \leq C_p r$ for every $r\geq 1$.
\end{defn}

We begin by studying the asymptotics of the cut-off first moment $\E_{\beta_c,r}|K|$. We recall from \cref{I-cor:mean_lower_bound} that this quantity satisfies
\begin{equation}
\label{eq:mean_lower_bound}
\E_{\beta_c,r}|K|\geq (1-o(1)) \frac{\alpha}{\beta_c} r^\alpha.
\end{equation}
as $r\to \infty$ regardless of the value of $\alpha>0$. This bound is proven in \cref{I-prop:first_moment} to be sharp for effectively long-range models of high and critical effective dimension. The following lemma shows that it is of the correct \emph{order} assuming only that the model is effectively long-range in a fairly weak sense (see also \cref{lem:first_moment_CL}). 

\begin{lemma}
\label{lem:first_moment_CLb}
If \eqref{CLp} holds for some $p>0$ then 
$\E_{\beta_c,r}|K|\asymp r^\alpha$
for all $r\geq 1$. 
\end{lemma}

\begin{proof}[Proof of \cref{lem:first_moment_CLb}]
The lower bound on $\E_{\beta_c,r}|K|$ follows from \eqref{eq:mean_lower_bound}.
For the coarse upper bound $\E_{\beta_c,r}|K| \preceq r^\alpha$, we can compute that if \eqref{CLp} holds for some $p>0$ then
\begin{equation}
\label{eq:CLp_application}
 \E_{\beta_c,r}|K \setminus B_{\lambda r}| \leq (\lambda r)^{-p} \E_{\beta_c,r}\sum_{x\in K}\|x\|^p \preceq_p \lambda^{-p}\E_{\beta_c,r}|K|
\end{equation}
for every $r,\lambda \geq 1$. It follows that if \eqref{CLp} holds for some $p>0$ then there exists a constant $C$ such that
\begin{equation}
\label{eq:CLp_application2}
  \E_{\beta_c,r}|K| \leq 2 \E_{\beta_c,r}|K \cap B_{C r}| \preceq r^\alpha
\end{equation}
as claimed, 
where the second inequality follows from the bound \eqref{eq:Sak_upper_restate} (i.e., the main theorem of \cite{hutchcroft2022sharp}). (Note that this bound together with \eqref{eq:mean_lower_bound} imply that $\E_{\beta_c,r}|K \cap B_{C r}| \asymp r^\alpha$ so that the lower bound of \eqref{eq:CL_Sak} holds in a spatially averaged sense when \eqref{CLp} holds for some $p>0$.)
\end{proof}

\cref{lem:first_moment_CLb} and its proof have the following corollary.

\begin{corollary}
\label{cor:large_alpha_not_CL}
If $\alpha>d$ then \eqref{CLp} does not hold for any $p>0$. In particular, if $\alpha>d$ then \eqref{CL} does not hold. If $\alpha=d$ and \eqref{CLp} holds for some $p>0$ then the model has a discontinuous phase transition in the sense that there exists an infinite cluster at $\beta_c$.
\end{corollary}

\begin{proof}[Proof of \cref{cor:large_alpha_not_CL}]
If \eqref{CLp} holds for some $p>0$ then we have by \cref{lem:first_moment_CLb} and \eqref{eq:CLp_application2} that there exists a constant $C$ such that
\[
  \E_{\beta_c,r}|K \cap B_{Cr}| \succeq \E_{\beta_c,r}|K| \succeq r^\alpha 
\]
for every $r\geq 1$. Since the left hand side is trivially bounded by $|B_{Cr}| =O(r^d)$, this is not possible when $\alpha>d$. When $\alpha=d$ this is only possible if $\limsup_{x\to\infty}\P_{\beta_c}(0\leftrightarrow x)>0$, which implies that an infinite cluster exists almost surely under $\P_{\beta_c}$ by Fatou's theorem and Kolmogorov's zero-one law.
\end{proof}

We now apply \cref{lem:first_moment_CLb} to bound the cut-off two-point function.

\begin{prop}
\label{prop:CL_to_two_point}
If \eqref{CLp} holds for some $p>\max\{d-\alpha,0\}$ then
\[
  \P_{\beta_c,r}(x\leftrightarrow y) \preceq_p \|x-y\|^{-d+\alpha} \left(1 \vee \frac{\|x-y\|}{r}\right)^{-p+d-\alpha}
\] 
for every $r\geq 1$ and distinct $x,y\in \Z^d$.
\end{prop}

\begin{proof}[Proof of \cref{prop:CL_to_two_point}]
We may assume by \cref{cor:large_alpha_not_CL} that $\alpha\leq d$.
We have by Russo's formula that
\[
  \frac{d}{dr}\P_{\beta_c,r}(x\leftrightarrow y) = \beta_c |J'(r)| \sum_{a,b\in \Z^d}\mathbbm{1}(\|a-b\|\leq r)\P_{\beta_c,r}\left( x\leftrightarrow a \nleftrightarrow b \leftrightarrow y\right)
\]
for every $x,y\in \Z^d$ and $r>0$ and hence by the BK inequality that
\begin{equation}
\label{eq:r_derivative_BK}
  \frac{d}{dr}\P_{\beta_c,r}(x\leftrightarrow y) \leq \beta_c |J'(r)| \sum_{a,b\in \Z^d}\mathbbm{1}(\|a-b\|\leq r)\P_{\beta_c,r}( x\leftrightarrow a)\P_{\beta_c,r}(b \leftrightarrow y).
\end{equation}
Using that $|J'(r)| \sim r^{-d-\alpha-1}$ and ignoring (for now) the constraint that $\|a-b\|\leq r$, it follows in particular that
\[
  \frac{d}{dr}\P_{\beta_c,r}(x\leftrightarrow y)  \preceq r^{-d-\alpha-1} \sum_{a,b\in \Z^d}\P_{\beta_c,r}( x\leftrightarrow a)\P_{\beta_c,r}(b \leftrightarrow y) \preceq r^{-d+\alpha-1},
\]
where we used \cref{lem:first_moment_CLb} in the second inequality. Integrating this inequality between $\|x-y\|/8$ and $\infty$ yields that
\[
  \P_{\beta_c,r}(x\leftrightarrow y) \preceq \P_{\beta_c,\|x-y\|/8}(x\leftrightarrow y) + \|x-y\|^{-d+\alpha}
\]
for every $r\geq \|x-y\|/8$. As such, it suffices to prove that
\begin{equation}
\label{eq:CL_Sak_small_r}
  \P_{\beta_c,r}(x\leftrightarrow y) \preceq_p \|x-y\|^{-d+\alpha} \left(\frac{\|x-y\|}{r}\right)^{-p+d-\alpha}
\end{equation}
for every $r\leq \|x-y\|/8$ and every $p>d-\alpha$ for which \eqref{CLp} holds. Note that if $r\leq \|x-y\|/8$ then for each pair $a,b \in \Z^d$ with $\|a-b\|\leq r$ we must either have that $\|x-a\|\geq \|x-y\|/4$, $\|y-b\|\geq \|x-y\|/4$, or both. Thus, we have by \eqref{eq:r_derivative_BK} and \eqref{eq:CLp_application} that if \eqref{CLp} holds and
 $r\leq \|x-y\|/8$ then
\[
  \frac{d}{dr} \P_{\beta_c,r}(x\leftrightarrow y) \preceq r^{-d-\alpha-1} \E_{\beta_c,r}|K|\E_{\beta_c,r}|K \setminus B_{\|x-y\|/4}| \preceq r^{-d+\alpha-1} \left(\frac{\|x-y\|}{r}\right)^{-p}.
\]
If $p>d-\alpha$ then the power of $r$ appearing here is strictly larger than $-1$ and we can integrate this differential inequality to obtain that if \eqref{CLp} holds and $r\leq \|x-y\|/8$ then
\begin{multline*}
   \P_{\beta_c,r}(x\leftrightarrow y) \preceq_p
r^{-d+\alpha+p} \|x-y\|^{-p} + \P_{\beta_c,1}(x\leftrightarrow y) 
\\
= \|x-y\|^{-d+\alpha} \left(\frac{\|x-y\|}{r}\right)^{-p+d-\alpha} + \P_{\beta_c,1}(x\leftrightarrow y).
\end{multline*}
The claim \eqref{eq:CL_Sak_small_r} follows since $\P_{\beta_c,1}(x\leftrightarrow y)$ decays exponentially in $\|x-y\|$ by the sharpness of the phase transition \cite{aizenman1987sharpness,duminil2015new,MR852458}.
\end{proof}

We are now ready to complete the proof of \cref{thm:CL_Sak}. The proof will apply the following elementary lemma which we make a note of since it will be used several times throughout the paper.

\begin{lemma}
\label{lem:scaling_sums}
Let $h:(0,\infty)\to (0,1]$ be a decreasing function decaying faster than any power in the sense that $h(x)=o(x^{-p})$ for every $p>0$. Then for each $a>-1$ there exists a constant $C_a=C_a(h)$ such that
\[
  \sum_{n=1}^\infty n^a h(n/r) \leq C_a r^{a+1}
\]
for every $r\geq 1$.
\end{lemma}

\begin{proof}[Proof of \cref{lem:scaling_sums}]
Since $h$ is decreasing, bounded, and decays faster than any power, it satisfies a bound of the form $h(x) \leq C_p(1 \vee x)^{-p}$ for every $p>0$. Taking $p=a+2>a+1$, we deduce that
\[
  \sum_{n=1}^\infty n^a h(n/r) \preceq_a \sum_{n=1}^{\lfloor r\rfloor} n^a + \sum_{n=\lfloor r\rfloor +1}^\infty n^{a-p} r^p \preceq_{a} r^{a+1} 
\]
for each $a>-1$ and every $r\geq 1$ as claimed.
\end{proof}

\begin{proof}[Proof of \cref{lem:CL_Sak_upper}]
The claimed consequences of \eqref{CL} regarding upper bounds on the two-point functions $\P_{\beta_c}(x\leftrightarrow y)$ and $\P_{\beta_c,r}(x\leftrightarrow y)$ follow immediately from \cref{prop:CL_to_two_point}. 
(Indeed, the claimed inequality with the rapidly decaying function $h$ is equivalent to satisfying an inequality of the form $\P_{\beta_c,r}(x\leftrightarrow y) \leq C_p \|x-y\|^{-d+\alpha} (1\vee \|x-y\|/r)^{-p}$ for every $p>0$.)
It remains to prove conversely that \eqref{eq:CL_h} implies \eqref{CL}.  To prove this, we can compute that
\[
  \E_{\beta_c,r}\sum_{x\in K} \|x\|_2^p \preceq \sum_{x\in \Z^d} \|x\|^{-d+\alpha+p} h(\|x\|/r) \preceq \sum_{n=1}^\infty n^{\alpha+p-1} h(n/r) \preceq_p r^{\alpha+p},
\]
where the final inequality follows from \cref{lem:scaling_sums}. Comparing this with the lower bound \eqref{eq:mean_lower_bound} yields that $\xi_p(r)\preceq_p r$ for every $p>0$ and $r\geq 1$ as required for \eqref{CL}.
\end{proof}

\subsection{The case $\alpha<1$}

In this section we prove \cref{thm:alpha<1_CL}, which states that \eqref{CL} holds when $\alpha<1$. In light of \cref{lem:CL_Sak_upper} it suffices to prove the following proposition.


\begin{prop}
\label{prop:alpha<1_sCL}
If $\alpha<1$ then there exists a positive constant $c$ such that
\[
  \P_{\beta_c,r}(x\leftrightarrow y) \preceq \|x-y\|^{-d+\alpha}\exp\left[-c\frac{\|x-y\|}{r}\right]
\]
for every $r\geq 1$ and distinct $x,y\in \Z^d$.
\end{prop}


For each integer $n\geq 1$, let $\Lambda_n = [-n,n]^d \cap \Z^d$ be the box of side-length $2n+1$ in $\Z^d$. Following \cite{duminil2015new}, for each set $W \subseteq \Z^d$ containing the origin and each $r,\beta \geq 0$ we define
\[
  \phi_{\beta,r}(W) = \sum_{x\in W, y\notin W} \P_{\beta,r}(0 \xleftrightarrow{W} x) \beta J_r(x,y),
\]
where $\{0 \xleftrightarrow{W} x\}$ denotes the event that there is an open path from $0$ to $x$ using only vertices of $W$.
The correlation length is known in several contexts to coincide with the scale $n$ for which $\phi(\Lambda_n)$ is a small constant \cite{hutchcroft2024pointwise,hutchcroft2023high}. The following lemma shows that this  correlation length is $O(r)$ for the measure $\P_{\beta_c,r}$ when $\alpha<1$; its proof is very similar to similar estimates appearing for the measures $\P_{\beta}$ with $\beta<\beta_c$ in \cite{hutchcroft2024pointwise} and is related to the ``random box'' method of \cite{baumler2022isoperimetric}.

\begin{lemma}
\label{lem:alpha<1_phi}
If $\alpha<1$ then there exists a constant $C$ such that
\[
 \frac{1}{\lceil \lambda r \rceil}\sum_{n=1}^{\lceil \lambda r \rceil} \phi_{\beta_c,r}(\Lambda_{n}) 
 \leq C \lambda^{-1+\alpha}
\]
for every $\lambda,r\geq 1$.
\end{lemma}

\begin{proof}[Proof of \cref{lem:alpha<1_phi}]
For $m\geq n \geq 1$ and each vertex $x \in \Lambda_n$, the total weight of edges with one endpoint at $x$ and the other in $\Lambda_m^c$ is of order at most $(m-n)^{-\alpha}$. Since $\alpha<1$ we can average over the sidelength of the box and obtain the bound
\begin{align*}
 \frac{1}{N}\sum_{n=1}^{N} \phi_{\beta_c,r}(\Lambda_{n}) \preceq \frac{1}{N} \sum_{x \in \Lambda_N} \P_{\beta_c,r}(x\in K) \sum_{k = 1}^{r} k^{-\alpha} 
 \preceq \frac{1}{N}r^{1-\alpha}\E_{\beta_c,r}|K\cap \Lambda_N| \preceq \left(\frac{r}{N}\right)^{1-\alpha},
\end{align*}
where we used that edges of length longer than $r$ do not contribute to $\phi_{\beta_c,r}(\Lambda_n)$ to truncate the sum at $r$ in the second term and used \cref{I-thm:two_point_spatial_average_upper} in the last inequality.
\end{proof}


\begin{proof}[Proof of \cref{prop:alpha<1_sCL}]
Let $C_1\geq 1$ be a constant such that $B_r$ is contained in $\Lambda_{C_1r}$.
By \cref{lem:alpha<1_phi}, there exists a constant $C_2 \geq 1$ such that
\[
 \frac{1}{\lceil C_2r \rceil}\sum_{n=1}^{\lceil C_2 r \rceil} \phi_{\beta_c,r}(\Lambda_{n})  \leq \frac{1}{2}
\]
for every $r\geq 1$.
Thus, for each $r\geq 1$ there exists an integer $R(r)$ with $R(r)\leq 2C_2r$ such that
\[\phi_{\beta_c,r}(\Lambda_{R(r)}) \leq \frac{1}{2}.\]
%
%
%
%
%
%
Let $R'(r)=R(r)+C_1 r$.
 If $0$ is connected to some point $x$ with $\|x\|_\infty\geq 8C_1 R(r)$, then there exists
  a sequence of points $0=x_0,x_1,\ldots,x_{2n}$ 
with 
\[
  n= \lfloor \|x\|_\infty / 2R'(r)\rfloor
\]
  such that $x_{2i+2}$ does not belong to $\Lambda_{R(r)}(x_{2i})$ for each $0\leq i \leq n-1$, $x_{2n}$ has $\ell_\infty$-distance at least $\|x\|_\infty/2$ from $\|x\|_\infty$, and the following events all occur disjointly:
\begin{enumerate}
  \item For each $0\leq i \leq n-1$, $x_{2i}$ is connected to $x_{2i+1}$ within the box $\Lambda_{R(r)}(x_{2i})$ and the edge $\{x_{2i+1},x_{2i+2}\}$ is open for each $1\leq i \leq n$.
  \item $x_{2n}$ is connected to $x$.
\end{enumerate}
Indeed, such a collection of points can be found by taking a simple open path from $0$ to $x$ and iteratively taking $\{x_{2i+1},x_{2i+2}\}$ to be the first edge taken by the path after $x_{2i}$ that leaves the box $\Lambda_{R(r)}(x_{2i})$, stopping after $n$ iterations; this final point $x_{2n}$ has 
\[\|x_{2n}\|_\infty \leq n(R(r)+C_1r) = n R'(r) \leq \frac{1}{2} \|x\|_\infty \]
so that $\|x_{2n}-x\|_\infty \geq \|x\|_\infty /2$ as required.
Applying a union bound and the BK inequality yields that there exists a positive constant $c$ such that
\[
 \P_{\beta_c,r}(0\leftrightarrow x) \preceq 2^{-\|x\|_\infty/2R'(r)} \|x\|^{-d+\alpha}
 \preceq \exp\left[-c\frac{\|x\|}{r} \right] \|x\|^{-d+\alpha},
\]
where we used that $\|x_n-x\|\succeq \|x\|$, that $R'(r)\preceq r$, and the pointwise upper bound of \eqref{eq:two_point_pointwise} (i.e.\ the main result of \cite{hutchcroft2024pointwise}) to bound the last connection probability.
\end{proof}

\section{Higher Gladkov inequalities and their interpretations}
\label{subsec:higher_gladkov_inequalities_and_the_k_point_function}

In this section we prove our results concerning \emph{upper bounds} on the $k$-point function in the effectively long-range low-dimensional regime. (We return to prove matching \emph{lower bounds} on the $k$-point function in \cref{subsec:_k_point_function_hyperscaling}.) 
%
%
%
%
We begin by proving the following $k$-point generalization of the Gladkov inequality \cite{gladkov2024percolation}, which applies to Bernoulli bond percolation on arbitrary weighted graphs.

\begin{thm}[Higher-order Gladkov inequality]
\label{thm:higher_Gladkov}
 Let $G=(V,E,J)$ be a weighted graph, let $\beta\geq 0$, and for each set of vertices $A$ let $\tau(A)=\tau_\beta^G(A)$ denote the probability that the set of vertices $A$ all belong to the same cluster of Bernoulli bond percolation on $G$. Then the inequality 
\[
  \tau(A)^2 \leq 4(|A|-1)^2 (2^{|A|-2}-1) \max  \tau(A_1\cup A_2)\tau(A_2\cup A_3)\tau(A_3\cup A_1)
\]
holds for every set $A$ with $|A|\geq 3$, where the maximum is taken over partitions of $A$ into three non-empty subsets $A_1$, $A_2$, and $A_3$.
\end{thm}

\begin{remark}
\label{rmk:improved_higher_Gladkov} The proof actually establishes a stronger but less aesthetically pleasing inequality in which we restrict $A_3=\{x\}$ to be a singleton where $x\in A \setminus \{x_0\}$ for some $x_0$ that we are free to choose. As we will see, the symmetrized inequality we state here still leads to sharp up-to-constants estimates on $k$-point functions in the low-effective-dimensional regime.
\end{remark}

\begin{remark}
In addition to their low-dimensional applications in this section, these inequalities also have applications to bounding error terms in our study of superprocess scaling limits at the upper critical dimension in \cref{I-sec:superprocesses} (see the proof of \cref{I-lem:gyration_derivative}). 
\end{remark}

We can of course apply this inequality recursively to obtain bounds on $k$-point functions in terms of two-point functions. For $k=4$, we obtain the upper bound
\begin{multline}
\label{eq:4pt_Gladkov}
  \tau(x_1,x_2,x_3,x_4)^4
  \\ \preceq \left[\prod_{i<j} \tau(x_i,x_j)\right] \left(\tau(x_1,x_2)\tau(x_3,x_4)+\tau(x_1,x_3)\tau(x_2,x_4)+\tau(x_1,x_4)\tau(x_2,x_3)\right).
\end{multline}
(Since we are not keeping track of constant prefactors there is no difference between taking sums and maxima.) This inequality can be expressed diagrammatically as
\begin{equation}
\label{eq:4pt_Gladkov_diagrammatic}
  \tau(x_1,x_2,x_3,x_4) \preceq \left(\begin{array}{l}\includegraphics{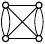}\end{array}+\begin{array}{l}\includegraphics{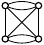}\end{array}+\begin{array}{l}\includegraphics{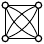}\end{array}\right)^{1/4},
\end{equation}
where each line or curve represents a copy of the two-point function. Continuing the same reasoning, it follows by induction that
\begin{equation}
\label{eq:kpt_Gladkov_graphs}
  \tau(x_1,\ldots,x_k) \preceq_k \left( \sum_Q \prod_{i<j} \tau(x_i,x_j)^{Q_{i,j}}\right)^{2^{-k+2}}
  \asymp_k  \sum_Q \prod_{i<j} \tau(x_i,x_j)^{2^{-k+2} Q_{i,j}}
\end{equation}
where the sum is taken over connected $2^{k-2}$-regular multigraphs $Q$ without self-loops on the vertex set $\{1,\ldots,k\}$ and $Q_{i,j}$ denotes the number of edges between $i$ and $j$. 

\medskip

\begin{proof}[Proof of \cref{thm:higher_Gladkov}]
We keep the proof brief as it is almost identical to the proof of the three-point Gladkov inequality \cite{gladkov2024percolation}. 
Let $A \subseteq V$ be finite, fix $a_0 \in A$, and consider exploring the cluster of $a_0$ in a depth first manner, stopping if and when all but one of the points $A$ have been found. Denote this stopping time by $T$. For each $a \in A\setminus \{a_0\}$, let $\mathscr{R}_a$ denote the event that the depth-first search discovers that the points $A \setminus \{a\}$ are connected before terminating or exploring the vertex $a$, so that $\tau(A)=\sum_{a\in A \setminus \{a_0\}} \P(\mathscr{R}_a \cap \{a_0 \leftrightarrow a\})$.  We can define two coupled copies of the depth-first search that are identical up to the stopping time $T$ and then continue in a conditionally independent manner from this stopping time onwards. If we define $\mathscr{A}$ to be the event that both coupled clusters include all of the points of $A$ and write $\mathscr{A}_a=\mathscr{A}\cap \mathscr{R}_a$ for each $a\in A\setminus \{a_0\}$, we can write
\begin{multline}
\label{eq:Gladkov_Jensen}
  \P(\mathscr{A}_a) = \P(\mathscr{R}_a) \E \left[ \P(a \in K \mid \mathcal{F}_T)^2 \mid \mathscr{R}_a \right] 
  \\\geq \P(\mathscr{R}_a)\P(a \in K\mid \mathscr{R}_a)^2 
  =\frac{\P(\{A\text{ all connected}\}\cap \mathscr{R}_a)^2}{\P(\mathscr{R}_a)}, 
\end{multline}
where $\mathcal{F}_T$ is the sigma-algebra generated by the first stage of the exploration process and where we used Jensen's inequality on the second line. On the other hand, if the event $\mathscr{A}_a$ holds, then in our two coupled clusters we must have a shared tree $H$ of open edges discovered in the first stage of the depth-first search connecting $a_0$ to all of the points in $A\setminus \{a_0,a\}$ together with paths $\gamma_1$ and $\gamma_2$ in each of the two continuations of the cluster connecting this tree to $a$; these two paths $\gamma_1$ and $\gamma_2$ may be taken to use \emph{only} edges that are revealed \emph{after} time $T$, when the two explorations are evolving in a conditionally independent manner, and may also be taken to visit the set of vertices of $H$ only at their starting points. By symmetry, the probability given $\mathscr{A}_a$ that $a_0$ is connected in $H$ to the first vertex of the path $\gamma_1$ by a path that does not use any vertices visited by $\gamma_2$ other than possibly at its endpoint is at least $1/2$.

\begin{figure}
\centering
\includegraphics{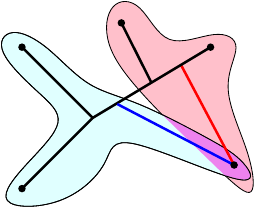}
\caption{Schematic illustration of the disjoint occurrence in the proof of the Gladkov inequality.}
\label{fig:Gladkov_proof}
\end{figure}

\medskip

On this event, we can split the tree $H$ into two edge-disjoint connected subgraphs meeting at the starting point of the path $\gamma_1$, with the vertex sets of these two subgraphs yielding a covering of the set $A\setminus \{a\}$ by two non-trivial subsets. The fact that both subsets are non-trivial follows from the fact that $H$ was computed by a depth-first search, since any edge adjacent to $H$ that was not explored in the first stage of the algorithm (and is therefore a candidate to be the first edge in one of the two paths $\gamma_1$ or $\gamma_2$) must lie on a path from $a_0$ to the last point of $A$ revealed when running the search up to time $T$. (It is possible that the paths meet $H$ \emph{at} this vertex, in which case we take one of the edge-disjoint subgraphs to be the graph consisting of this vertex and no edges. This is not a problem.) See \cref{fig:Gladkov_proof} for an illustration.  As such, we have by Gladkov's BK inequality for decision trees \cite[Theorem 4.3]{gladkov2024percolation} that
\begin{equation}
\label{eq:Gladkov_BK}
  \P(\mathscr{A}_a) \leq  2\sum_{\substack{S \subseteq A \setminus \{a_0,a\} \\ \text{nonempty}}} \tau(S \cup \{a\})\tau(A \setminus S).
\end{equation}
Putting the two inequalities \eqref{eq:Gladkov_Jensen} and \eqref{eq:Gladkov_BK} together and rearranging, we obtain that
\begin{equation*}
  \P(\{A \text{ all connected}\}\cap \mathscr{R}_a)^2 \leq 
  2\P(\mathscr{R}_a)\sum_{\substack{S \subseteq A \setminus \{a_0,a\} \\ \text{nonempty}}} \tau(S \cup \{a\})\tau(A \setminus S)
\end{equation*}
for every $2\leq i \leq n$. Noting that $\P(\mathscr{R}_a) \leq \tau(A \setminus \{a\})$ for every $a\in A\setminus \{a_0\}$, we have by Cauchy-Schwarz that
\begin{align*}
  \tau(A)^2 &= \Bigl(\sum_{a\in A\setminus\{a_0\}} \P(\{A \text{ all connected}\}\cap \mathscr{R}_a)\Bigr)^2
  \\
  &\leq (|A|-1) \sum_{a\in A\setminus \{a_0\}} \P(\{A \text{ all connected}\}\cap \mathscr{R}_a)^2 
  \\ &\leq 2(|A|-1) \sum_{a\in A\setminus \{a_0\}}
  \tau(A\setminus \{a\})\sum_{\substack{S \subseteq A \setminus \{a_0,a\} \\ \text{nonempty}}} \tau(S \cup \{a\})\tau(A \setminus S)
  \\ &\leq 2(|A|-1)^2 (2^{|A|-2}-1) \max \tau(A_1\cup A_2)\tau(A_2\cup A_3)\tau(A_3\cup A_1)
\end{align*}
as claimed, where the maximum is taken over partitions of $A$ into three non-empty subsets.
\end{proof}

\subsection{Geometric Gladkov}
\label{subsec:geometric_gladkov}

In the case that the two-point function has power-law decay, the recursive structure of the higher Gladkov inequality is captured by the function $S$ introduced in \eqref{eq:S_def}:

\begin{corollary}
Let $J$ be a symmetric kernel on $\Z^d$ and suppose that critical long-range percolation with kernel $J$ satisfies a two-point function bound of the form $\tau_{\beta_c}(x,y) \leq C \|x-y\|_2^{-d+2-\eta}$ for some $\eta\in \R$ and $C>0$. Then 
\[\tau_{\beta_c}(A) \preceq_{|A|} S(A)^{-\frac{d-2+\eta}{2}} \]
for every finite set $A \subseteq \Z^d$ with $|A|\geq 2$.
\end{corollary}

Our next goal is to give a more transparent geometric interpretation of this inequality that will be useful when we prove matching lower bounds in \cref{subsec:_k_point_function_hyperscaling}.

\begin{defn}[Sweep and spread]
\label{def:sweep_and_spread}
An \textbf{arboresence} on $A$ is a directed tree with vertex set $A$ oriented towards some root vertex (so that every vertex other than the root has exactly one directed edge emanating from it). Given an arboresence on $A$, we define the \textbf{parent} of each non-root vertex $x$ to be the other endpoint of the unique directed edge in the arboresence emanating from $x$, and for each $x\in A$ define $A_x$ to be the set of vertices consisting of $x$, the parent of $x$, and all vertices having $x$ as an ancestor, setting $A_x=A$ when $x$ is the root vertex. We define the \textbf{sweep}  of a set $A \subseteq \Z^d$ to be
\[
  \operatorname{sweep}(A) = \min \prod_{x\in A} \operatorname{diam}(A_{x})
\]
where the minimum is taken over arboresences on $A$ and diameters are computed with respect to the Euclidean norm. (See \cref{fig:sweep} for an illustration.) Note that if $A=\{x,y\}$ then $\operatorname{sweep}(\{x,y\})=\|x-y\|_2^2$.
Following Benjamini, Kesten, Peres, and Schramm \cite{BeKePeSc04}, we also define the \textbf{spread} of a set $A \subseteq \Z^d$ to be 
\[
  \operatorname{spread}(A) := \min_T \prod_{\{x,y\} \in T} \|x-y\|_2
\]
where the minimum is taken over subsets $T$ of the set of unordered pairs of distinct elements of $A$ that form the edge set of a tree with vertex set $A$. By \cite[Lemma 2.6 and Remark 2.7]{BeKePeSc04}, the spread is given up to constants by
\begin{equation}
\label{eq:spread_simple_expression}
  \operatorname{spread}(\{x_1 \ldots x_n\}) \asymp_n \prod_{i=2}^{n} d(x_i,\{x_1,\ldots,x_{i-1}\}).
\end{equation}
 (It follows in particular that this product does not depend on the ordering of the $x_i$ up to an $n$-dependent constant.)
 \end{defn}
%

\begin{figure}
\centering
\includegraphics[height=6.5cm]{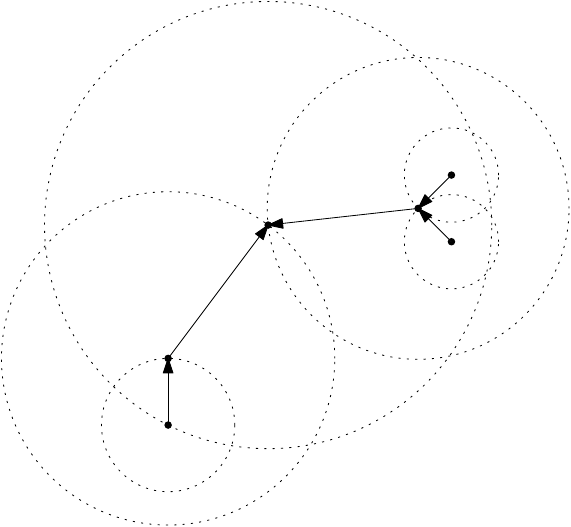}
\caption{An example of the arboresence attaining the minimum in the definition of the sweep on a set of six points in the plane. Each point $x\in A$ is associated to a disc that bounds the diameter of its associated set $A_x$ in the arboresence (dotted circles); the sweep is the product of the radii of these discs in the optimal arboresence.}
\label{fig:sweep}
\end{figure}


The definition we have given of the sweep is chosen to make the proof of our $k$-point \emph{lower bound} of \cref{thm:k_point_S,thm:planar_k_point} go through as cleanly as possible. The following lemma relates the sweep to the spread and therefore yields a simple method to compute the sweep up to constants.

\begin{lemma}
\label{lem:sweep_and_spread}
If $A\subseteq \Z^d$ is finite with $|A| \geq 2$ then $\operatorname{sweep}(A) \asymp_{|A|} \operatorname{diam}(A) \operatorname{spread}(A)$.
\end{lemma}

\begin{proof}[Proof of \cref{lem:sweep_and_spread}]
Given an arboresence $T^\rightarrow$ with associated undirected tree $T$,  we have trivially that
\[
  \prod_{x\in A} \operatorname{diam}(A_x) \geq \operatorname{diam}(A) \prod_{\{x,y\} \in T} \|x-y\|_2,
\]
where the $\operatorname{diam}(A)$ term comes from the root of the arboresence. Taking minima, this implies that 
\[\operatorname{sweep}(A) \geq \operatorname{diam}(A) \operatorname{spread}(A)\]
for every set $A\subseteq \Z^d$.
To prove the other bound, first observe that 
if we fix $a_1 \in A$ arbitrarily and then at each subsequent step pick $a_i \in A \setminus \{a_1,\ldots,a_{i-1}\}$ to maximize $d(a_i,\{a_1,\ldots,a_{i-1}\})$ then the resulting sequence of distances $(d(a_i,\{a_1,\ldots,a_{i-1}\}))_{i=2}^n$ is decreasing. If we consider the associated arboresence that has root $a_1$ and $a_i$ connected to the point in $\{a_1,\ldots,a_{i-1}\}$ of minimal distance to $a_i$ for each $2\leq i \leq n$, which we denote by $p(a_i)$, then the diameter of the set $A_{a}$ is bounded by $|A|\cdot \|a_i-p(a_i)\|_2$ for each $2\leq i \leq n$. This ensures that 
\[
  \operatorname{sweep}(A) \leq |A|^{|A|-1} \operatorname{diam}(A)   \prod_{i=2}^n d(a_i,\{a_1,\ldots,a_{i-1}\}) \asymp \operatorname{diam}(A) \operatorname{spread}(A),
\]
where the final estimate follows from \eqref{eq:spread_simple_expression}. \qedhere
\end{proof}

We now want to relate the sweep to the recursive structure of the Gladkov inequality. This lemma also establishes the up-to-constants equivalence between the sweep and the quantity $S(A)$ defined in the introduction: We recall that
this quantity was defined recursively by $S(x,y)=\|x-y\|_2^2$ and  
$S(A) = \min \sqrt{S(A_1 \cup A_2)S(A_2 \cup A_3)S(A_3 \cup A_1)}$ when $|A|\geq 3$, where the minimum is taken over partitions of $A$ into three non-empty subsets.

\begin{lemma}
\label{lem:sweep_recursion}
If $A\subseteq \Z^d$ is finite with $|A|\geq 3$ then
\[\operatorname{sweep}(A) \asymp_{|A|} \min \sqrt{\operatorname{sweep}(A_1 \cup A_2)\operatorname{sweep}(A_2 \cup A_3)\operatorname{sweep}(A_3 \cup A_1)},
\]
 where the minimum is taken over partitions of $A$ into three non-empty subsets $A_1$, $A_2$, and $A_3$. Consequently, the estimate
 \[
   \operatorname{sweep}(A) \asymp_{|A|} S(A)
 \]
holds for every subset $A \subseteq \Z^d$ with at least two points.
\end{lemma}

\begin{proof}[Proof of \cref{lem:sweep_recursion}]
All constants in this lemma may depend on $|A|$, and we omit this dependence from our notation throughout the proof.
We begin with the upper bound.
Let $x,y\in A$ satisfy $\|x-y\|_2 =\operatorname{diam}(A)$, let $A=A_1\cup A_2 \cup A_3$ be a partition of $A$ into three non-empty sets, and assume without loss of generality that $x,y\notin A_1$ so that $\operatorname{diam}(A_2\cup A_3)=\operatorname{diam}(A)$.
If we pick $z \in A_1$ then either $\|z-x\|_2\geq \frac{1}{2} \|x-y\|_2$ or $\|z-y\|_2\geq \frac{1}{2} \|x-y\|_2$, and if we assume without loss of generality that $x\in A_2$ and $\|z-x\|_2\geq \frac{1}{2} \|x-y\|_2$ then
\begin{equation}
\label{eq:tripartition_diameter}
  \operatorname{diam}(A_1 \cup A_2) \operatorname{diam}(A_2 \cup A_3) = \operatorname{diam}(A_1 \cup A_2) \operatorname{diam}(A) \geq \frac{1}{2} \operatorname{diam}(A)^2.
\end{equation}
Pick an enumeration $A=\{a_1,\ldots,a_n\}$ with $a_1\in A_1$ and $a_2\in A_3$, and for each $i \geq 1$ and $j\in \{1,2,3\}$ let $A_{j,i}^* = \{a_1,\ldots,a_{i-1}\}\setminus A_j$. We can use \eqref{eq:spread_simple_expression} to write
\begin{multline*}
  \operatorname{spread}(A_1 \cup A_2)\operatorname{spread}(A_2 \cup A_3)\operatorname{spread}(A_3 \cup A_1)
  \\ \asymp \prod_{i\in A_1} d(a_i,A_{2,i}^*)d(a_i,A_{3,i}^*) \prod_{i\in A_2} d(a_i,A_{1,i}^*)d(a_i,A_{3,i}^*) \prod_{i\in A_3} d(a_i,A_{1,i}^*)d(a_i,A_{2,i}^*)
\end{multline*}
where we set $d(a_i,A_j^*)=1$ if $A_j^*=\emptyset$, which is only possible if $i=1$ or $j=1$ and $i=2$. Observe that the distance $d(a_i,A_j^*)$ is at least the distance $d(a_i,\{a_n:n<i\})$ except when $j=1$ and $i=2$, so that
\begin{align*}
  \operatorname{spread}(A_1 \cup A_2)\operatorname{spread}(A_2 \cup A_3)\operatorname{spread}(A_3 \cup A_1) &\succeq \frac{\operatorname{spread}(A)^2}{d(a_1,a_2)} \geq \frac{\operatorname{spread}(A)^2}{\operatorname{diam}(A_1 \cup A_3)}.
\end{align*}
Putting this together with \eqref{eq:tripartition_diameter} yields that
\[\operatorname{sweep}(A) \preceq \sqrt{\operatorname{sweep}(A_1 \cup A_2)\operatorname{sweep}(A_2 \cup A_3)\operatorname{sweep}(A_3 \cup A_1)}, \]
and the claimed upper bound follows since our tripartition $A=A_1\cup A_2 \cup A_3$ was arbitrary (up to our choice of labellings of the three sets which we assumed had some specific properties without loss of generality). 

\medskip

We now turn to the lower bound. 
Let $|A|\geq 3$, let $a_1,a_2 \in A$ be such that $d(a_1,a_2)=\operatorname{diam}(A)$, 
let $a_3$ be chosen to maximize $d(a_3,\{a_1,a_2\})$ 
 and, if necessary, reorder the pair $a_1,a_2$ so that $d(a_3,a_1) \leq d(a_3,a_2)$.
  We define the sets $A_1,A_2,A_3$ greedily starting with $a_1\in A_1$, $a_2\in A_2$, and $a_3\in A_3$ and for each $i\geq 4$ letting $a_i$ belong to the set $A_j$ for which $d(a_i,\{a_1,\ldots,a_{i-1}\}) = d(a_i,\{a_1,\ldots,a_{i-1}\} \cap A_j)$.
\begin{multline}
  \operatorname{spread}(A_1 \cup A_2)\operatorname{spread}(A_2 \cup A_3)\operatorname{spread}(A_3 \cup A_1)
 \\  \asymp d(a_1,a_2) d(a_2,a_3) d(a_3,a_1) \prod_{i=4}^n d(a_i,\{a_j:1\leq j \leq i\})^2 
\\ \asymp \frac{d(a_2,a_3)}{d(a_1,a_2)d(a_1,a_3)} \operatorname{spread}(A)^2\asymp \frac{\operatorname{spread}(A)^2}{d(a_1,a_3)},
\label{eq:spread_d(a1a3)}
\end{multline}
where in the final estimate we used that $a_3$ is closer to $a_1$ than to $a_2$, we must have that $d(a_2,a_3) \asymp d(a_1,a_2)$. The same reasoning shows that $\operatorname{diam}(A)=\operatorname{diam}(A_1 \cup A_2) \asymp \operatorname{diam}(A_2 \cup A_3)$, so that to complete the proof it suffices to show that $\operatorname{diam}(A_1 \cup A_3) \asymp d(a_1,a_3)$.
If $d(a_1,a_3) \geq \frac{1}{4} d(a_1,a_2)$ then both distances are of the same order as $d(a_1,a_2)=\operatorname{diam}(A)$ and the claim is immediate. Otherwise, by choice of $a_1,a_2,a_3$, the entire set $A$ is contained in the union of the two balls of radius $d(a_1,a_3) \leq \frac{1}{4} \operatorname{diam}(A)$ around $a_1$ and $a_2$. In this case, our definition of the three sets $A_1,A_2$, and $A_3$ ensures that $A_1$ and $A_3$ are contained in the ball of radius $d(a_1,a_3)$ around $a_1$, while $A_2$ is contained in the ball of radius $d(a_1,a_3)$ around $a_2$. In this case we therefore have that $\operatorname{diam}(A_1 \cup A_3) \leq 2 d(a_1,a_3)$ and the desired inequality
\[
  \operatorname{sweep}(A_1 \cup A_2)\operatorname{sweep}(A_2 \cup A_3)\operatorname{sweep}(A_3 \cup A_1) \preceq \operatorname{sweep}(A)^2
\]
follows by substituting these diameter estimates into \eqref{eq:spread_d(a1a3)} and rearranging.\qedhere
\end{proof}

We are now ready to state our geometric version of the higher Gladkov inequality for percolation models with polynomially decaying two-point functions.

\begin{corollary}[Geometric Gladkov]
\label{cor:geometric_Gladkov}
Let $J$ be a symmetric kernel on $\Z^d$ and suppose that critical long-range percolation with kernel $J$ satisfies a two-point function bound of the form $\tau_{\beta_c}(x,y) \leq C \|x-y\|_2^{-d+2-\eta}$ for some $\eta\in \R$ and $C>0$. Then 
\[\tau_{\beta_c}(A) \preceq_{|A|} \operatorname{sweep}(A)^{-\frac{d-2+\eta}{2}} \asymp_{|A|} S(A)^{-\frac{d-2+\eta}{2}} \]
for every finite set $A \subseteq \Z^d$ with $|A|\geq 2$.
\end{corollary}

\begin{proof}[Proof of \cref{cor:geometric_Gladkov}]
 We prove the claim by induction on $|A|$. The base case for sets of size two holds by assumption, while if $|A|>2$ then
\begin{align*}
  \tau(A) &\preceq_{|A|} \max \sqrt{\tau(A_1 \cup A_2) \tau(A_2 \cup A_3)\tau(A_3\cup A_1)}
  \\
  &\preceq_{|A|} \left( \min \operatorname{sweep}(A_1 \cup A_2)\operatorname{sweep}(A_2 \cup A_3)\operatorname{sweep}(A_3 \cup A_1)\right)^{-(d-2+\eta)/2}
  \\
  &\preceq_{|A|} \operatorname{sweep}(A)^{-(d-2+\eta)/2},
\end{align*}
where we used the \cref{thm:higher_Gladkov} in the first inequality, the induction hypothesis in the second, and \cref{lem:sweep_recursion} in the third.
\end{proof}

See \cref{subsec:_k_point_function_hyperscaling} for an interpretation of the geometric Gladkov inequality as a hyperscaling inequality.

\subsection{M\"obius invariance of $S$}
\label{subsec:conformal_invariance}

We now connect the functions $S$ and $\operatorname{Sweep}$ studied in \cref{subsec:geometric_gladkov} to the discussion of conformal invariance in the introduction by proving that the functional $S(A)$ is conformally covariant.

\begin{prop}
\label{prop:sweep_conformal_invariance}
The function $S(A)$ is conformally covariant of scaling dimension $-1$:
If $\psi \in \Mobius(\R^d)$ then
\[
  S(\psi(A)) = S(A) \prod_{a\in A} |\det D \psi(a)|^{1/d}
\]
for every finite set $A \subseteq \R^d \setminus \psi^{-1}(\infty)$.
\end{prop}

This Proposition is used in particular to deduce \cref{cor:Mobius} from \cref{thm:k_point_S}.


\begin{proof}[Proof of \cref{prop:sweep_conformal_invariance}]
 We will prove the claim by induction on $|A|$.
We begin by verifying the claim for sets of size two, in which case $S(\{x,y\})=\|x-y\|_2^2$ is just the squared Euclidean distance. 
 Since $S$ clearly transforms in the desired way under dilations and Euclidean isometries, it suffices to verify the claim for the inversion in the unit sphere, $\psi(x)=-x/\|x\|_2^2$. The Jacobian of this inversion map is
\[
 \det D\psi(x) = - \|x\|_2^{-2d},
\]
so that it suffices to show that
\[
  \|\psi(x)-\psi(y)\|_2^2 = \frac{\|x-y\|_2^2}{\|x\|_2^2\|y\|_2^2}.
\]
To prove this identity, it suffices moreover to consider the case $d=2$, as we can do all computations in the plane spanned by $x$, $y$, and $0$ (which is fixed by $\psi$). Following this reduction to $\R^2$, we can use complex notation to write $\psi(z)=1/z$ and compute
\[
  \psi(z_1)-\psi(z_2) = \frac{1}{z_1}- \frac{1}{z_2} = \frac{z_2-z_1}{z_1z_2}
\qquad\text{so that}\qquad
  |\psi(z_1)-\psi(z_2)|^2 = \frac{|z_1-z_2|^2}{|z_1|^2|z_2|^2}
\]
as claimed.

Now suppose that $|A|\geq 3$ and that the claim has been established for all strictly smaller sets. To lighten notation we write $D_\psi(A):= \prod_{a\in A} |\!\det D \psi(a)|^{1/d}$. We have by the recursive definition of $S(A)$ and the induction hypothesis that 
\begin{align*}
  S(\psi(A)) &= \min \left[S(\psi(A_1) \cup \psi(A_2))S(\psi(A_2) \cup \psi(A_3))S(\psi(A_3) \cup \psi(A_1))\right]^{1/2}
  \\
  &= \min \bigl[
  D_\psi(A_1\cup A_2)
   S(A_1 \cup A_2) D_\psi(A_2\cup A_3) S(A_2 \cup  A_3)
    D_\psi(A_3\cup A_1) S(A_3 \cup A_1)\bigr]^{1/2}
   \\
   &= \min D_\psi(A) \left[
   S( A_1 \cup A_2) S(A_2 \cup A_3) S( A_3 \cup A_1)\right]^{1/2}
   \\
   &=D_\psi(A) S(A),
\end{align*}
where in the third equality we used that
\[
  D_\psi(A_1 \cup A_2) D_\psi(A_2 \cup A_3) D_\psi(A_3 \cup A_1) = D_\psi(A)^2
\]
for any partition of $A$ into three non-empty subsets (a property that holds for any functional of the form $F(A)=\prod_{a\in A}f(a)$). This completes the induction step and hence the proof.
\end{proof}

\section{The maximum cluster size}
\label{sec:the_maximum_cluster_size}

The primary goal of this section is to understand the size of the largest critical cluster in the ball of radius $r$ in the effectively long-range low-dimensional regime. More precisely, we will study the large-$r$ asymptotics of the \textbf{edians}
\[
  M_r =\min\left\{n\geq 0: \P_{\beta_c,r}\left(\max_{x\in \Z^d} |K_x \cap B_r| \geq n\right)\leq e^{-1}\right\}
\]
and
\[
  M_r^* =\min\left\{n\geq 0: \P_{\beta_c}\left(\max_{x\in \Z^d} |K_x \cap B_r| \geq n\right)\leq e^{-1}\right\},
\]
which capture the typical size of the largest intersection of cluster with the ball $B_r$ under the measures $\P_{\beta_c,r}$ and $\P_{\beta_c}$ respectively. (The universal tightness theorem \cite[Theorem 2.2]{hutchcroft2020power} implies that $M_r$ and $M_r^*$ are of the same order as the \emph{mean} of the same random variables.) Our analysis of this quantity, along with the proof of \cref{thm:main_low_dim} in the next section, will not require the full strength of \eqref{CL} but can be carried out using the following weaker condition, which we show is implied by \eqref{CL} in \cref{lem:CL_and_wCL}.

\begin{defn}
\label{def:wCL}
 We say that the model satisfies the \textbf{weak correlation length condition for effectively long-range critical behaviour} \eqref{wCL}
  if $\beta_c<\infty$ and
%
for each $\eps>0$ there exists a constant $C=C(\eps)$ such that
\begin{equation}
\label{wCL}
\tag{wCL}
 \E_{\beta_c,r} |K|^p \leq (1+\eps)\E_{\beta_c,r} |K \cap B_{Cr}|^p
\end{equation}
for every $r\geq 1$ and each $p=1,2,3,4$. 
\end{defn}

The following theorem is the main result of this section.

\begin{theorem}[The maximum cluster size]
\label{thm:max_cluster_size_LD}
If $d<3\alpha$ and \eqref{wCL} holds then $M_r \asymp M_r^* \asymp r^{(d+\alpha)/2}$ for every $r\geq 1$.
\end{theorem}

Let us now give some context for this theorem and outline its proof.
As explained in detail in \cref{I-sub:previous_results_on_long_range_percolation} (see \cref{I-cor:M_upper_bound_general}), it is an immediate consequence of the results of \cite{hutchcroft2022sharp} that 
\begin{equation}
\label{eq:M_r_upper_restate}
M_r \leq M_r^* \preceq r^{(d+\alpha)/2}
\end{equation} for every $r\geq 1$, regardless of the value of $\alpha>0$ (the bound being vacuous for $\alpha\geq d$). The geometric significance of this bound is that if $A_1$ and $A_2$ are disjoint subsets of the ball $B_r$ then the expected number of edges between $A_1$ and $A_2$ that are open in the configuration associated to the measure $\P_{\beta_c,2r}$ but not $\P_{\beta_c,r}$ (in the standard monotone coupling) is of order $|A_1||A_2|r^{-d-\alpha}$, so that the probability such an open edge exists is high when $|A_1|,|A_2| \gg r^{(d+\alpha)/2}$, low when $|A_1|,|A_2| \ll r^{(d+\alpha)/2}$, and bounded away from $0$ and $1$ when $|A_1|$ and $|A_2|$ are both of order $r^{(d+\alpha)/2}$. Thus, the bound \eqref{eq:M_r_upper_restate} shows that the largest clusters in two adjacent balls of radius $r$ cannot have a \emph{high} probability of merging via the direct addition of a single edge when we pass from scale $r$ to scale $2r$. The fact that the inequality \eqref{eq:M_r_upper_restate} is sharp in low effective dimensions (as established by \cref{thm:max_cluster_size_LD}) but \emph{not} in high or critical effective dimensions (as proven in \cref{I-cor:HD_hydro} and \cref{III-thm:critical_dim_hydro}) is an important qualitative distinction between these two regimes, closely related to the (in)validity of the \emph{hyperscaling relations} (as discussed in more detail in \cref{subsec:hyperscaling}). Note in particular that when $M_r \asymp r^{(d+\alpha)/2}$ the large clusters on scale $r$ also have a good probability to be joined by \emph{more than one} edge on the same scale, so that clusters contain macroscopic cycles on all scales and do \emph{not} look like trees; this contrasts the high- and critical-dimensional cases in which clusters look like trees and have superprocess scaling limits (\cref{I-thm:superprocess_main,III-cor:superprocess_main_CD}).


\medskip

 Following \cref{I-def:Hydro}, we say that the \textbf{hydrodynamic condition} \eqref{Hydro} holds if the upper bound \eqref{eq:M_r_upper_restate} admits a strict improvement in the sense that
%
%
\begin{equation}\label{Hydro}
\tag{Hydro}
M_r=o(r^{(d+\alpha)/2}) \qquad \text{ as $r\to \infty$}.
\end{equation}
The hydrodynamic condition is proven to hold for $d>3\alpha$ in \cref{I-cor:HD_hydro} and for $d=3\alpha$ in \cref{III-thm:critical_dim_hydro}, and is shown to imply the validity of the mean-field asymptotic ODEs \eqref{eq:volume_moments_ODE_intro} for $d= 3\alpha$ in \cref{I-sec:analysis_of_moments}. (When $d>3\alpha$ we prove that these ODEs hold without using the hydrodynamic condition directly.)
In \cref{sub:the_hydrodynamic_condition_does_not_hold} we show that the mean-field asymptotic ODE \eqref{eq:volume_moments_ODE_intro} also holds for $p=1,2$ whenever \eqref{CL} and \eqref{Hydro} both hold, regardless of the value of $d$ and $\alpha$. 
(Studying these ODEs for the $p$th moment leads to error terms involving the $(p+2)$th moment, which is why we need assumptions on the first \emph{four} moments in the definition of \eqref{wCL}.) 
 We then show that these asymptotic ODEs are \emph{inconsistent} with the upper bound $M_r^* \preceq r^{(d+\alpha)/2}$ when $d<3\alpha$, so that in fact \eqref{CL} and \eqref{Hydro} \emph{cannot} both hold when $d<3\alpha$. 
 To prove the first part of \cref{thm:hyperscaling} we need the stronger result that $M_r \succeq r^{(d+\alpha)/2}$ on \emph{every} scale (rather than on an unbounded set of scales as shown by the argument we have just sketched); this is proven via a more careful,  finitary version of the same argument in \cref{subsec:low_dim_max_cluster}. Together with the upper bound \eqref{eq:M_r_upper_restate} this shows that $M_r \asymp r^{(d+\alpha)/2}$, completing the proof of \cref{thm:max_cluster_size_LD} and giving us our first computation of a non-mean-field exponent $d_f=(d+\alpha)/2$ under the hypotheses of \cref{thm:main_low_dim}.


\subsection{Volume moments under \eqref{CL} and the weak correlation length condition}

In this section we prove the following two results regarding volume moments in the effectively long-range low-dimensional regime.

\begin{theorem}
\label{thm:CL_volume_moments}
If \eqref{CL} holds and $d<3\alpha$ then $\E_{\beta_c,r}|K|^p \asymp_p r^{\alpha+(p-1)\frac{d+\alpha}{2}}$ for each integer $p\geq 1$ and every $r\geq 1$.
\end{theorem}


\begin{lemma}
\label{lem:CL_and_wCL}
If \eqref{CL} holds and $d<3\alpha$ then for each integer $p\geq 1$ and $\eps>0$ there exists a constant $C=C(p,\eps)$ such that
\[
  \E_{\beta_c,r}|K|^p \leq (1+\eps) \E_{\beta_c,r}|K \cap B_{C r}|^p
\]
for every $r\geq 1$. 
In particular, if \eqref{CL} holds and $d<3\alpha$ then \eqref{wCL} holds also.
\end{lemma}

\begin{remark}
The conclusions of \cref{lem:CL_and_wCL} also hold when $d\geq 3\alpha$ and \eqref{CL} holds (equivalently, $d\geq 3\alpha$ and $\alpha<2$), in which case it follows from \cref{I-thm:scaling_limit_diagrams,III-thm:critical_dim_hydro,cor:HD_CL}.
\end{remark}


We begin by using the Gladkov inequality to prove the following upper bounds on cluster volume moments under \eqref{CL}.

\begin{lemma}
\label{lem:sCL_moments}
If \eqref{CL} holds then
\[
  \E_{\beta_c,r}|K|^p \preceq_p r^{\alpha+ (p-1)\frac{d+\alpha}{2}} \qquad \text{ and } \qquad \E_{\beta_c,r}\left[|K|^{p-1}\sum_{x\in K} \|x\|^2_2 \right]  \preceq_p  r^{2+\alpha+ (p-1)\frac{d+\alpha}{2}}
\]
for each integer $p\geq 1$ and $r\geq 1$.
\end{lemma}

\begin{proof}[Proof of \cref{lem:sCL_moments}]
We may assume by \cref{cor:large_alpha_not_CL} that $\alpha \leq d$.
Fix $r\geq 1$ and write $\tau=\tau_{\beta_c,r}$ for the $k$-point function associated to $\P_{\beta_c,r}$.
Iteratively applying the higher Gladkov inequality (\cref{thm:higher_Gladkov}) as in \eqref{eq:kpt_Gladkov_graphs} yields that 
\begin{equation}
\label{eq:kpt_Gladkov_graphs_restate}
 \tau(x_1,\ldots,x_k) \preceq_k \sum_{Q} \left(\prod_{i<j} \tau(x_i,x_j)^{Q_{i,j}} \right)^{2^{-k+2}}
\end{equation}
where the sum is taken over adjacency matrices of $2^{k-2}$-regular connected multigraphs over $\{1,\ldots,k\}$ without self-loops. 
For each such multigraph $Q$, $k$-tuple of points $x_1,\ldots,x_k$, and index $2\leq \ell \leq k$ we define
\[
  \tau_Q(x_1,\ldots,x_\ell) :=  \left(\prod_{i<j \leq \ell} \tau(x_i,x_j)^{Q_{i,j}} \right)^{2^{-k+2}}.
\]
Letting $Q_\ell = \sum_{i<\ell} Q_{i,\ell}$ for each $1 \leq \ell \leq k$, we can apply H\"older's inequality then use transitivity to deduce that
\begin{multline*}
 \sum_{x_\ell \in \Z^d} \tau_Q(x_1,\ldots,x_\ell) 
 =\tau_Q(x_1,\ldots,x_{k-1}) \sum_{x_\ell\in \Z^d} \prod_{i<\ell} \tau(x_i,x_\ell)^{2^{-k+2}Q_{i,\ell}}
 \\\leq \tau_Q(x_1,\ldots,x_{k-1}) \sum_{x_\ell \in \Z^d} \tau(x_1,x_\ell)^{2^{-k+2} Q_\ell}
\end{multline*}
and hence by induction on $\ell$ that
\begin{equation}
\label{eq:Gladkov_to_moments}
  \sum_{x_3,x_4,\ldots,x_k \in \Z^d} \tau_Q(x_1,\ldots,x_k) \leq \tau(x_1,x_2)^{2^{-k+2}Q_{1,2}} \prod_{\ell=3}^k \sum_{x_\ell \in \Z^d} \tau(x_1,x_\ell)^{2^{-k+2} Q_\ell}.
\end{equation}
Using \cref{lem:CL_Sak_upper}, the assumption \eqref{CL} ensures that there exists a decreasing function $h:(0,\infty)\to(0,1]$ decaying faster than any power such that 
$\tau_{\beta_c,r}(0,x)\preceq \|x-y\|^{-d+\alpha}h(\|x\|/r)$ and hence that
\begin{equation}
\label{eq:sCL_power_of_tau_sum}
\sum_{x\in \Z^d} \tau_{\beta_c,r}(0,x)^{\eta} \preceq_\eta \sum_{x\in \Z^d} \|x\|^{-\eta(d-\alpha)} h(\|x\|/r)^\eta \preceq_\eta r^{d-\eta(d-\alpha)}
\end{equation}
and
\[
\sum_{x\in \Z^d} \|x\|^2_2 \tau_{\beta_c,r}(0,x)^{\eta} \preceq_\eta \sum_{x\in \Z^d} \|x\|^{2-\eta(d-\alpha)} h(\|x\|/r)^\eta \preceq_\eta r^{2+d-\eta(d-\alpha)}
\]
for every $0<\eta\leq 1$ and $r\geq 1$, where we applied \cref{lem:scaling_sums} in both estimates.
 Substituting these estimates into \eqref{eq:Gladkov_to_moments} yields that
\[
  \sum_{x_2,\ldots,x_k \in \Z^d} \tau_Q(x_1,\ldots,x_k) \preceq_k r^{d(k-1)-(d-\alpha) 2^{-k+2}\sum_{\ell=2}^k Q_\ell} = r^{d(k-1)-\frac{k}{2}(d-\alpha)} = r^{\alpha +(k-2)\frac{d+\alpha}{2}}
\]
and 
\[
  \sum_{x_2,\ldots,x_k \in \Z^d} \|x_1-x_2\|_2^2 \tau_Q(x_1,\ldots,x_k) \preceq_k r^{2+d(k-1)-(d-\alpha) 2^{-k+2}\sum_{\ell=2}^k Q_\ell} = r^{2+\alpha +(k-2)\frac{d+\alpha}{2}},
\]
where we used that $\sum_{\ell=2}^k Q_\ell = k 2^{k-3}$ is the number of edges in the multigraph $Q$. The claim follows from this and \eqref{eq:kpt_Gladkov_graphs_restate}.
\end{proof}

In order to complete the proofs of \cref{thm:CL_volume_moments,lem:CL_and_wCL}, it suffices to prove that 
\[\E_{\beta_c,r}|K|^p \succeq_p r^{\alpha+(p-1) \frac{d+\alpha}{2}}\]
for every $p\geq 1$ under the assumption that $d<3\alpha$ and \eqref{CL} holds. 
(Indeed, \cref{thm:CL_volume_moments} is an immediate consequence of this and the first estimate of \cref{lem:sCL_moments}, while \cref{lem:CL_and_wCL} follows from this and the second estimate of \cref{lem:sCL_moments} by Markov's inequality by the same argument used in \eqref{eq:CLp_application}.)
Since we already have that $\E_{\beta_c,r}|K| \asymp r^\alpha$ by \cref{I-cor:mean_lower_bound} and \cref{lem:sCL_moments}, it suffices to prove that 
\[\E_{\beta_c,r}|K|^2 \succeq r^{\alpha+\frac{d+\alpha}{2}},\] with the corresponding lower bounds on higher moments following by H\"older's inequality.
(Indeed, the ratios $\E X^{p+1}/\E X^p$ are increasing in $p\geq 0$ for any non-negative random variable $X$.) The proof of the lower bound on the second moment will make use of the following inequality for general transitive weighted graphs, which is another consequence of the higher Gladkov inequality. Recall that a transitive weighted graph $G=(V,E,J)$ is said to be \textbf{unimodular} if it satisfies the mass-transport principle $\sum_{x\in V}F(o,x)=\sum_{x\in V}F(x,o)$ for every $F:V^2\to [0,\infty]$ satisfying $F(\gamma x,\gamma y)=F(x,y)$ for every $x,y\in V$ and automorphism $\gamma$ of $G$.
 To see that this holds for our translation-invariant model on $\Z^d$ note simply that
 \begin{equation}
 \label{eq:MTP}
   \sum_{x\in \Z^d} F(0,x)=\sum_{x\in \Z^d} F(0,-x)=\sum_{x\in \Z^d} F(x,0).
 \end{equation}
More generally, if $G=(V,E,J)$ is unimodular and if $F:V^k \to[0,\infty]$ satisfies $F(\gamma x_1,\ldots,\gamma x_k)=F( x_1,\ldots, x_k)$ for every $x_1,\ldots,x_k\in V$ and automorphism $\gamma$ then
 \begin{multline}
 \label{eq:MTP_general}
   \sum_{x_2,\ldots,x_k \in V} F(o,x_2,\ldots,x_k)=\sum_{x_1,x_3,\ldots,x_k\in V} F(x_1,o,x_3,\ldots,x_k)\\=\cdots = \sum_{x_1,\ldots,x_{k-1}\in V} F(x_1,x_2,\ldots,x_{k-1},o),
 \end{multline}
 as follows from repeated application of \eqref{eq:MTP}.

\begin{lemma}
\label{lem:third_moment_Gladkov}
Consider percolation on a unimodular transitive weighted graph $G=(V,E,J)$ and let $o$ be a vertex of $G$. The third moment of the volume of the cluster of $o$ can be bounded by
\[\E_{\beta}|K|^3 \preceq \E_\beta |K|^2 \sum_{x\in V} \sqrt{\tau_\beta(o,x)}\]
for every $\beta \geq 0$, where the implicit constant is universal.
\end{lemma}

\begin{proof}[Proof of \cref{lem:third_moment_Gladkov}]
We have by the higher Gladkov inequality (\cref{thm:higher_Gladkov}) that the four-point function can be bounded 
\[
  \tau(x_1,x_2,x_3,x_4) \preceq \sum_\pi \sqrt{\tau(x_{\pi(1)},x_{\pi(2)})\tau(x_{\pi(1)},x_{\pi(3)},x_{\pi(4)})\tau(x_{\pi(2)},x_{\pi(3)},x_{\pi(4)})}
\]
where the sum is taken over permutations of $\{1,2,3,4\}$. Since $G$ is unimodular, the sum 
\[\sum_{x_2,x_3,x_4}\sqrt{\tau(x_{\pi(1)},x_{\pi(2)})\tau(x_{\pi(1)},x_{\pi(3)},x_{\pi(4)})\tau(x_{\pi(2)},x_{\pi(3)},x_{\pi(4)})}\] does not depend on the choice of permutation $\pi$ and we obtain by Cauchy-Schwarz that
\begin{align*}
\E_\beta |K|^3 &\preceq \sum_{x,y,z\in V} \sqrt{\tau(o,y,z) \tau(x,y,z) \tau(o,x)} 
\\
&\leq \sum_{x\in V} \sqrt{\tau(o,x)} \sqrt{\sum_{y,z} \tau(o,y,z) \sum_{y,z} \tau(x,y,z)} = \E_\beta |K|^2 \sum_{x\in V} \sqrt{\tau_\beta(o,x)}
\end{align*}
as claimed.
\end{proof}

Returning to our usual setting of long-range percolation on $\Z^d$, we saw in \eqref{eq:sCL_power_of_tau_sum} that if \eqref{CL} holds then $\sum_{x\in \Z^d} \sqrt{\tau_{\beta_c,r}(0,x)} \preceq r^{(d+\alpha)/2}$, so that \cref{lem:third_moment_Gladkov} implies that
\begin{equation}
\label{eq:third_moment_Gladkov_sCL}
  \E_{\beta_c,r}|K|^3 \preceq  r^{(d+\alpha)/2} \E_{\beta_c,r}|K|^2
\end{equation}
for every $r\geq 1$. The following proof
 will also apply the correlation inequalities of \cref{I-lem:BK_disjoint_clusters_covariance}, which state that if $G=(V,E,J)$ is a weighted graph and $\beta\geq 0$ then
 \begin{equation}
 \label{eq:BK_disjoint_clusters_covariance}
\E_\beta\left[F(K_x)G(K_y)\mathbbm{1}(x\nleftrightarrow y)\right] \leq \E_\beta \left[F(K_x)\right] \E_\beta \left[G(K_y)\right]
\end{equation}
holds for every $x,y\in V$ and every pair of increasing non-negative functions $F$ and $G$ and that
\begin{equation}
 \label{eq:BK_disjoint_clusters_covariance2}
\E_\beta\left[F(K_x)G(K_y)\mathbbm{1}(x\nleftrightarrow y)\right] \geq \E_\beta \left[F(K_x)\right] \E_\beta \left[G(K_y)\right] - \E_\beta\left[F(K_x)G(K_y)\mathbbm{1}(x\leftrightarrow y)\right]
\end{equation}
and
\begin{equation}
\label{eq:FG_covariance}
  0\leq \E_\beta [F(K_x)G(K_y)] -  \E_\beta \left[F(K_x)\right] \E_\beta \left[G(K_y)\right] \leq \E_\beta\left[F(K_x)G(K_y)\mathbbm{1}(x\leftrightarrow y)\right]
\end{equation}
hold for every $x,y\in V$ and every pair of increasing non-negative functions $F$ and $G$ such that $\E_\beta F(K_x),\E_\beta G(K_y)<\infty$. (In fact in this proof we will only use \eqref{eq:BK_disjoint_clusters_covariance2}.)


\begin{proof}[Proof of \cref{thm:CL_volume_moments,lem:CL_and_wCL}]
As explained above, it suffices to prove that $\E_{\beta_c,r}|K|^2 \succeq r^{\alpha+\frac{d+\alpha}{2}}$ for every $r\geq 1$. To prove this, it suffices to prove that there exists $\delta>0$ and $r_0<\infty$ such that if $r\geq r_0$ and $\E_{\beta_c,r}|K|^2 \leq \delta r^{\alpha+(d+\alpha)/2}$ then
\[
 r^{\alpha + (d+\alpha)/2} \frac{d}{dr}\left(r^{-\alpha-(d+\alpha)/2} \E_{\beta_c,r}|K|^2 \right) =  \frac{d}{dr}  \E_{\beta_c,r}|K|^2
 - \left(\alpha+\frac{d+\alpha}{2}\right)  r^{-1} \E_{\beta_c,r}|K|^2 \geq 0.
\]
Indeed, once this is proven it will follow by the mean value theorem that $r^{-\alpha-(d+\alpha)/2} \E_{\beta_c,r}|K|^2 \geq \min\{\delta, r_0^{-\alpha-(d+\alpha)/2} \E_{\beta_c,r_0}|K|^2\}$ for every $r\geq r_0$.
By \eqref{eq:third_moment_Gladkov_sCL}, there exists a constant $C$ such that 
\begin{equation}
\label{eq:sCL_third_moment_upper}
\E_{\beta_c,r}|K|^3 \leq C r^{(d+\alpha)/2} \E_{\beta_c,r}|K|^2
\end{equation}
 for every $r\geq 1$.
  Let $\delta>0$ be such that
\begin{equation}
\label{eq:sCL_delta_def1}
3 \alpha (1-\delta^{1/4})(1-\delta)  - 3 \beta_c (1+\delta) \delta^{1/2} C  \geq \alpha + \frac{d+\alpha}{2}
\end{equation}
and
\begin{equation}
\label{eq:sCL_delta_def2}
(1-\delta) \left(1-\delta^{1/2} (1-\delta)^{-1} \frac{\beta_c}{\alpha}\right) \geq (1-\delta^{1/4}),
\end{equation}
where the first property is possible to satisfy since $d<3\alpha$, 
and let $r_0$ be sufficiently large that 
\begin{equation}
\label{eq:sCL_r_0_def}
\E_{\beta_c,r}|K| \geq (1-\delta) \frac{\alpha}{\beta_c} r^\alpha, \quad |J'(r)|\leq (1+\delta) r^{-d-\alpha-1}, \quad \text{ and } \quad |J'(r)||B_r| \geq (1-\delta) r^{-\alpha-1} 
\end{equation}
for every $r\geq r_0$; the last two constraints are possible to satisfy by our normalization convention \eqref{eq:normalization_conventions} while the first is possible to satisfy by \eqref{eq:mean_lower_bound}.
Suppose that $r\geq r_0$ and that $\E_{\beta_c,r}|K|^2\leq \delta r^{\alpha+ (d+\alpha)/2}$.
Applying Markov's inequality to the size-biased measure $\hat \E_{\beta_c,r}$, which has mean
\[
  \hat \E_{\beta_c,r}|K| = \frac{\E_{\beta_c,r}|K|^2}{\E_{\beta_c,r}|K|} \leq \delta (1-\delta)^{-1} \frac{\beta_c}{\alpha} r^{(d+\alpha)/2}
\] we deduce using \eqref{eq:sCL_delta_def2} that 
\begin{multline}
\label{eq:sCL_truncated_susceptibility_contradiction}
  \E_{\beta_c,r}\left[|K| \mathbbm{1}(|K|\leq \delta^{1/2} r^{(d+\alpha)/2}) \right]  = \E_{\beta_c,r}|K| \hat \P_{\beta_c,r}(|K|\leq \delta^{1/2} r^{(d+\alpha)/2})
  \\\geq (1-\delta)  \frac{\alpha}{\beta_c} \left(1-\delta^{1/2} (1-\delta)^{-1} \frac{\beta_c}{\alpha}\right) r^\alpha \geq (1-\delta^{1/4}) \frac{\alpha}{\beta_c} r^\alpha.
\end{multline}
Using Russo's formula, we can write
\begin{align*}
\frac{d}{dr}\E_r|K|^2 &= \beta_c |J'(r)| \left(
\E_r\left[|K| \sum_{y\in B_r} \mathbbm{1}(y\notin K) |K_y|^{2} \right]
+
2\E_r\left[|K|^{2} \sum_{y\in B_r} \mathbbm{1}(y\notin K) |K_y| \right]\right)\nonumber\\
&=3\beta_c |J'(r)| \E_r\left[|K|^{2} \sum_{y\in B_r} \mathbbm{1}(y\notin K) |K_y| \right] 
\nonumber\\&\geq 
3\beta_c |J'(r)| \E_r\left[|K|^{2} \sum_{y\in B_r} \mathbbm{1}(y\notin K) \min\{|K_y|,\delta^{1/2}r^{(d+\alpha)/2}\} \right]
\end{align*}
where the second equality follows by the mass-transport principle. Applying \eqref{eq:BK_disjoint_clusters_covariance2}, we deduce that
\begin{align*}
  \frac{d}{dr} \E_{\beta_c,r}|K|^2 &\geq 
  3 \beta_c|J'(r)| |B_r| \E_{\beta_c,r}\min\{|K|,\delta^{1/2} r^{(d+\alpha)/2}\} \E_{\beta_c,r}|K|^2 
  \\&\hspace{4cm}-3 \beta_c|J'(r)|\E_{\beta_c,r}\left[\sum_{y\in B_r} \mathbbm{1}(y\in K)|K|^2  \min\{|K|,\delta^{1/2} r^{(d+\alpha)/2}\}\right]
  \\
  &\geq 
  3 \beta_c|J'(r)| |B_r| \E_{\beta_c,r}\min\{|K|,\delta^{1/2} r^{(d+\alpha)/2}\} \E_{\beta_c,r}|K|^2 
  \\&\hspace{6.5cm}-3 \beta_c|J'(r)|\E_{\beta_c,r}\left[|K|^3  \min\{|K|,\delta^{1/2} r^{(d+\alpha)/2}\}\right]
  \\&\geq \left[3 \alpha (1-\delta^{1/4})(1-\delta)  - 3 \beta_c (1+\delta) \delta^{1/2} C \right] r^{-1} \E_{\beta_c,r}|K|^2 \geq \left[\alpha + \frac{d+\alpha}{2}\right] r^{-1} \E_{\beta_c,r}|K|^2
\end{align*}
as required, where we applied \eqref{eq:sCL_third_moment_upper} and \eqref{eq:sCL_truncated_susceptibility_contradiction} 
used the conditions placed on $\delta$ and $r_0$ in \eqref{eq:sCL_delta_def1} and \eqref{eq:sCL_r_0_def} in the last line.
\end{proof}

The estimates proven in this subsection already yield the estimate $M_r\asymp M_r^*\asymp r^{(d+\alpha)/2}$ of \cref{thm:max_cluster_size_LD} under the stronger assumption that \eqref{CL} holds. The proof will apply the \emph{universal tightness theorem} of \cite[Theorem 2.2]{hutchcroft2020power} via the following inequality, which is a special case of \cref{I-cor:universal_tightness_moments}: 
\begin{equation}
\label{eq:universal_tightness_moments_restate}
  \E_{\beta_c,r}|K \cap B_R|^{p+q} \preceq_{p,q} M_{R}^q\E_{\beta_c,r}|K \cap B_R|^p
\end{equation}
for every $R\geq r>0$ and $p,q\geq 1$. (In fact this is true for all $\beta$, as is clear from the much more general statement given in \cref{I-cor:universal_tightness_moments}; we have stated it this way to avoid introducing notation for the relevant analogue of $M_r$.)

\begin{proof}[Proof of \cref{thm:max_cluster_size_LD} assuming \eqref{CL}]
The upper bound $M_r\leq M_r^* \preceq r^{(d+\alpha)/2}$ follows from the main result of \cite{hutchcroft2022sharp} as discussed in \eqref{eq:M_r_upper_restate}, so that it suffices to prove the lower bound. Using \cref{thm:CL_volume_moments,lem:CL_and_wCL} and applying \eqref{eq:universal_tightness_moments_restate} with $p=q=1$ we obtain that $\E_{\beta_c,r}|K|\asymp r^\alpha$ and that there exists a constant $C$ such that
\[
  r^{\alpha+\frac{d+\alpha}{2}}\preceq \E_{\beta_c,r}|K|^2 \preceq \E_{\beta_c,r}|K \cap B_{Cr}|^2 \preceq M_{Cr} \E_{\beta_c,r}|K|,
\]
which rearranges to give the claimed lower bound on $M_r$.
\end{proof}

Concluding from the weaker assumption \eqref{wCL} will require a different argument and is done in the next two subsections. These sections will also introduce important techniques that will be used in \cref{sec:negligibility_of_mesoscopic_clusters_and_hyperscaling}.

\subsection{The hydrodynamic condition does not hold}
\label{sub:the_hydrodynamic_condition_does_not_hold}

The goal of this subsection is to prove the following proposition.

\begin{prop}
\label{prop:low_dim_not_hydro}
If $d<3\alpha$ and \eqref{wCL} holds then the hydrodynamic condition \eqref{Hydro} does \emph{not} hold. 
\end{prop}


As discussed at the beginning of the section, we will prove this proposition by assuming for contradiction that the hydrodynamic condition \emph{does} hold, using this to implement the high-dimensional analysis of \cref{I-sec:analysis_of_moments}, and then using this analysis to deduce false consequences regarding the growth of the second moment $\E_{\beta_c,r}|K|^2$ under the hypotheses of \cref{prop:low_dim_not_hydro}. More precisely, the high-dimensional analysis we will perform under the fictitious assumption that \eqref{wCL} and \eqref{Hydro} both hold will lead to second moment asymptotics of the form
\[
  \E_{\beta_c,r}|K|^2 = r^{3\alpha \pm o(1)}.
\]
This is the correct scaling of the second moment in the regimes treated in \cite{LRPpaper1,LRPpaper3} but for $d<3\alpha$ is inconsistent with the upper bound
\[
  \E_{\beta_c,r}|K|^2 \preceq 
  r^{\alpha+\frac{d+\alpha}{2}},
\]
which follows from the results of \cite{hutchcroft2022sharp} under the assumption \eqref{wCL} as we explain in detail below.

\medskip

We begin by recalling several key steps from the analysis of \cref{I-sec:analysis_of_moments}. The starting point of this analysis is the following formula for $r$-derivatives of $p$th moments (restated from \cref{I-lem:moment_derivative}), which is a consequence of Russo's formula:
%
\begin{align}
\frac{\partial}{\partial r}\E_{\beta,r}|K|^p &= \beta |J'(r)| \sum_{\ell=0}^{p-1} \binom{p}{\ell}\E_{\beta,r}\left[|K|^{\ell+1} \sum_{y\in B_r} \mathbbm{1}(y\notin K) |K_y|^{p-\ell} \right]
\label{eq:pth_moment_derivative}
\end{align}
for every integer $p\geq 1$, every $0\leq \beta \leq \beta_c$ and every $r>0$.
%
%
%
%
%
Fixing $\beta=\beta_c$, we can rewrite the $p=1,2$ cases of the equality \eqref{eq:pth_moment_derivative} as in \eqref{I-eq:first_moment_ODE_error_def} and \eqref{I-eq:second_moment_ODE_error_def} to obtain that
\begin{equation}
\label{eq:E1r_def_ODE}
\frac{d}{dr}\E_{\beta_c,r}|K| 
= (1-\cE_{1,r}) \beta_c |J'(r)| |B_r| (\E_{\beta_c,r}|K|)^2
\end{equation}
and
\begin{align}
\frac{d}{dr}\E_{\beta_c,r}|K|^2 &= 
 3(1-\cE_{2,r}) \beta_c |J'(r)| |B_r| \E_{\beta_c,r}|K|\E_{\beta_c,r}|K|^2
\label{eq:E2r_def_ODE}
\end{align}
where the error terms $\cE_{1,r}$ and $\cE_{2,r}$ are defined by
\begin{align*}
\cE_{1,r} &:= 
\frac{|B_r|(\E_{\beta_c,r}|K|)^2 - \E_{\beta_c,r} \left[|K| \sum_{y\in B_r} \mathbbm{1}(y\notin K) |K_y|\right]}{|B_r| (\E_{\beta_c,r}|K|)^2}
\intertext{and}
\cE_{2,r} &:= 
\frac{|B_r|\E_{\beta_c,r}|K|\E_{\beta_c,r}|K|^2 - \E_{\beta_c,r} \left[|K| \sum_{y\in B_r} \mathbbm{1}(y\notin K) |K_y|^2\right]}{|B_r| \E_{\beta_c,r}|K|\E_{\beta_c,r}|K|^2}.
\end{align*}
As in \eqref{I-eq:barE11_def_ODE}, we also define the errors $\cE_{0,r}$ and $\overline{\cE}_{1,r}$ by
\[
(1-\cE_{0,r}):=\frac{|J'(r)||B_r|}{r^{-\alpha-1}} \text{ and } (1-\overline{\cE}_{1,r}) = (1-\cE_{1,r})\frac{|J'(r)||B_r|}{r^{-\alpha-1}} = (1-\cE_{1,r})(1-\cE_{0,r}),
\]
so that
\begin{equation}
\label{eq:barE11_def_ODE}
\frac{d}{dr}\E_{\beta_c,r}|K| 
= (1-\overline{\cE}_{1,r}) \beta_c r^{-\alpha-1} (\E_{\beta_c,r}|K|)^2.
\end{equation}
It is noted in \cref{I-lem:E1_1} that the error $\cE_{1,r}$ satisfies the bounds
\begin{equation}
  0 \leq \cE_{1,r} \leq \frac{\E_{\beta_c,r}|K|^2|K\cap B_r|}{|B_r|(\E_{\beta_c,r}|K|)^2}.
\label{eq:E1r_bounds}
\end{equation}
Since $\cE_{0,r}$ converges to zero by our normalization assumptions \eqref{eq:normalization_conventions}, it follows that $\liminf_{r\to \infty}\overline{\cE}_{1,r} \geq 0$ and hence by an elementary ODE analysis (encapsulated in \cref{I-lem:f'=f^2}) that
\begin{equation}
\E_{\beta_c,r}|K|\geq (1-o(1))\frac{\alpha}{\beta_c} r^{\alpha}
\end{equation}
as $r\to \infty$ (\cref{I-cor:mean_lower_bound}) as we stated earlier in \eqref{eq:mean_lower_bound}.
The following lemma gives an upper bound matching this estimate to first order when \eqref{wCL} and \eqref{Hydro} both hold; it should be compared with \cref{I-prop:first_moment} which gives a similar estimate in the effectively high- and critical-dimensional cases without assuming \eqref{wCL}.

\begin{lemma}
\label{lem:first_moment_CL}
If \eqref{wCL} holds then 
$\E_{\beta_c,r}|K|\asymp r^\alpha$
for all $r\geq 1$. 
If the hydrodynamic condition \eqref{Hydro} also holds then $\cE_{1,r}\to 0$ as $r\to\infty$ and
\begin{equation}
  \E_{\beta_c,r}|K| \sim \frac{\alpha}{\beta_c}r^\alpha
\label{eq:first_moment_Hydro}
\end{equation}
as $r\to \infty$.
\end{lemma}

\begin{proof}[Proof of \cref{lem:first_moment_CL}]
The up-to-constants estimate $\E_{\beta_c,r}|K|\asymp r^\alpha$ follows by essentially the same argument as \cref{lem:first_moment_CLb}.
 We now turn to the analysis of $\cE_{1,r}$ under the hydrodynamic condition.
It follows from \eqref{wCL} that there exists a constant $C$ such that
\[
\E_{\beta_c,r} \left[|K|^2 |K\cap B_r|\right] \leq \E_{\beta_c,r} \left[|K|^3\right] \leq 2 \E_{\beta_c,r} \left[|K\cap B_{Cr}|^3\right].
\]
 Applying the universal tightness theorem via the inequality \eqref{eq:universal_tightness_moments_restate}, it follows that
\[
\E_{\beta_c,r} \left[|K|^2 |K\cap B_r|\right]  \preceq \E_{\beta_c,r} \left[|K\cap B_{Cr}|\right] M_{Cr}^2 = o(r^{d+2\alpha}),
\]
where the final estimate follows by a second application of \eqref{eq:Sak_upper_restate} and the hydrodynamic condition. The claim that $\cE_{1,r}\to 0$ as $r\to \infty$ follows from this together with \eqref{eq:E1r_bounds} and the lower bound \eqref{eq:mean_lower_bound}. The convergence $\cE_{1,r}$ of $\cE_{1,r}$ and hence of $\overline{\cE}_{1,r}$ to zero  implies that
\[
  \frac{d}{dr} \E_{\beta_c,r} |K| \sim \beta_c r^{-\alpha-1} (\E_{\beta_c,r} |K|)^2
\]
as $r\to \infty$, and the asymptotic estimate \eqref{eq:first_moment_Hydro} follows by an elementary ODE analysis as encapsulated by \cref{I-lem:f'=f^2} (see also \eqref{I-eq:first_moment_E1_formula}). 
\end{proof}


\cref{lem:first_moment_CL} implies the following corollary by the same reasoning used to deduce \cref{cor:large_alpha_not_CL} from \cref{lem:first_moment_CLb}.

\begin{corollary}
\label{cor:large_alpha_not_wCL}
If $\alpha>d$ then \eqref{wCL} does not hold. If $\alpha=d$ and \eqref{wCL} holds then the model has a discontinuous phase transition in the sense that there exists an infinite cluster at $\beta_c$.
\end{corollary}

We next analyze the second moment. Recall that a measurable function $f:(0,\infty)\to(0,\infty)$ is said to be \textbf{regularly varying} of index $a\in \R$ if $f(\lambda x)/f(x)\to \lambda^a$ as $x\to\infty$ for each fixed $\lambda>0$, which implies that $f(x)=x^{a\pm o(1)}$ as $x\to \infty$.

\begin{lemma}
\label{lem:CL_Hydro_second_moment}
If \eqref{wCL} and \eqref{Hydro} both hold then $\E_{\beta_c,r}|K|^2$ is regularly varying of index $3\alpha$.
\end{lemma}


\begin{proof}[Proof of \cref{lem:CL_Hydro_second_moment}]
It suffices to prove that $\cE_{2,r}\to 0$ as $r\to \infty$. Indeed, once this is established it will follow from \eqref{eq:E2r_def_ODE} that
\[
  \frac{d}{dr}\E_{\beta_c,r}|K|^2 \sim 3 \beta_c|J'(r)| |B_r| \E_{\beta_c,r} |K| \E_{\beta_c,r}|K|^2 \sim 3 \alpha r^{-1} \E_{\beta_c,r}|K|^2,
\]
where in the second asymptotic equality we applied both our normalization conventions \eqref{eq:normalization_conventions} and the asymptotic equality $ \E_{\beta_c,r} |K| \sim \frac{\alpha}{\beta_c} r^\alpha$ established in \cref{lem:first_moment_CL}; the asymptotic relationship $f'\sim 3\alpha r^{-1} f$ implies regular variation of index $3\alpha$ for positive functions as explained in \cref{I-lem:ODE_self_referential}.
We will prove that $\cE_{2,r}\to 0$ using the inequality
\begin{equation}
\label{eq:E2r_bounds}
0\leq \cE_{2,r} \leq \frac{\E_{\beta_c,r}|K|^3|K\cap B_r|
}{|B_r| \E_{\beta_c,r}|K|\E_{\beta_c,r}|K|^2},
\end{equation}
established in \eqref{I-eq:E2_bound}. To bound the numerator, we use \eqref{wCL} to deduce the existence of a constant $C$ such that
\[
  \E_{\beta_c,r}|K|^3|K\cap B_r| \leq \E_{\beta_c,r}|K|^4 \leq 2 \E_{\beta_c,r}|K\cap B_{Cr}|^4
\]
and apply \eqref{eq:universal_tightness_moments_restate} to obtain that
\[
  \E_{\beta_c,r}|K\cap B_{Cr}|^4 \preceq M_{Cr}^2 \E_{\beta_c,r}|K\cap B_{Cr}|^2 \leq
  M_{Cr}^2 \E_{\beta_c,r}|K|^2.
\]
Substituting this estimate into \eqref{eq:E2r_bounds} we obtain that
\[
\cE_{2,r} \preceq \frac{M_{Cr}^2}{|B_r| \E_{\beta_c,r}|K|} \asymp r^{-d-\alpha} M_{Cr}^2
\]
from which the claim $\cE_{2,r} \to 0$ follows by a direct application of the hydrodynamic condition.
\end{proof}

\begin{proof}[Proof of \cref{prop:low_dim_not_hydro}]
It follows from \eqref{wCL} and the universal tightness theorem applied as in \eqref{eq:universal_tightness_moments_restate} that there exists a constant $C$ such that
\begin{equation}
\label{eq:CL_second_moment_upper}
  \E_{\beta_c,r}|K|^2 \leq 2 \E_{\beta_c,r}|K \cap B_{Cr}|^2 \preceq M_{Cr} \E_{\beta_c,r}|K \cap B_{Cr}| 
  \preceq r^{\alpha+\frac{d+\alpha}{2}},
\end{equation}
where the final inequality follows from \eqref{eq:Sak_upper_restate} and \eqref{eq:M_r_upper_restate}. On the other hand, it follows from \cref{lem:CL_Hydro_second_moment} that if \eqref{wCL} and \eqref{Hydro} both hold then $\E_{\beta_c,r}|K|^2$ is regularly varying of index $3\alpha$ and hence that $\E_{\beta_c,r}|K|^2 = r^{3\alpha \pm o(1)}$ as $r\to \infty$. These two estimates are inconsistent with each other when $d<3\alpha$ and $\alpha + \frac{d+\alpha}{2}<3\alpha$, so that \eqref{wCL} and \eqref{Hydro} cannot both hold when $d<3\alpha$.
\end{proof}


\subsection{Estimates at every scale}

\label{subsec:low_dim_max_cluster}

We saw in \cref{prop:low_dim_not_hydro} that if $d<3\alpha$ and \eqref{wCL} holds then the hydrodynamic condition does \emph{not} hold, so that $M_r$ is of order $r^{(d+\alpha)/2}$ for an unbounded set of $r$. In this section we complete the proof of \cref{thm:max_cluster_size_LD} by improving this to a  bound of the same order at \emph{every} scale.
%
%
The proof will use the inequality
\begin{equation}
\label{eq:moments_bounded_below_by_M}
\E_{\beta_c,r}|K|^p \succeq_p \frac{M_{r}^{p+1}}{r^d}
\end{equation}
which holds for every $r\geq 1$ and integer $p\geq 1$ as a special case of \cref{I-lem:moments_bounded_below_by_M}. (This inequality is an immediate consequence of the definition of $M_r$ and translation invariance.)

\begin{proof}[Proof of \cref{thm:max_cluster_size_LD}]
To lighten notation we write $\E_r=\E_{\beta_c,r}$.
It suffices to prove the lower bound $M_r \succeq r^{(d+\alpha)/2}$, with the upper bound $M_r^* \preceq r^{(d+\alpha)/2}$ (stated in \eqref{eq:M_r_upper_restate}) following from the results of \cite{hutchcroft2022sharp} as explained in \cref{I-sub:previous_results_on_long_range_percolation}. We have as in the proof of \cref{lem:CL_Hydro_second_moment} that
\begin{equation}\label{eq:diff_eq_E3_oops}
\frac{d}{dr} \log \E_r |K|^2  = (1-\cE_{2,r}) 3\beta_c r^{-d-\alpha-1}|B_r|\E_r|K|
\end{equation}
where 
\[
0\leq \cE_{2,r} 
\leq \frac{\E_{r}|K|^3|K\cap B_r|
}{|B_r| \E_{r}|K|\E_{r}|K|^2}.
\]
Moreover, the lower bound $\E_r|K| \geq (1-o(1)) \frac{\alpha}{\beta_c} r^\alpha$ of \eqref{eq:mean_lower_bound} (restated from \cref{I-cor:mean_lower_bound}) and the asymptotic formula $|B_r|\sim r^d$ (which comes from our normalization assumptions \eqref{eq:normalization_conventions}) lets us bound
\begin{equation}\label{eq:diff_eq_E3_oops}
\frac{d}{dr} \log \E_r |K|^2  \geq  (1\pm o(1))(1-\cE_{2,r}) 3\alpha r^{-1}.
\end{equation}
The assumption \eqref{wCL} implies that there exists a constant $C\geq 1$ such that
\[
\cE_{2,r} 
\leq \frac{\E_{r}|K|^4
}{|B_r| \E_{r}|K|\E_{r}|K|^2} \leq \frac{2\E_{r}|K \cap B_{Cr}|^4
}{|B_r| \E_{r}|K|\E_{r}|K \cap B_{Cr}|^2} \preceq \frac{M_{Cr}^2}{r^{d+\alpha}},
\]
where we used the lower bound $\E_r |K| \succeq r^\alpha$ from \eqref{eq:mean_lower_bound} and the universal tightness theorem as in \eqref{eq:universal_tightness_moments_restate} to bound $\E_{r}|K \cap B_{Cr}|^4\preceq M_{Cr}^2\E_{r}|K \cap B_{Cr}|^2$ in the last inequality. By increasing $C$ if necessary, we may assume without loss of generality that
\[
\cE_{2,r} 
\leq C\frac{M_{Cr}^2}{r^{d+\alpha}}
\]
for every $r\geq 1$.
Let $0<\delta\leq 1$ be
 maximal such that
\[
\limsup_{r\to \infty} r^{-(d+\alpha)/2} M_r \geq 2 \delta, \qquad C^{1+d+\alpha}\delta \leq 1, \qquad \text { and } \qquad 3(1-\delta)^2\alpha>\alpha+(d+\alpha)/2
\]
where the first condition is possible to satisfy by \cref{prop:low_dim_not_hydro} and the last condition is possible to satisfy since $d<3\alpha$, and let $r_0\geq 1$ be sufficiently large that $|B_r|\geq (1-\delta) r^d$. It suffices to prove that there exists a constant $\lambda_0$ such that there do not exist any intervals of the form $[s,\lambda_0s]$ with $s\geq r_0$ such that $M_r \leq \delta r^{(d+\alpha)/2}$ for every $r\in [s,\lambda_0 s]$.
Indeed, given this claim it follows that for each $r \geq \lambda_0r_0$ there exists $\lambda_0^{-1}r \leq s \leq r$ such that 
\[
M_r \geq M_s \geq \delta s^{(d+\alpha)/2} \geq \delta \lambda_0^{-(d+\alpha)/2} r^{(d+\alpha)/2} 
\]
as desired. Since $\limsup_{r\to \infty} r^{-(d+\alpha)/2} M_r > \delta$ and $M_r$ is non-decreasing in $r$, it suffices to prove that there exists a constant $\lambda_0$ such that if $R \geq r\geq r_0$ satisfy $M_{r} \geq \delta r^{(d+\alpha)/2}$ and $M_s \leq \delta s^{(d+\alpha)/2}$ for every $r < s \leq R$ then $R/r \leq \lambda_0$, that is, to bound the log-scale length of a \emph{maximal} interval of bad scales.
 Suppose that $R \geq r \geq r_0$ is such a maximal interval of bad scales with $R\geq Cr$. (If no such intervals exist then there is nothing to prove.) The condition that $C^{1+d+\alpha}\delta \leq 1$ ensures that 
\[\cE_{2,s}\leq Cs^{-d-\alpha}M^2_{Cs} \leq \delta^2 C^{1+d+\alpha} \leq \delta \] for every $r \leq s \leq R/C$. Writing $R'=R/C$ and integrating the estimate \eqref{eq:diff_eq_E3_oops} from $r$ to $R'$ yields that
\[
\frac{\E_{R'}|K|^2}{\E_r|K|^2} \geq \exp\left[ 3\alpha (1-\delta)^2 \int_r^{R'} \frac{1}{s} \dif s\right] = \left(\frac{R'}{r}\right)^{3(1-\delta)^2\alpha}.
\]
On the other hand, since $M_r \geq \delta r^{(d+\alpha)/2}$ we can use \eqref{eq:moments_bounded_below_by_M} to lower bound
$\E_r|K|^2 \succeq r^{-d}M_r^3 = r^{\alpha +(d+\alpha)/2}$,
so that 
\[
\E_{R'}|K|^2 \succeq \left(\frac{R'}{r}\right)^{3(1-\delta)^2\alpha} r^{\alpha +(d+\alpha)/2}.
\]
Comparing this with the upper bound
$\E_{R'}|K|^2 \preceq \E_{R'}|K\cap B_{C'R'}|^2 \preceq (R')^{\alpha + (d+\alpha)/2}$
which follows from \eqref{wCL} and \eqref{eq:CL_second_moment_upper}, yields that 
\[
\left(\frac{R'}{r}\right)^{3(1-\delta)^2\alpha} \preceq \left(\frac{R'}{r}\right)^{\alpha + (d+\alpha)/2}.
\]
Since we chose $\delta$ so that $3(1-\delta)^2\alpha>\alpha+(d+\alpha)/2$, we can rearrange this inequality to deduce that $R/r \preceq R'/r$ is bounded by a constant, concluding the proof. 
\end{proof}

Before moving on, let us note the following corollary of \cref{thm:max_cluster_size_LD}, which is a weaker analogue of \cref{thm:CL_volume_moments,lem:CL_and_wCL} holding under the weaker assumption \eqref{wCL}.

\begin{corollary}
\label{cor:low_dim_moments}
If $d<3\alpha$ and \eqref{wCL} holds then
\[
 \E_{\beta_c}|K \cap B_r|^p \asymp_p \E_{\beta_c,r}|K \cap B_r|^p \asymp_p r^{\alpha + \frac{d+\alpha}{2}(p-1)}
\]
as $r\to \infty$ for each fixed integer $p\geq 1$. Moreover, the bound
\[\E_{\beta_c,r}|K|^p \asymp \E_{\beta_c,r}|K \cap B_r|^p \asymp r^{\alpha + \frac{d+\alpha}{2}(p-1)}\]
holds for $p=1,2,3,4$.
\end{corollary}

\begin{proof}[Proof of \cref{cor:low_dim_moments}]
The upper bounds for $\E_{\beta_c}|K \cap B_r|^p$ follow from \eqref{eq:Sak_upper_restate}, \eqref{eq:M_r_upper_restate}, and the universal tightness theorem (applied as in \eqref{eq:universal_tightness_moments_restate}), which yield that
\[
 \E_{\beta_c,r}|K \cap B_r|^p \leq  \E_{\beta_c}|K \cap B_r|^p \preceq_p \E_{\beta_c}|K\cap B_r| M_r^{p-1} \preceq_p r^\alpha r^{(p-1)(d+\alpha)/2}.
\]
This bound does not require the assumptions that $d<3\alpha$ or that \eqref{wCL} holds. The corresponding upper bound on $\E_{\beta_c,r}|K|^p$ for $p=1,2,3,4$ follows from this bound and the the definition of \eqref{wCL}. The lower bound follows from \cref{thm:max_cluster_size_LD} and \eqref{eq:moments_bounded_below_by_M}, which yield that
\[
   \E_{\beta_c,r}|K \cap B_r|^p \succeq_p \frac{1}{r^d}M_{r/2}^{p+1} \succeq_p r^{(p+1)(d+\alpha)/2-d} = r^{\alpha+(p-1)(d+\alpha)/2}
\]
as required.
\end{proof}

\section{Hyperscaling}
\label{sec:negligibility_of_mesoscopic_clusters_and_hyperscaling}

In this section we complete the proofs of \cref{thm:main_low_dim,thm:k_point_S}, along with an extension of \cref{thm:main_low_dim} and the lower bound of \cref{thm:k_point_S} requiring only the weaker assumption \eqref{wCL}. We begin by explaining how \cref{thm:main_low_dim} relates to heuristic hyperscaling theory in \cref{subsec:hyperscaling}, in which we also state a related theorem (\cref{thm:hyperscaling}) verifying the basic assumptions of this heuristic theory when $d<3\alpha$ and \eqref{wCL} holds. In \cref{subsec:negligibility_of_mesoscopic_clusters} we then build on the techniques of \cref{sub:the_hydrodynamic_condition_does_not_hold,subsec:low_dim_max_cluster} to prove a proposition concerning the ``neglibility of mesoscopic clusters'' which we then show implies \cref{thm:main_low_dim,thm:hyperscaling}. Finally, in \cref{subsec:_k_point_function_hyperscaling} we complete the proof of \cref{thm:k_point_S} by proving $k$-point function lower bounds matching the upper bounds proven in \cref{subsec:higher_gladkov_inequalities_and_the_k_point_function}. We also prove a similar $k$-point function estimate for nearest-neighbour bond percolation on $\Z^2$, supporting the claim that the $k$-point estimate of \cref{thm:k_point_S} should be thought of as a hyperscaling relation rather than a special feature of long-range percolation.

\subsection{Overview of heuristic hyperscaling theory}
\label{subsec:hyperscaling}

In this section we explain how our results relate to the heuristic theory of \emph{hyperscaling relations}:
identities relating various critical exponents that are predicted to hold below but not above the upper critical dimension. 
 (The prefix \emph{hyper} distinguishes these relations from the other scaling relations, such as those we study in \cref{sec:subcritical}, which are predicted to hold in \emph{all} dimensions \cite[Chapter 9]{grimmett2010percolation}.) We first recall how heuristic hyperscaling theory leads to the prediction \eqref{eq:Sak_delta} and then explain how the results of our paper verify not only the resulting critical exponent formulae but the entire picture underlying their heuristic derivation.

\medskip

Roughly speaking, the hyperscaling relations can be thought of as encoding the fact that, in low effective dimension, there are only $O(1)$ ``macroscopic'' clusters on each scale that carry all the interesting geometric features of the model, whereas in high dimensions there is instead a ``soup'' of many large, weakly interacting clusters that are all equally important to the behaviour of the model. The effectively high- and critical-dimensional parts of this picture are explored in detail for long-range percolation in the other two papers of this series, where it is shown (in a certain precise sense, see \cref{I-thm:superprocess_main,III-cor:superprocess_main_CD} and the discussion around \eqref{I-eq:number_of_large_clusters} and \eqref{III-eq:number_of_large_clusters}) that the number of ``typical large clusters'' $N(r)$ on scale $r$ grows as a power of $r$ in high effective dimension and as a power of $\log r$ when $d=3\alpha \leq 6$; see \Cref{table:large_clusters} for a summary and \cite{MR1431856} for further discussion and related results regarding finite-range percolation in high dimensions.

\medskip

Suppose that we are able to define a ``fractal dimension'' exponent $d_f$, with the interpretation that a ``typical large critical cluster on scale $r$'' has size growing roughly like $r^{d_f}$; below the upper-critical dimension this should be the same as saying that the \emph{largest} critical cluster on scale $r$ has this size. (Above the upper critical dimension, the largest cluster should be larger than a typical large cluster by a logarithmic factor \cite{MR1431856}.) An important example of a hyperscaling relation is
\begin{equation}
\label{eq:hyperscaling_eta_df}
2-\eta=2d_f-d
\end{equation}
which comes from the heuristic derivation
\[
r^{2-\eta} \approx \sum_{x\in B_r}\P_{\beta_c}(x\leftrightarrow y) = \E_{\beta_c}|K \cap B_r| 
\approx \P_{\beta_c}(0 \text{ in the largest cluster on scale $r$}) r^{d_f}
\approx r^{2d_f-d},
\]
where the central ``equality'' supposes that the expectation of $|K\cap B_r|$ is of the same order as the contribution from the largest cluster on scale $r$. 
(Here we reserve the symbol $\approx$ for use in heuristic calculations where its precise meaning is left unspecified.)
Combining this with Sak's prediction one obtains the prediction that
\begin{equation}
\label{eq:df_hyperscaling}
  d_f = \frac{d+\alpha}{2}
\end{equation}
throughout the effectively long-range, low-dimensional regime. Writing $\eta(r)$ for the (imprecisely defined) probability that the origin belongs to a ``typical large cluster" on scale $r$, we should have moreover that
\begin{equation}
\label{eq:pth_moment_in_box_scaling}
  \E_{\beta_c}|K\cap B_r|^p \approx \eta(r) r^{d_fp} \approx r^{(p-1)d_f} \E_{\beta_c}|K\cap B_r|
\end{equation}
for every $p\geq 1$. (This leads to the \emph{definition} of the exponent $d_f$ via $\E_{\beta_c}|K\cap B_r|^2/\E_{\beta_c}|K\cap B_r| =r^{d_f\pm o(1)}$ that we gave in the introduction; this is a reasonable definition both in high dimensions and low dimensions, with \eqref{eq:pth_moment_in_box_scaling} being a scaling relation rather than a hyperscaling relation.) 

\medskip

One can also use heuristic hyperscaling theory to express the volume tail critical exponent $\delta$, 
defined by
\[
  \P_{\beta_c}(|K|\geq n) \approx n^{-1/\delta}
\]
in terms of the fractal dimension $d_f$ for models of low effective dimension. To do this, the key heuristic is that the best way for the cluster of the origin to be large is for it to belong to a typical large cluster on some scale $r$, which should satisfy $r^{d_f}\approx n$ for the cluster to have size at least $n$. Thus, in low effective dimension one should have that
\[
 n^{-1/\delta} \approx \P_{\beta_c}(|K|\geq n) \approx \P_{\beta_c}(0 \text{ in the largest cluster on scale $n^{1/d_f}$}) \approx n^{1-d/d_f}.
\]
Together with \eqref{eq:hyperscaling_eta_df}, this yields the (heuristic) hyperscaling relations
\begin{equation}
\label{eq:delta_hyperscaling}
  \delta = \frac{d_f}{d-d_f} = \frac{d+2-\eta}{d-2+\eta}
\end{equation}
in low effective dimension. Combining this with Sak's prediction and the prediction that $\delta=2$ in the high and critical effective dimensional regime one obtains the prediction \eqref{eq:Sak_delta}.


\begin{table}[t]
\centering
\renewcommand{\arraystretch}{1.6} 
\begin{tabular}{|c|c|c|}
\hline
 & \# Large Clusters & Typical size \\ 
\hline
LR HD & $r^{d-3\alpha}$ & $r^{2\alpha}$  \\ 
\hline
mSR HD & $r^{d-6}(\log r)^3$ & $r^4 (\log r)^{-2}$ \\ 
\hline
SR HD & $r^{d-6}$ & $r^4$  \\ 
\hline
LR CD & $\log r$ & $r^{2\alpha} (\log r)^{-1/2}$  \\ 
\hline
LR LD & $O(1)$ & $r^{(d+\alpha)/2}$ \\ 
\hline
\end{tabular}
\caption{Summary of results concerning the ``typical size of a large cluster'' and the ``number of typical large clusters'' in a box of radius $r$ for long-range percolation on $\Z^d$. Precise statements are given in \cref{I-thm:superprocess_main}, \cref{thm:hyperscaling} and \cref{III-cor:superprocess_main_CD}. (See also \eqref{I-eq:number_of_large_clusters} and \eqref{III-eq:number_of_large_clusters}.) In the effectively long-range low-dimensional case treated in this paper, the ``$O(1)$'' means more precisely that for each $\eps>0$ the number of clusters in the box with size at least $\eps r^{(d+\alpha)/2}$ is tight as $r\to \infty$.}
\label{table:large_clusters}
\end{table}

\medskip

The hyperscaling relations \eqref{eq:hyperscaling_eta_df} and \eqref{eq:delta_hyperscaling} (along with several other scaling and hyperscaling relations, some of which were discussed in \cref{subsec:subcritical_intro}) were verified for nearest-neighbour percolation on planar lattices by Kesten \cite{MR879034}, and played an important role in the eventual computation of critical exponents for site percolation on the triangular lattice \cite{smirnov2001critical}.
Beyond two dimensions, the hyperscaling relations have been verified for finite-range models conditional on appropriate ``hyperscaling postulates'' that remain unproven in dimensions $d=3,4,5$ \cite{MR1716769}. One-sided \emph{inequality} versions of the hyperscaling relations, which hold for both long-range and finite-range models in every dimension, were established in \cite{hutchcroft2020power} as a corollary of the \emph{universal tightness theorem} \cite[Theorem 2.2]{hutchcroft2020power} and play an important role in this paper as well as in the previous works \cite{hutchcroft2022sharp,hutchcroft2024pointwise,baumler2022isoperimetric}.

\begin{remark}
The standard approach to (heuristic) scaling and hyperscaling theory for \emph{finite-range} percolation places a strong emphasis on \emph{one-arm events}, with a basic hyperscaling relation being $\P_{p_c}(x\leftrightarrow y) \asymp \P_{p_c}(0\leftrightarrow \partial [-r,r]^d)^2$ for $r=\|x-y\|/4$. This approach to (hyper)scaling does not work well for long-range models, since arm events can be dominated by the presence of ``hairs'': long edges that connect the origin to a large distance but not too a large number of vertices at this large distance. (This is related to the ``instantaneous propagation'' property of L\'evy superprocesses \cite{perkins1990polar}.) See \cite{MR3418547} for a detailed analysis of this phenomenon for long-range percolation in high effective dimensions. 
\end{remark}

As mentioned above, the proof of our main theorems establish not only the asymptotic estimates claimed therein but also verifies the entire picture underlying the heuristic derivation of the hyperscaling relations we have just sketched. This is encapsulated formally in the following theorem, which shows that the largest cluster on scale $r$ has size of order $r^{(d+\alpha)/2}$, that clusters significantly smaller than this do not contribute significantly to cluster volume moments on scale $r$, and that the volume tail probability $\P_{\beta_c}(|K|\geq n)$ is dominated by the contribution from belonging to a typical large cluster on a scale of order $r=n^{2/(d+\alpha)}$. 


\begin{theorem}[Hyperscaling]
\label{thm:hyperscaling}
If $d<3\alpha$ and \eqref{wCL} holds then the following hold:
\begin{enumerate}
\item For each $\eps>0$ and integer $p\geq 1$ there exists $\delta>0$ such that
\[
\E_{\beta_c}\left[ |K \cap B_r|^p \mathbbm{1}\left(|K \cap B_r| \leq \delta r^{(d+\alpha)/2}\right)\right] \leq
\eps \E_{\beta_c}\left[ |K \cap B_r|^p \right]
\]
for every $r\geq 1$.
\item For every $\eps>0$ there exists $\lambda<\infty$ such that if we define $r(n)=n^{2/(d+\alpha)}$ then
\[
  \P_{\beta_c}\left(|K|\geq n \text{ \emph{but} } |K \cap (B_{\lambda r(n)}\setminus B_{\lambda^{-1} r(n)})| \leq n/2
  \right) \leq \eps \P_{\beta_c}(|K|\geq n)
\]
for every $n\geq 1$.
\end{enumerate}
\end{theorem}

This theorem also gives another precise sense in which the fractal dimension exponent $d_f$ is well-defined and equal to $(d+\alpha)/2$ in the effectively long-range, low-dimensional regime.


\subsection{Negligibility of mesoscopic clusters}
\label{subsec:negligibility_of_mesoscopic_clusters}

The goal of this section is to complete the proofs of \cref{thm:main_low_dim,thm:hyperscaling}.
 Both theorems will be deduced from the following proposition, which establishes a kind of ``hyperscaling postulate" for the first moment in the cut-off model. 

\begin{prop}[Negligibility of mesoscopic clusters]
\label{prop:negligibility_of_mesoscopic}
 If $d<3\alpha$ and \eqref{wCL} holds then for every $\eps>0$ there exists $\delta>0$ such that
\[
\E_{\beta_c,r}\left[ \min\left\{|K|, \delta r^{(d+\alpha)/2}\right\}\right]  \leq
\eps \E_{\beta_c,r}|K|
\]
for every $r\geq 1$.
\end{prop}

 We begin by proving the following simple differential inequality regarding truncated first moments. (See also \cref{III-sub:relations_between_the_beta_derivative_and_the_second_moment} for a different upper bound on the same quantity that is sharper in some regimes.)

\begin{lemma}
\label{lem:truncated_expectation_differential_inequality}
$\frac{\partial}{\partial r}\E_{\beta,r} [|K|\wedge m] \leq \beta |J'(r)||B_r|\E_{\beta,r} [|K|\wedge m]^2$ for every $\beta,r>0$.
\end{lemma}

\begin{proof}[Proof of \cref{lem:truncated_expectation_differential_inequality}]
We denote minima with $m$ using subscripts, so that $|K|_m=|K|\wedge m$. Using Russo's formula as usual, we can write
\begin{align*}
\frac{\partial}{\partial r}\E_{\beta,r} [|K|\wedge m] &= \beta |J'(r)| \E\left[\sum_{x\in K} \sum_{y\in B_r(x)} \mathbbm{1}(y\notin K) [(|K|+|K_y|)_m - |K|_m]\right]\\
&=\beta |J'(r)| \E\left[|K|\sum_{y\in B_r} \mathbbm{1}(y\notin K) [(|K|+|K_y|)_m - |K|_m]\right].
\end{align*}
To proceed, we claim that the inequality
\[
a[(a+b)_m-a_m] \leq a_mb_m
\]
holds for all $a,b,m\geq 0$. We verify this inequality via case analysis:
\begin{enumerate}
  \item If $a+b \leq m$ then $a[(a+b)_m-a_m]=ab=a_mb_m$ as required.
  \item If $a> m$ then $a[(a+b)_m-a_m]=0$, which is stronger than required.
  \item If $a\leq m$ and $a+b>m$ then $a[(a+b)_m-a_m]=a(m-a)\leq a_mb_m$ as required, where the final inequality holds since $b>m-a$.
\end{enumerate}
Thus, we have that
\begin{align*}
\frac{\partial}{\partial r}\E_{\beta,r} [|K|\wedge m] &\leq\beta |J'(r)| \E_{\beta,r}\left[|K|_m \sum_{y\in B_r} \mathbbm{1}(y\notin K) |K_y|_m\right],
\end{align*}
and the claim follows from \cref{I-lem:BK_disjoint_clusters_covariance} (as restated in \eqref{eq:BK_disjoint_clusters_covariance}).
\end{proof}

We next apply \cref{lem:truncated_expectation_differential_inequality} to prove the following lemma, which states roughly that if truncating the first moment at some volume $m$ does not change its order at some scale $r$, then at every much smaller scale the truncated first moment must admit a ``mean-field lower bound'' as in \cref{I-cor:mean_lower_bound} (restated here as \eqref{eq:mean_lower_bound}).

\begin{lemma}
\label{lem:mean-field_lower_bound_for_truncated_expectation}
For every $\eps>0$ there exists $\delta>0$ and $r_0<\infty$ such that the chain of implications
\begin{multline*}
\left(\E_{\beta_c,r} [|K|\wedge m] \geq \eps \E_{\beta_c,r} |K| \right) \Rightarrow \left(\E_{\beta_c,r} [|K|\wedge m] \geq \frac{\eps \beta_c}{2\alpha} r^\alpha \right)
\\ \Rightarrow \left( \E_{\beta_c,\ell}[|K| \wedge m] \geq (1-\eps)\frac{\alpha}{\beta} \ell^\alpha \text{ for every $\ell \leq \delta r$}\right)
\end{multline*}
holds for every $r \geq r_0$ and $m\geq 1$.
\end{lemma}

\begin{proof}[Proof of \cref{lem:mean-field_lower_bound_for_truncated_expectation}]
To lighten notation we write $\E_r=\E_{\beta_c,r}$.
The first implication is an immediate consequence of \eqref{eq:mean_lower_bound}; we focus on the second.
We can rewrite the differential inequality of \cref{lem:truncated_expectation_differential_inequality} as
\[\frac{d}{dr} \frac{1}{\E_r [|K|\wedge m]} \geq -\beta_c |J'(r)| |B_r| \sim -\beta_c r^{-\alpha-1}.\]
Integrating this inequality yields that
\[
\frac{1}{\E_\ell [|K|\wedge m]}  \leq \frac{1}{\E_r [|K|\wedge m]}+ \int_\ell^r \beta_c t^{-d-\alpha-1} |B_t|\dif t \sim \frac{\beta_c}{\alpha} \ell^{-\alpha} + \frac{1}{\E_r [|K|\wedge m]}-\frac{\beta_c}{\alpha}r^{-\alpha}
\]
for every $r\geq \ell$, where the asymptotic estimate holds uniformly as $\ell\to\infty$ and $r\geq \ell$ is arbitrary. This is easily seen to imply the claim, since the second two terms on the right hand side are much smaller than the first when $r\gg \ell$ and $\E_r [|K|\wedge m]$ is of order at least $r^\alpha$.
\end{proof}

Next, we prove a differential inequality for the second moment $\E_{\beta_c,r}|K|^2$ in terms of the truncated first moment $\E_{\beta_c,r}|K \wedge m|$. This is the first place in this section where we make use of the assumption that $d<3\alpha$ and that \eqref{wCL} holds.

\begin{lemma}
\label{lem:lower_bounding_second_moment_derivative_with_truncation} 
If $d<3\alpha$ and \eqref{wCL} holds then there exists a constant $C$ such that
\[
\frac{d}{dr} \E_{\beta_c,r}|K|^2 \geq 3 \beta_c |J'(r)||B_r| \Biggl(1- C \frac{m r^{\frac{\alpha-d}{2}}}{\E_{\beta_c,r}[|K|\wedge m]}\Biggr) \E_{\beta_c,r}[|K|\wedge m] \E_{\beta_c,r}[|K|^2]
\]
for every $r,m\geq 1$.
\end{lemma}

\begin{proof}[Proof of \cref{lem:lower_bounding_second_moment_derivative_with_truncation}] 
We continue to write $\E_r=\E_{\beta_c,r}$. It suffices to consider the case $m\leq  r^{\frac{d+\alpha}{2}}$, the case $m\geq r^{\frac{d+\alpha}{2}}$ holding vacuously since $\frac{d}{dr} \E_{r}|K|^2\geq 0$ and $\E_r [|K|\wedge m] \leq\E_r|K| \asymp r^\alpha$ by \cref{lem:first_moment_CL}. As usual, we can write
\begin{align}
\frac{d}{dr}\E_r|K|^2 &= \beta_c |J'(r)| \left(
\E_r\left[|K| \sum_{y\in B_r} \mathbbm{1}(y\notin K) |K_y|^{2} \right]
+
2\E_r\left[|K|^{2} \sum_{y\in B_r} \mathbbm{1}(y\notin K) |K_y| \right]\right)\nonumber\\
&=3\beta_c |J'(r)| \E_r\left[|K|^{2} \sum_{y\in B_r} \mathbbm{1}(y\notin K) |K_y| \right],
\end{align}
where the second equality follows by the mass-transport principle. We can therefore use  \eqref{eq:BK_disjoint_clusters_covariance2} to bound
\begin{align}
\frac{d}{dr}\E_r|K|^2 &\geq 3\beta_c |J'(r)| \E_r\left[|K|^{2} \sum_{y\in B_r} \mathbbm{1}(y\notin K) (|K_y|\wedge m) \right]\label{eq:truncated_diff_ineq_expanded}\\
&\geq 
3\beta_c |J'(r)| |B_r| \E_r[|K|^{2}] \E_r[|K|\wedge m]
-3\beta_c |J'(r)| \E_r\left[|K|^{2}(|K|\wedge m) |K\cap B_r| \right].\nonumber
\end{align}
Using the trivial inequality
\[
|K|^{2}(|K|\wedge m) |K\cap B_r| \leq |K|^3(|K|\wedge m) \leq m |K|^3+m^3 (|K|\wedge m)
\]
we obtain that
\begin{align}
\E_r\left[|K|^{2}(|K|\wedge m) |K\cap B_r| \right] &\leq m\E_r|K|^3+m^3\E_r[|K|\wedge m]
\nonumber\\
&\asymp  \frac{mr^{\frac{\alpha-d}{2}}}{\E_r[|K|\wedge m]} |B_r|\E_r[|K|\wedge m] \E_r|K|^2
+\left(\frac{m}{r^{\frac{d+\alpha}{2}}}\right)^3|B_r|\E[|K|\wedge m] \E_r|K|^2
\nonumber\\
&\preceq \frac{mr^{\frac{\alpha-d}{2}}}{\E_r[|K|\wedge m]} |B_r| \E_r[|K|\wedge m] \E_r|K|^2,
\label{eq:truncation_diff_ineq_almost_done}
\end{align}
where we used that $\E_r|K|^3 \asymp r^{\frac{d+\alpha}{2}} \E_r|K|^2$ and $|B_r|\E_r|K|^2 \asymp (r^{\frac{d+\alpha}{2}})^3$ in the second line and the assumption that $m\leq r^{\frac{d+\alpha}{2}}$ and the inequality $\E_r[|K|\wedge m]\preceq r^\alpha$ in the third line: all three of these moment bounds hold due to \cref{cor:low_dim_moments}, which requires the assumption \eqref{wCL}. The claim follows by substituting \eqref{eq:truncation_diff_ineq_almost_done} into \eqref{eq:truncated_diff_ineq_expanded}.
\end{proof}

\begin{proof}[Proof of \cref{prop:negligibility_of_mesoscopic}]
We continue to write $\E_r=\E_{\beta_c,r}$. We begin by proving the claim concerning the first moment. Let $\eps,\delta>0$ and let $\mathscr{R}(\eps,\delta)=\{r\geq 1: \E_r[|K|\wedge \delta r^{\frac{d+\alpha}{2}}] \geq \eps r^\alpha\}$. By \cref{lem:mean-field_lower_bound_for_truncated_expectation}, there exists a constant $\lambda_0=\lambda_0(\eps)\geq 1$ such that if we define $\lambda_0^{-1}\mathscr{R}(\eps,\delta)=\{r\geq 1: \lambda_0r \in \mathscr{R}(\eps,\delta) \}$ 
 then
\[
\E_r [|K| \wedge \delta (2\lambda_0 r)^{\frac{d+\alpha}{2}}] \geq (1-\eps) \frac{\alpha}{\beta_c}r^\alpha
\]
for every $r\in \lambda_0^{-1}\mathscr{R}(\eps,\delta)$. (We stress that $\lambda_0$ does \emph{not} depend on the choice of $\delta$.) Thus, for each $\eps>0$ there exists $r_0=r_0(\eps)<\infty$ such that if $r\geq r_0$ then $|B_r|\geq (1-\eps)^{1/2} r^d$ and $|J'(r)|\geq(1-\eps)^{1/2}r^{-d-\alpha-1}$ and if $r$ also belongs to $\lambda_0^{-1}\mathscr{R}(\eps,\delta)$ then
\[
\frac{d}{dr} \E_r|K|^2 \geq 3 \alpha (1-\eps)^2 r^{-1}\Biggl(1- \frac{C \beta_c\delta (2\lambda_0)^{\frac{d+\alpha}{2}}}{(1-\eps)\alpha}\Biggr)  \E_r|K|^2.
\]
Since $\lambda_0$ does not depend on the choice of $\delta$, it follows that for each $\eps>0$ there exists $\delta_0(\eps)>0$ such that if $r\in \lambda_0^{-1}\mathscr{R}(\eps,\delta_0(\eps))$ is at least $r_0$ then
\[
\frac{d}{dr} \E_r|K|^2 \geq 3 \alpha (1-\eps)^3 r^{-1} \E_r|K|^2.
\]
Consider the constant
\[a= \frac{1}{2}\cdot \frac{d+3\alpha}{2\alpha} + \frac{1}{2} \cdot 3\]
which satisfies $a<3$ since $d<3\alpha$, 
 and let $\eps_0>0$ be maximal such that $3(1-\eps_0)^3\geq a$. Fix $0<\eps\leq \eps_0$ and let $\delta=\delta_0(\eps)$, so that
\[
\frac{d}{dr} \log \E_r|K|^2 \geq a \alpha r^{-1} 
\]
for every $r\geq r_0(\eps)$ in $\lambda_0^{-1}(\eps)\mathscr{R}(\eps,\delta_0(\eps))$. Integrating this inequality yields that
\[
\frac{\E_{r_2}|K|^2}{\E_{r_1}|K|^2} \geq \left(\frac{r_2}{r_1}\right)^{a \alpha}
\]
whenever $r_2\geq r_1 \geq r_0$ are such that $[r_1,r_2]\subseteq \lambda_0^{-1}\mathscr{R}(\eps,\delta_0(\eps))$. Since we also have by \cref{cor:low_dim_moments} that $\E_{r}|K|^2 \asymp r^{\alpha+\frac{d+\alpha}{2}}$ and by definition of $a$ that $(d+\alpha)/2 < a \alpha$, it follows that there exists a constant $C_1=C_1(\eps) > 1$ such that if $r_2 \geq r_1 \geq r_0$ are such that $[r_1,r_2]\subseteq \lambda_0^{-1}\mathscr{R}(\eps,\delta_0(\eps))$ then $r_2/r_1\leq C$. Thus, it follows from the definition of $\lambda_0^{-1}\mathscr{R}$ that if $r_2 \geq r_1 \geq \lambda_0 r_0$ are such that $[r_1,r_2] \subseteq \mathscr{R}$ then $r_2/r_1\leq C$ also. As such, for each sufficiently large $r$ there exists $r\leq r'\leq Cr$ not belonging to $\mathscr{R}(\eps,\delta_0(\eps))$, so that
\[
\E_r [|K| \wedge \delta_0(\eps) C^{\frac{d+\alpha}{2}} r^{\frac{d+\alpha}{2}}] \leq \E_{r'}[|K| \wedge \delta_0(\eps) (r')^{\frac{d+\alpha}{2}}] \leq \eps C^\alpha r^\alpha.
\]
This is easily seen to imply the claim. 
\end{proof}

\cref{thm:main_low_dim} will be deduced from the following more general theorem requiring \eqref{wCL} instead of \eqref{CL} (the latter implying the former when $d<3\alpha$ by \cref{lem:CL_and_wCL}).

\begin{theorem}
\label{thm:main_low_dim_wCL}
If $d<3\alpha$ and 
 \eqref{wCL} holds then 
\[
\P_{\beta_c}(|K|\geq n) \asymp 
n^{-(d-\alpha)/(d+\alpha)} 
\qquad \text{ and } \qquad \E_{\beta_c}|K \cap B_r|^p  \asymp_p 
r^{\alpha+(p-1)\frac{d+\alpha}{2}}
\]
for every $n,r\geq 1$ and integer $p\geq 1$.
 In particular, the exponents $\delta$ and $d_f$ are well-defined and given by $\delta=(d+\alpha)/(d-\alpha)$ and $d_f=(d+\alpha)/2$ respectively.
\end{theorem}

This theorem will in turn be deduced from 
 \cref{thm:max_cluster_size_LD,prop:negligibility_of_mesoscopic} using \cref{I-prop:up-to-constants_volume-tail}, which we restate here for the reader's convenience.

\begin{prop}[Up-to-constants volume-tail asymptotics from moment asymptotics; restated from \cref{I-prop:up-to-constants_volume-tail}]
\label{prop:up-to-constants_volume-tail}
Suppose that $\E_{\beta_c,r}|K| \asymp r^\alpha$ as $r\to \infty$ and let $f:(0,\infty)\to(0,\infty)$ be an increasing function whose inverse $f^{-1}$ is doubling.
\begin{enumerate}
  \item If there exist positive constants $c$ and $C$ such that $\hat \P_{\beta_c,r}(cf(r) \leq |K| \leq Cf(r))\geq c$ for every $r\geq 1$ then
  \[
\P_{\beta_c}(|K|\geq n) \succeq \frac{f^{-1}(n)^\alpha}{n}
  \]
  for every $n\geq 1$.
  \item If for each $\eps>0$ there exists $\delta>0$ such that $\hat \P_{\beta_c,r}(|K| \leq \delta f(r)) \leq \eps$ for every $r\geq 1$ then
  \[
\P_{\beta_c}(|K|\geq n) \preceq  \frac{f^{-1}(n)^\alpha}{n}
\]
for every $n\geq 1$.
\end{enumerate}
\end{prop}

We are now ready to prove \cref{thm:main_low_dim,thm:hyperscaling}.

\begin{proof}[Proof of \cref{thm:main_low_dim,thm:main_low_dim_wCL}]
The claim concerning the moments $\E_{\beta_c}|K \cap B_r|^p$ has already been proven in \cref{cor:low_dim_moments}. The claim concerning the tail of the volume is an immediate consequence of \cref{thm:max_cluster_size_LD}, \cref{prop:negligibility_of_mesoscopic}, and \cref{I-prop:up-to-constants_volume-tail} (applied with $f(r)=r^{(d+\alpha)/2}$).
\end{proof}

\begin{proof}[Proof of \cref{thm:hyperscaling}]
For item 1, we have by Cauchy-Schwarz that
\begin{align*}
  &\E_{\beta_c}\left[|K\cap B_r|^p \mathbbm{1}(|K\cap B_r|\leq \delta r^{(d+\alpha)/2})\right] 
  \\&\hspace{5.5cm}\leq \left(\E_{\beta_c}|K\cap B_r|^{2p-2}\right)^{1/2}\left(\E_{\beta_c}|K\cap B_r| \mathbbm{1}(|K\cap B_r|\leq \delta r^{(d+\alpha)/2})\right)^{1/2}
  \\
  &\hspace{5.5cm}\leq 
  \left(\E_{\beta_c}|K\cap B_r|^{2p-2}\right)^{1/2}\left(\E_{\beta_c}\min\{|K|, \delta r^{(d+\alpha)/2})\}\right)^{1/2},
\end{align*}
for every $r,\delta>0$, where in the second inequality we used that $x \mathbbm{1}(x \leq y) \leq \min\{x,y\}$ and that the latter function is increasing in $x$. 
Using \cref{thm:main_low_dim} we obtain that
\begin{equation*}
\E_{\beta_c}\min\{|K|, \delta r^{(d+\alpha)/2})\} = \sum_{n=1}^{\lfloor \delta r^{(d+\alpha)/2}\rfloor} \P_{\beta_c}(|K|\geq n) \\\asymp 
 \sum_{n=1}^{\lfloor \delta r^{(d+\alpha)/2}\rfloor} n^{-(d-\alpha)(d+\alpha)}
 \preceq \delta^{2\alpha/(d+\alpha)} r^\alpha.
\end{equation*}
Applying \cref{cor:low_dim_moments} to bound $\E_{\beta_c}|K\cap B_r|^{2p-2}$ we deduce that for each $\eps>0$ there exists $\delta>0$ such that
\[
  \E_{\beta_c}\left[|K\cap B_r|^p \mathbbm{1}(|K\cap B_r|\leq \delta r^{(d+\alpha)/2})\right] \leq \eps (r^{\alpha+(2p-2)(d+\alpha)/2})^{1/2}(r^\alpha)^{1/2} = \eps r^{\alpha+(p-1)(d+\alpha)/2},
\]
from which the claim follows by a second application of \cref{cor:low_dim_moments}.

We now turn to item 2; this proof is closely related to the proof of \cref{prop:up-to-constants_volume-tail}.
In light of \cref{thm:main_low_dim}, it suffices to prove that for each $\eps>0$ there exists $\lambda<\infty$ such that 
\begin{equation}
\label{eq:small_ball}
  \P_{\beta_c}\left(|K\cap B_{\lambda^{-1} n^{2/(d+\alpha)}}| \geq n/2 \right) \leq \eps n^{-(d-\alpha)/(d+\alpha)}
\end{equation}
and
\begin{equation}
\label{eq:large_ball}
  \P_{\beta_c}\left(|K\setminus B_{\lambda n^{2/(d+\alpha)}}| \geq n/2 \right) \leq \eps n^{-(d-\alpha)/(d+\alpha)}.
\end{equation}
The first inequality \eqref{eq:small_ball} follows immediately from \eqref{eq:Sak_upper_restate} (i.e., the main result of \cite{hutchcroft2022sharp}) and Markov's inequality, which yield that
\[
   \P_{\beta_c}\left(|K\cap B_{\lambda^{-1} n^{2/(d+\alpha)}}| \geq n/2 \right) \preceq n^{-1} \E_{\beta_c}|K\cap B_{\lambda^{-1} n^{2/(d+\alpha)}}| \preceq \lambda^{-\alpha} n^{-(d-\alpha)/(d+\alpha)}.
\]
For the second inequality \eqref{eq:large_ball}, it is established in the proof of \cref{I-prop:up-to-constants_volume-tail} (see \eqref{I-eq:volume_tail_union_bound}) that there exists a constant $C$ such that
\begin{equation}
\label{eq:volume_tail_union_bound}
\E_{\beta_c}[1-e^{-h|K|}] \leq 
 \E_{\beta_c,r}[1-e^{-h|K|}] + 
Cr^{-\alpha}  \E_{\beta_c,r}\left[|K|e^{-h|K|}\right]
 \E_{\beta_c}[1-e^{-h|K|}],
\end{equation}
for every $h>0$ and $r\geq 1$. It follows from 
 \cref{lem:first_moment_CL,prop:negligibility_of_mesoscopic} that for each $\eps>0$ there exists $\lambda<\infty$ such that if $0<h\leq 1$ and $r=\lambda h^{-2/(d+\alpha)}$ then $Cr^{-\alpha}  \E_{\beta_c,r}\left[|K|e^{-h|K|}\right] \leq \eps$ and hence that
\begin{equation}
\label{eq:volume_tail_union_bound}
  \E_{\beta_c}[1-e^{-h|K|}] \leq (1-\eps)^{-1} \E_{\beta_c,\lambda h^{-2/(d+\alpha)}}[1-e^{-h|K|}] 
\end{equation}
for every $0<h\leq 1$. Fix one such $\lambda<\infty$, let $0<h\leq 1$, and write $r=h^{-2/(d+\alpha)}$. Letting $K\supseteq \tilde K$ be the clusters of the origin in the standard monotone coupling of the two measures $\E_{\beta_c}$ and $\E_{\beta_c,\lambda r}$ and letting $\cG$ be an independent ghost field of intensity $h$, the inequality \eqref{eq:volume_tail_union_bound} is equivalent to the inequality
\begin{equation}
\label{eq:hyperscaling_volume_tail1}
  \P(K \cap \cG \neq \emptyset \text{ but } \tilde K \cap \cG = \emptyset) \leq \eps \P(K \cap \cG \neq \emptyset),
\end{equation}
where we write $\P$ for the joint law of $K$, $\tilde K$, and $\cG$. 
On the other hand, it follows from \eqref{wCL} that for each $\eps>0$ there exists a constant $C$ such that
\begin{equation}
\label{eq:hyperscaling_volume_tail2}
  \P(\tilde K \cap \cG \setminus B_{C\lambda r} \neq \emptyset) \leq h\E |\tilde K \setminus B_{C\lambda r}| \leq \eps h \E |\tilde K| \preceq \eps h^{(d-\alpha)/(d+\alpha)},
\end{equation}
where the first inequality is an instance of Markov's inequality. Putting \eqref{eq:hyperscaling_volume_tail1} and \eqref{eq:hyperscaling_volume_tail2} together yields by a union bound that
\begin{multline*}
  \P(K \cap \cG \neq \emptyset \text{ but } K \cap \cG \cap B_{C\lambda r} = \emptyset) \leq 
  \P(K \cap \cG \neq \emptyset \text{ but } \tilde K \cap \cG = \emptyset) + \P(\tilde K \cap \cG \setminus B_{C\lambda r} \neq \emptyset)
  \\
  \leq \eps \P(K \cap \cG \neq \emptyset) + \eps h^{(d-\alpha)/(d+\alpha)} \preceq \eps h^{(d-\alpha)/(d+\alpha)}
\end{multline*}
for every $0<h \leq 1$, which is easily seen to imply \eqref{eq:large_ball} since $\eps>0$ was arbitrary.
\end{proof}

\subsection{$k$-point function hyperscaling}
\label{subsec:_k_point_function_hyperscaling}

In this section we complete the proof of \cref{thm:k_point_S} by proving the following proposition. 

\begin{prop}
\label{prop:k-point_hyperscaling_lower}
If $d<3\alpha$ and \eqref{wCL} holds then the critical $k$-point function satisfies
\[
\tau_{\beta_c}(A) \succeq_{|A|} \operatorname{sweep}(A)^{-\frac{d-\alpha}{2}}
\]
for every finite set $A \subseteq \Z^d$ with $|A|\geq 2$.
\end{prop}

(Note that this proposition also completes the proof of \cref{thm:CL_Sak} in the low-dimensional case $d<3\alpha$. In general, the relevant lower bound on the two-point function will be established in \cref{lem:two_point_lower} via a unified argument that does not distinguish between low and high dimensions.)

\medskip

 To see that the relationship $\tau_{\beta_c}(A)\asymp_{|A|} \operatorname{sweep}(A)^{-(d-\alpha)/2}$ should be thought of as a hyperscaling relation,
  note that, in low effective dimensions, $\operatorname{sweep}(A)^{-(d-\alpha)/2}$ is comparable to the probability that each point $x\in A$ belongs to a ``typical large cluster'' on an appropriately chosen scale $R_x$, with the different scales $R_x$ chosen ``minimally'' such that (assuming there are ``$O(1)$ typical large clusters'' in each box) all the points in $A$ belong to the \emph{same} cluster with good probability on the event that each of them belongs to the typical large cluster at their own scale. 

 \medskip

To further emphasize the point that \cref{thm:k_point_S} encodes a hyperscaling relation, rather than something specific to long-range percolation, we also prove the following analogous theorem for bond percolation on the square lattice. (The proof works for bond percolation on any planar lattice satisfying appropriate Russo-Seymour-Welsh estimates, and can be extended to site percolation using an appropriate generalization of the Gladkov inequality.) It is presumably possible to extend this theorem to nearest-neighbour percolation in dimension $d=3,4,5$ under appropriate hyperscaling postulates (as was done for other hyperscaling relations in following \cite{MR1716769}); this may be related to the results of \cite{kiss2014large}. (See \cite{camia2024conformal} for more precise results for the scaling limit of site percolation on the triangular lattice.)

\begin{theorem}
\label{thm:planar_k_point}
Consider nearest-neighbour Bernoulli bond percolation on the square lattice $\Z^2$ with $p=p_c=1/2$. If there exists $\eta>0$ such that
$
  \P_{p_c}(x\leftrightarrow y) \asymp \|x-y\|_2^{-\eta}
  $
for every $x \neq y\in \Z^2$ then
\[
\tau_{p_c}(x_1,\ldots,x_k) \asymp_k S(x_1,\ldots,x_k)^{-\eta/2}
\]
for every distinct $k$-tuple of points $x_1,\ldots,x_k \in \Z^2$. 
\end{theorem}


We will in fact prove \cref{thm:planar_k_point} \emph{before} proving \cref{prop:k-point_hyperscaling_lower} since the proof is easier.

\begin{proof}[Proof of \cref{thm:planar_k_point}]
The upper bound follows immediately from \cref{cor:geometric_Gladkov}; it suffices to prove the matching lower bound.
Let $\{x_1,\ldots,x_k\}=A$; all constants in this proof may depend on~$k$. Pick an arboresence achieving the minimum in the definition of $\operatorname{sweep}(A)$ (see \cref{def:sweep_and_spread}) with associated diameters $R_x=\operatorname{diam}(A_x)$. As proven in \cite{kesten1987scaling}, the two-point function and one-arm probabilities are related via
\[
  \P_{1/2}(x\leftrightarrow y) \asymp \P_{1/2}(0 \leftrightarrow \partial [-\|x-y\|_2,\|x-y\|_2]^2)^2
\]
for every $x,y\in \Z^2$ so that, by our assumption of power-law two-point function scaling $\P_{1/2}(x\leftrightarrow y) \asymp \|x-y\|^{-\eta}$, the one-arm probabilities satisfy
\[\P_{1/2}(0 \leftrightarrow \partial [-r,r]^2) \asymp r^{-\eta/2}\]
for all $r\geq 1$. Let $\mathscr{A}$ denote the event that each point $x\in A$ is connected to the boundary of the box $x+[-8R_x,8R_x]^2$. Harris's inequality implies that
\[
  \P_{1/2}(\mathscr{A}) \geq \prod_{x\in A} \P_{1/2}(0 \leftrightarrow \partial [-8R_x,8R_x]^2) \asymp_k \operatorname{sweep}(A)^{-\eta/2}.
\]
Now, let $\mathscr{B}$ denote the event that for each $x\in A$, there is an open circuit in the annulus $x+([-4R_x,4R_x]\setminus [-2R_x,2R_x])$. By the Russo-Seymour-Welsh theorem (see \cite[Section 11.7]{grimmett2010percolation}) and Harris's inequality, the probability of $\mathscr{B}$ is at least $c^k$ for some positive constant $c$. As such, a further application of Harris's inequality yields that
\[
    \P_{1/2}(\mathscr{A} \cap \mathscr{B}) \succeq_k \operatorname{sweep}(A)^{-\eta/2}.
\]
The claimed lower bound now follows from \cref{lem:sweep_recursion} and the topological observation that the points of $A$ must all belong to the same cluster on the event $\mathscr{A}\cap \mathscr{B}$.
\end{proof}

We now turn to \cref{prop:k-point_hyperscaling_lower}.
We begin by proving the following lemma, which states roughly that if the origin is in a typical large cluster on scale $R$ then it is likely to be in a typical large cluster on each smaller scale $r$.

\begin{lemma}
\label{lem:large_clusters_are_large_on_meso_scales}
If $d<3\alpha$ and \eqref{wCL} holds then for each $\eps>0$ there exists $\delta>0$ and $\lambda<\infty$ such that if $R\geq r$ then
\begin{equation*}
  \P\Bigl(|K^{r} \cap B_{ \lambda r}| \leq \delta r^{(d+\alpha)/2} \;\Big|\; |K^{R} \cap B_{R}| \geq \eps R^{(d+\alpha)/2}\Bigr) \leq \eps ,
\end{equation*}
where $K^r$ and $K^R$ denote the two clusters of the origin in the standard monotone coupling of the measures $\P_{\beta_c,r}$ and $\P_{\beta_c,R}$.
\end{lemma}

\begin{proof}[Proof of \cref{lem:large_clusters_are_large_on_meso_scales}]
It follows from \cref{thm:hyperscaling} that for each $\eps'>0$ there exists a constant $\delta=\delta(\eps')$ such that
\begin{align*}
  &\E\left[|K^{r}| \mathbbm{1}(|K^{r} \cap B_{ \delta^{-1}r}| 
  \leq \delta r^{(d+\alpha)/2})\right] \\
  &\hspace{3cm}\leq  \E |K^{r} \setminus B_{ \delta^{-1}r}| + \E\left[|K^{r} \cap B_{\delta^{-1}r}| \mathbbm{1}(|K^{r} \cap B_{ \delta^{-1}r}|\leq \delta r^{(d+\alpha)/2})\right]
  \leq \eps' r^\alpha.
\end{align*}
Considering the first edge that does not belong to $K^r$ used by the path connecting the origin to a vertex of $K^R$, we have by the standard BK argument that
\begin{multline*}
  \E\left[|K^{R}| \mathbbm{1}(|K^{r} \cap B_{ \delta^{-1}r}|\leq \delta r^{(d+\alpha)/2})\right] 
  \\\preceq (1+r^{-\alpha} \E|K^R|)\E\left[|K^{r}| \mathbbm{1}(|K^{r} \cap B_{ \delta^{-1}r}|\leq \delta r^{(d+\alpha)/2})\right]
   \preceq \eps' R^\alpha,
\end{multline*}
where $r^{-\alpha}$ is an upper bound on the order of the number of edges adjacent to the origin that flip from closed to open when passing between the measures $\P_{\beta_c,r}$ and $\P_{\beta_c,R}$, and the $1$ inside the term $(1+r^{-\alpha} \E|K^R|)$ accounts for the number of vertices that already belong to $K^r$.
Letting $\eps>0$, it follows by Markov's inequality and \eqref{eq:low_dim_locally_large_cluster} that
\begin{multline*}
  \P\left(|K^{r} \cap B_{ \delta^{-1}r}| \leq \delta r^{(d+\alpha)/2} \mid |K^{R} \cap B_{R}| \geq \eps R^{(d+\alpha)/2}\right)  \\\leq \frac{\E\left[|K^{R}| \mathbbm{1}(|K^{r} \cap B_{ \delta^{-1}r}|\leq \delta r^{(d+\alpha)/2})\right] }{\eps R^{(d+\alpha)/2}\P(|K^{R} \cap B_{R}| \geq \eps R^{(d+\alpha)/2})} 
  \preceq \frac{\eps' R^\alpha}{\eps^{1-(d-\alpha)/(d+\alpha)}R^{(d+\alpha)/2}R^{-(d-\alpha)/2}} = \eps' \eps^{-2\alpha/(d+\alpha)},
\end{multline*}
and the claim follows by taking $\eps'= c\eps^{1+2\alpha/(d+\alpha)}$ for a sufficiently small constant $c>0$.
\end{proof}

\begin{proof}[Proof of \cref{prop:k-point_hyperscaling_lower}]
Fix a set $A$ and an arborsence attaining the minimum in the definition of the sweep (\cref{def:sweep_and_spread}) and let $R_a =\operatorname{diam}(A_a)$ for each $a\in A$.
All constants in this proof may depend on $|A|$, and we omit this dependence from our notation. Enumerate $A=\{a_1,\ldots,a_n\}$ so that the sequence $(R_i)_{i=1}^n:=(R_{a_i})_{i=1}^n$ is increasing, $a_n$ is the root of the arboresence, and the function $\sigma:\{1,\ldots,n-1\}\to \{2,\ldots,n\}$ sending each vertex $a_i$ to its parent $a_{\sigma(i)}$ in the arboresence satisfies $\sigma(i) > i$ for each $1\leq i\leq n-1$. (Such an enumeration can be constructed by listing the points of $a$ in increasing order of their associated radii $r_a$, which are always increasing as we move towards the root of the arboresence, and breaking ties using the arboresence when necessary.) Let $\eps = c_2|A|^{-2}>0$ and let $\delta>0$ and $1\leq \lambda<\infty$ be as in \cref{lem:large_clusters_are_large_on_meso_scales}. Define the sequence $R_i'=(4\lambda)^{i} R_i$, so that $R_i\asymp_{|A|} R_i'$ for every $1\leq i \leq n$ and $R_{i+1}' \geq 4\lambda R_i'$ for each $1\leq i \leq n-1$.

\medskip

Consider the standard monotone coupling of the measures $\E_{\beta_c,r}$ with $r\geq 0$, so that we have an increasing family of percolation configurations $(\omega_r)_{r\geq 0}$ such that $\omega_r$ has law $\P_{\beta_c,r}$ for each $r\geq 0$ and $\omega_\infty=\bigcup_{r\geq 0}\omega_r$ has law $\P_{\beta_c}$. We write $\P$ and $\E$ for probabilities and expectations taken with respect to these coupled configurations. For each $1\leq i,j\leq n$ let $K_{i,j}$ and $K_{i,j}'$ be the clusters of $a_i$ in the configurations associated to the scales $R_j'$ and $2R_j'$ respectively.
We first note that, by \cref{thm:max_cluster_size_LD} and \cref{lem:large_clusters_are_large_on_meso_scales}, there exists a constant $c_1>0$ such that
\begin{equation}
\label{eq:low_dim_locally_large_cluster}
\P_{\beta}\bigl(|K_{i,j} \cap B_{R_j'}(a_i)| \geq c_1 R_j^{(d+\alpha)/2} \text{ for every $1\leq j\leq i$} \bigr) \succeq R_i^{-(d-\alpha)/2}
\end{equation}
for every $1\leq i \leq n$.
For each $1\leq i \leq n$ let $\mathscr{A}_i$ and $\mathscr{G}_i$ denote the events that the following conditions hold:
\begin{enumerate}
  \item $\mathscr{A}_i$: For each  $j \geq i$, the intersection $|K_{j,i} \cap B_{R_i}(a_j)|$ has size at least $\delta R_i^{(d+\alpha)/2}$.
    \item $\mathscr{G}_i$: If $a_i$ is the child of $a_j$ for some $j > i$, then $a_i$ belongs to $K_{j,i}'$.
\end{enumerate}
Let $\mathcal{F}_i$ be the sigma-algebra generated by the configurations $(\omega_r:r\leq R_i')$, so that $\mathscr{A}_i$ is measurable with respect to $\mathcal{F}_j$ for every $j\geq i$ and $\mathscr{G}_i$ is measurable with respect to $\mathcal{F}_j$ for every $j>i$. Considering the possibility that the clusters $K_{j,i}$ and $K_{i,i}$ merge via the direct addition of an edge when we pass to the configuration $K'_{j,i}$ yields that there exists a positive constant $c_2$ such that
\[
  \P(\mathscr{G}_i \mid \mathcal{F}_i) \geq c_2 \mathbbm{1}(\mathscr{A}_i)
\]
for every $1\leq i \leq n-1$. Applying the Harris-FKG inequality to the conditional measure $\cF_i$ (which is a product measure), it follows that
\[
\P\left(\left(\cap_{j=i+1}^n \mathscr{A}_{j}\right) \cap \mathscr{G}_i \mid \mathcal{F}_i\right)
   \geq c_2   \P\left(\cap_{j=i+1}^n \mathscr{A}_{j} \mid \mathcal{F}_i\right) \mathbbm{1}(\mathscr{A}_i)
\]
for every $1\leq i \leq n-1$ and hence by the tower property of conditional expectation that
\begin{align*}
  \P\left(\left(\cap_{j=1}^{n}\mathscr{A}_{j}\right) \cap \left(\cap_{j=1}^i \mathscr{G}_j\right)\right) 
  &= \E\left[\P\left(\left(\cap_{j=1}^{i+1}\mathscr{A}_{j}\right) \cap \left(\cap_{j=1}^n \mathscr{G}_j\right) \mid \mathcal{F}_i\right) \right]
\\&=
   \E\left[\mathbbm{1}\left(\left(\cap_{j=1}^{i}\mathscr{A}_{j}\right) \cap \left(\cap_{j=1}^{i-1} \mathscr{G}_j\right)\right)\P\left(\left(\cap_{j=i+1}^n \mathscr{A}_{j}\right) \cap \mathscr{G}_i \mid \mathcal{F}_i\right) \right]
  \\
  & \geq 
   c_2\E\left[\mathbbm{1}\left(\left(\cap_{j=1}^{i}\mathscr{A}_{j}\right) \cap \left(\cap_{j=1}^{i-1} \mathscr{G}_j\right)\right)\P\left(\cap_{j=i+1}^n \mathscr{A}_{j} \mid \mathcal{F}_i\right) \right] 
   \\&= c_2  \P\left(\left(\cap_{j=1}^{n}\mathscr{A}_{j}\right) \cap \left(\cap_{j=1}^{i-1} \mathscr{G}_j\right)\right)
\end{align*}
for every $1\leq i \leq n-1$. It follows inductively that
\[
  \P\left(\left(\cap_{j=1}^{n}\mathscr{A}_{j}\right) \cap \left(\cap_{j=1}^{n-1} \mathscr{G}_j\right)\right) \geq c_2^n \P\left(\cap_{j=1}^{n}\mathscr{A}_{j}\right) \succeq_n \prod_{i=1}^n R_i^{-(d-\alpha)/2},
\]
where the final inequality follows from \eqref{eq:low_dim_locally_large_cluster} and the Harris-FKG inequality (applied to the intersection of the events in \eqref{eq:low_dim_locally_large_cluster} over the different vertices of $A$). Using \cref{lem:sweep_recursion}, the claim follows since the points $a_1,\ldots,a_n$ all belong to the same cluster on the event $\cap_{j=1}^{n-1}\mathscr{G}_j$.
\qedhere
\end{proof}

\begin{proof}[Proof of \cref{thm:k_point_S}]
This follows immediately from \cref{cor:geometric_Gladkov} (for the upper bound) and \cref{prop:k-point_hyperscaling_lower} (for the lower bound).
\end{proof}

\begin{remark}
The same proof also establishes similar up-to-constants estimates on the critical $k$-point function for long-range percolation on the \emph{hierarchical lattice} when $d<3\alpha$; the inputs needed for the proof (i.e., the upper bound on the two-point function and the analogues of \cref{thm:max_cluster_size_LD} and \cref{thm:hyperscaling}) are established in \cite{hutchcrofthierarchical,hutchcroft2022critical}.
\end{remark}

\begin{remark}
\label{remark:HD_spread}
The correct analogue of \cref{thm:k_point_S} in the effectively \emph{high-dimensional} regime should be that
\begin{equation}
\label{eq:HD_spread}
  \tau_{\beta_c}(A) \asymp_{|A|} \operatorname{spread}(A)^{-d+2(\alpha \wedge 2)} \operatorname{diam}(A)^{-(\alpha \wedge 2)}
\end{equation}
for $\alpha\neq 2$ (and with logarithmic corrections when $\alpha=2$). The corresponding statement for the high-dimensional \emph{uniform spanning tree} on $\Z^d$ is proven in \cite{BeKePeSc04,hutchcroft2017component}, and the arguments presented there show that this estimate is equivalent to the two-point asymptotics $\tau_{\beta_c}(x,y) \asymp \|x-y\|^{-d+(\alpha \wedge 2)}$ whenever the tree-graph inequalities give the correct order estimates on the $k$-point function. Since the spread and the diameter now appear with \emph{different} powers, one loses the connection to the sweep and hence to the M\"obius-covariant quantity $S$ when working in high effective dimension. (Both powers in \eqref{eq:HD_spread} become equal to $(d+\alpha)/4=\alpha$ when $d=3\alpha<6$ as one would expect for continuity of exponents.) 
In the critical-dimensional case $d=3\alpha<6$ the $k$-point function should be given up-to-constants by
\[
  \tau_{\beta_c}(A) \asymp_{|A|}  \log(\operatorname{diam}(A)) \prod_{x\in A} \frac{1}{R_x^\alpha (\log R_x)^{1/2}}
\]
where $(R_x)_{x\in A}$ are the diameters of the sets $(A_x)_{x\in A}$ associated to the arboresence achieving the minimum in the definition of $\operatorname{sweep}(A)$; this is proven for the two-point and three-point functions in \cref{III-thm:pointwise_three_point}.
\end{remark}


\section{Slightly subcritical scaling relations}
\label{sec:subcritical}

In this section we prove our results concerning scaling relations in the slightly subcritical regime, \cref{thm:Fisher_relation,thm:subcritical_volume}. We also complete the proof of \cref{thm:CL_Sak} by establishing the lower bound on the critical two-point function under \eqref{CL}. All of the results proven here apply to the entire effectively long-range regime (as defined by \eqref{CL}) and are not specific to low effective dimension. As mentioned in the introduction, the proofs in the critical-dimensional case $d=3\alpha<6$ will sometimes invoke results not proven until the third paper of the series.

\subsection{The slightly subcritical two-point function}

In this section we prove \cref{thm:Fisher_relation} and complete the proof of \cref{thm:CL_Sak}. We begin with the following simple estimate on the susceptibility. Recall that we will sometimes use the notation $\chi(\beta)=\E_\beta|K|$ for the susceptibility.

\begin{lemma}
\label{lem:subcritical_susceptibility}
  If \eqref{CL} holds then $\E_{\beta,r}|K| \asymp \min\{r^\alpha,\chi(\beta)\}$ for every $\beta\leq \beta_c$ and $r\geq 1$.
\end{lemma}

\begin{proof}[Proof of \cref{lem:subcritical_susceptibility}]
The upper bound is a trivial consequence of \cref{lem:first_moment_CLb}. The lower bound follows from the differential inequality
\[
  \frac{\partial}{\partial r} \E_{\beta,r}|K| \leq \beta |J'(r)||B_r| (\E_{\beta,r}|K|)^2 \preceq r^{-\alpha-1} (\E_{\beta,r}|K|)^2,
\]
which we can integrate to obtain that
\[
  \frac{1}{\E_{\beta,r}|K|} -\frac{1}{\E_{\beta}|K|} =  - \int_r^\infty \frac{\partial}{\partial s} \frac{1}{\E_{\beta,s}|K|} \dif s 
  \preceq \int_r^\infty r^{-\alpha-1} \dif s  \preceq r^{-\alpha}
\]
for every $0<\beta\leq \beta_c$ and $r\geq 1$.
\end{proof}

We next prove that the defining feature of \eqref{CL} (i.e., that the $L^p$ radius of gyration associated to the measure $\P_{\beta_c,r}$ is $O(r)$ for every $p$) extends uniformly to the subcritical regime.

\begin{prop}
\label{prop:subcritical_gyration}
 If \eqref{CL} holds then
$\E_{\beta,r} \sum_{x\in K}\|x\|^{p} \preceq_p r^{p} \E_{\beta,r}|K|$ 
for every integer $p\geq 1$, $0<\beta\leq \beta_c$, and $r\geq 1$.
\end{prop}

Applying the same argument used in \eqref{eq:CLp_application}, it follows from \cref{lem:subcritical_susceptibility,prop:subcritical_gyration} that there exists a constant $C\geq 1$ such that
\[
  \E_{\beta,r}|K\cap B_{Cr}| \geq \frac{1}{2}\E_{\beta,r}|K| \succeq r^\alpha
\]
for every $\beta_c/2\leq \beta<\beta_c$ and $r\leq \xi^*(\beta)$ and hence that
\begin{equation}
  \E_{\beta,r}|K\cap B_{r}| \geq \E_{\beta,r/C} |K\cap B_{C(r/C)}| \succeq r^\alpha
  \label{eq:subcritical_susceptibility_inside_box}
\end{equation}
for every $\beta_c/2\leq \beta<\beta_c$ and $r\leq \xi^*(\beta)$ also.

\begin{proof}[Proof of \cref{prop:subcritical_gyration}]
We prove the claim by induction on $p \geq 0$, the base case $p=0$ being trivial. For $p\geq 1$, we have by Russo's formula that
\[
  \frac{\partial}{\partial r} \E_{\beta,r}\sum_{x\in K}\|x\|^{p} = \beta |J'(r)| \E_{\beta,r}\left[ \sum_{x\in K} \sum_{y\in B_r} \sum_{z \in K_y} \mathbbm{1}(x\nleftrightarrow y) \|z\|^p \right]
\]
for every $0<\beta\leq \beta_c$ and $r>0$.
Writing $\|z\|\leq \|z-y\|+\|y-x\|+\|x\|$, expanding out the resulting trinomial, and using the BK inequality, we obtain that
\[
  \frac{\partial}{\partial r} \E_{\beta,r}\sum_{x\in K}\|x\|^{p} \preceq_p r^{-\alpha-1} \sum_{a+b+c=p}  \Bigl(\E_{\beta,r}\sum_{x\in K}\|x\|^a\Bigr)\Bigl(\E_{\beta,r}\sum_{x\in K}\|x\|^b\Bigr) r^c,
\]
where we bounded $\|y-x\|\leq r$ to simplify the resulting expression.
(See \cref{I-subsec:the_full_displacement_distribution} for more precise versions of the same calculation.) Using the induction hypothesis we obtain that there exists a constant $C_p$ such that
\[
  \frac{\partial}{\partial r} \E_{\beta,r}\left[\sum_{x\in K}\|x\|^{p}\right] \leq C_p r^{-\alpha-1} \E_{\beta,r}|K|  \E_{\beta,r}\left[\sum_{x\in K}\|x\|^{p}\right] + C_p r^{p-\alpha-1} (\E_{\beta,r} |K|)^2
\]
for all $0<\beta \leq \beta_c$ and $r\geq 1$ and hence that
\[
  \frac{\partial}{\partial r} \left(r^{-p} \E_{\beta,r}\left[\sum_{x\in K}\|x\|^{p}\right] \right) \leq (C_p r^{-\alpha} \E_{\beta,r}|K| -p)  r^{-p-1}   \E_{\beta,r}\left[\sum_{x\in K}\|x\|^{p}\right] + r^{-\alpha-1} (\E_{\beta,r} |K|)^2
\]
for all $0<\beta \leq \beta_c$ and $r\geq 1$. Since $\E_{\beta,r}|K| \leq \chi(\beta)$ and $\xi^*(\beta):=\chi(\beta)^{1/\alpha}$, it follows that there exists a constant $\lambda_p<\infty$ such that if $r\geq \lambda_p \xi^*(\beta)$ then the coefficient of the first term on the right hand side of this inequality is negative, so that
\[
  \frac{\partial}{\partial r} \left(r^{-p} \E_{\beta,r}\left[\sum_{x\in K}\|x\|^{p}\right] \right) \leq  C_p r^{-\alpha-1} \chi(\beta)^2
\]
for all $r\geq \lambda_p \xi^*(\beta)$. Integrating this inequality  yields that
\begin{align*}
  \sup_{r\geq \lambda_p \xi^*(\beta)} r^{-p} \E_{\beta,r}\left[\sum_{x\in K}\|x\|^{p}\right] &\leq (\lambda_p \xi^*(\beta))^{-p} \E_{\beta,\lambda_p \xi^*(\beta)}\left[\sum_{x\in K}\|x\|^{p}\right] + \frac{C_p}{\alpha} (\lambda_p \xi^*(\beta))^{-\alpha} \chi(\beta)^2 \\
  &\preceq_p   \chi(\beta),
\end{align*}
yielding a bound of the correct order when $r\geq \lambda_p \xi^*(\beta)$.
On the other hand, if $r\leq \lambda_p \xi^*(\beta)$ then a bound of the correct order follows immediately from \cref{lem:subcritical_susceptibility} and the definition of \eqref{CL} which yield that \[\E_{\beta,r} \sum_{x\in K}\|x\|^{p} \leq \E_{\beta_c,r} \sum_{x\in K}\|x\|^{p} \preceq_p r^{p} \E_{\beta_c,r}|K| \asymp_p r^p \E_{\beta,r}|K|\]
when  $r\leq \lambda_p \xi^*(\beta)$. This completes the induction step and hence the proof.
\end{proof}

We next prove the upper bound of \cref{thm:Fisher_relation}. Recall that we define $\xi^*(\beta):=\chi(\beta)^{1/\alpha}$.

\begin{lemma}
\label{lem:slightly_subcritical_two_point_upper}
  If \eqref{CL} holds then 
  \[
    \P_{\beta}(x\leftrightarrow y) \preceq 
    \|x-y\|^{-d+\alpha}\left(1\vee \frac{\|x-y\|}{\xi^*(\beta)}\right)^{-2\alpha}
  \] for every $\beta\leq \beta_c$ and $r\geq 1$.
\end{lemma}

\begin{proof}[Proof of \cref{lem:slightly_subcritical_two_point_upper}] It suffices by \cref{lem:CL_Sak_upper} to prove that
  \[
    \P_{\beta}(x\leftrightarrow y) \preceq 
    \chi(\beta)^2 \|x-y\|^{-d-\alpha}
  \]
 for every $\beta_c/2\leq \beta \leq \beta_c$ and $\|x-y\|\geq \xi^*(\beta)$. Fix $0<\beta \leq \beta_c$ and one such pair $x,y\in \Z^d$.
We have as in the proof of \cref{lem:CL_Sak_upper} that
\[
  \frac{\partial}{\partial r} \P_{\beta,r}(x\leftrightarrow y) 
  \leq \beta |J'(r)| \sum_{a,b\in \Z^d} \mathbbm{1}(\|a-b\|\leq r) \P_{\beta,r}(x\leftrightarrow a)\P_{\beta,r}(b\leftrightarrow y) 
\]
for every $r>0$. For $r\geq \|x-y\|/8\geq \xi^*(\beta)/8$, we can ignore the constraint that $\|a-b\|\leq r$ to obtain that
\[
  \frac{\partial}{\partial r} \P_{\beta,r}(x\leftrightarrow y) 
  \preceq r^{-d-\alpha-1} \chi(\beta)^2
\]
and hence that
\begin{equation}
  \P_{\beta}(x\leftrightarrow y) - \P_{\beta,\|x-y\|/8}(x\leftrightarrow y) \preceq \chi(\beta)^2 \int_{\|x-y\|}^\infty r^{-d-\alpha-1}\dif r \asymp \chi(\beta)^2 \|x-y\|^{-d-\alpha}.
\label{eq:large_r_two_point_integrated}
\end{equation}
We now bound the derivative for $r \leq \|x-y\|/8$.
As in the proof of \cref{lem:CL_Sak_upper}, if $r\leq \|x-y\|/8$ then every pair $a,b \in \Z^d$ contributing to this sum must satisfy at least one of the inequalities $\|x-b\|\geq \|x-y\|/4$ or $\|y-a\|\geq \|x-y\|/4$, so that
\[
    \frac{\partial}{\partial r} \P_{\beta,r}(x\leftrightarrow y) 
  \preceq r^{-d-\alpha-1} (\E_{\beta,r}|K|)(
  \E_{\beta,r} |K\setminus B_{\|x-y\|/4}|).
\]
Using \cref{prop:subcritical_gyration} in the same way we applied \eqref{CL} to obtain \eqref{eq:CLp_application} we can bound
\[\E_{\beta,r} |K\setminus B_{\lambda r}| \preceq_p \lambda^{-p}\E_{\beta,r} |K|, \]
for every $p\geq 1$ and $r,\lambda\geq 1$, so that
\[
   \frac{\partial}{\partial r} \P_{\beta,r}(x\leftrightarrow y) \preceq_p r^{-d-\alpha-1} (\E_{\beta,r}|K|)^2 \left(\frac{\|x-y\|}{r}\right)^{-p}
\] 
for every $p\geq 1$ and $r\leq \|x-y\|/8$. Taking $p=\lceil d+\alpha+1\rceil>d+\alpha$ and integrating this inequality between $\xi^*(\beta)$ and $\|x-y\|/8$ yields that
\[
  \P_{\beta,\|x-y\|/8}(x\leftrightarrow y) - \P_{\beta,\xi^*(\beta)}(x\leftrightarrow y) \preceq \chi(\beta)^2 \|x-y\|^{-p} \int_{\xi^*(\beta)}^{\|x-y\|/8} r^{p-d-\alpha-1} \dif r \asymp \chi(\beta)^2 \|x-y\|^{-d-\alpha}.
\]
Putting this together with \eqref{eq:large_r_two_point_integrated} and \cref{lem:CL_Sak_upper} we obtain that there exists a decreasing function $h:(0,\infty)\to (0,1]$ decaying faster than any power such that
\begin{align*}
  \P_{\beta}(x\leftrightarrow y) &\leq [\P_{\beta}(x\leftrightarrow y) - \P_{\beta,\xi^*(\beta)}(x\leftrightarrow y)]
  +
  \P_{\beta,\xi^*(\beta)}(x\leftrightarrow y)
  \\&\preceq \chi(\beta)^2 \|x-y\|^{-d-\alpha} + \|x-y\|^{-d+\alpha} h(\|x-y\|/\xi^*(\beta))\\
  &\preceq \chi(\beta)^2 \|x-y\|^{-d-\alpha}
\end{align*}
as claimed, where we bounded $h(t) \preceq (1\vee t)^{-2\alpha}$ in the final inequality.
\qedhere
\end{proof}

To complete the proofs of \cref{thm:CL_Sak,thm:Fisher_relation} it suffices to prove the following lemma.

\begin{lemma}
\label{lem:two_point_lower}
If \eqref{CL} holds then
\[
  \P_{\beta}(x\leftrightarrow y) \succeq \|x-y\|^{-d+\alpha} \left(1 \vee \frac{\|x-y\|}{\xi^*(\beta)}\right)^{-2\alpha}.
\]
for every $x\neq y \in \mathbb{Z}^d$ and $\beta_c/2\leq \beta \leq \beta_c$, where we set $\xi^*(\beta_c)=\infty$.
\end{lemma}

The proof of this lemma will apply the following lemma, which is a simple consequence of \cref{lem:CL_Sak_upper} and the four-point Gladkov inequality \eqref{eq:4pt_Gladkov}.

\begin{lemma}
\label{lem:Gladkov_application_for_two_point_lower}
If \eqref{CL} holds then there exists a decreasing function $h:(0,\infty) \to (0,1]$ decaying faster than any power such that
\[
  \E_{\beta_c,r}\left[|K_x \cap B_r(x)||K_y \cap B_r(x)| \mathbbm{1}(x\leftrightarrow y) \right] \preceq r^{d+\alpha} \|x-y\|^{-d+\alpha} h(\|x-y\|/r)
\]
for every $r\geq 1$ and $x,y\in \Z^d$ with $\|x-y\|\geq 4r$.
\end{lemma}

(Note that this bound will suffice for our applications even in the regime $d\geq 3\alpha$ where it should not be optimal.)

\begin{proof}[Proof of \cref{lem:Gladkov_application_for_two_point_lower}]
We apply the two-point upper bound of \cref{lem:CL_Sak_upper} and the four-point case of the higher Gladkov inequality \cref{thm:higher_Gladkov} to bound $\tau_{\beta,r}(x,y,a,b)$ with $\|x-a\|,\|y-b\| \leq \|x-y\|/4$ by
\[
  \tau_{\beta,r}(x,y,a,b) \preceq 
   \|x-a\|^{-(d-\alpha)/2}\|y-b\|^{-(d-\alpha)/2} \|x-y\|^{-(d-\alpha)} h(\|x-y\|/r)
\]
for some decreasing function $h:(0,\infty) \to (0,1]$ decaying faster than any power: here we used here that $ \|x-y\|$, $\|x-b\|$, $\|y-a\|$, and $\|a-b\|$ are all of the same order and that the main contribution to \eqref{eq:4pt_Gladkov} is from the diagram in which the double edges are between the two closer pairs of points $\{x,a\}$ and $\{y,b\}$. Summing this estimate over the two balls yields the claim.
\end{proof}

\begin{proof}[Proof of \cref{lem:two_point_lower}]
Fix $\beta_c/2\leq \beta \leq \beta_c$ and $x\neq y\in \Z^d$. Let $r\leq \min\{\|x-y\|,\xi^*(\beta)\}/4$ be a parameter to be determined and consider the standard monotone coupling of $\P_{\beta,r}$ and $\P_{\beta}$, writing $\omega_r$ and $\omega$ for the two relevant configurations. By the splitting property of Poisson processes, we can write $\omega$ as the union of $\omega_r$ with an independent configuration $\omega'$ in which each edge $\{x,y\}$ is included with probability $1-e^{-\beta (J-J_r)(x,y)}$.
 Write $\P$ for the joint law of $(\omega_r,\omega',\omega)$ and let $Z$ denote the (random) set of pairs $(a,b)$ in $\Z^d$ such that $\|a-x\|\leq r$, $\|b-y\|\leq r$, $\{a,b\}$ is open in $\omega'$, $x$ is connected to $a$ in $\omega_r$, and $y$ is connected to $b$ in $\omega_r$. This definition ensures that
$\P_\beta(x\leftrightarrow y) \geq \P(|Z|>0)$.
It follows from \eqref{eq:subcritical_susceptibility_inside_box} and the Harris-FKG inequality that
\begin{equation}
  \E |Z| \succeq (\E_{\beta,r}|K\cap B_r|)^2 \|x-y\|^{-d-\alpha} \succeq r^{2\alpha} \|x-y\|^{-d-\alpha} = \|x-y\|^{-d+\alpha} \left(\frac{\|x-y\|}{r}\right)^{-2\alpha}
  \label{eq:Z_mean_lower}
\end{equation}
On the other hand, we also have by \cref{lem:Gladkov_application_for_two_point_lower} that
\begin{multline*}
\E |Z| \mathbbm{1}(x\leftrightarrow y \text{ in $\omega_r$}) \leq \|x-y\|^{-d-\alpha} \E_{\beta,r}\left[ |K_x \cap B_r(x)| |K_y \cap B_r(y)| \mathbbm{1}(x \leftrightarrow y)\right] \\
\preceq \|x-y\|^{-d-\alpha} r^{d+\alpha} \P_{\beta,r}(x\leftrightarrow y)  \preceq \|x-y\|^{-d+\alpha} h(\|x-y\|/r)
\end{multline*}
for some function $h$ decaying faster than any polynomial, where we used that $r\leq \|x-y\|/4$ in the final inequality.
Comparing this to the lower bound \eqref{eq:Z_mean_lower}, it follows that there exists a positive constant $c$ such that if $r\leq c\|x-y\|$ then
\begin{equation}
\label{eq:good_Z_lower}
  \E |Z| \mathbbm{1}(x\nleftrightarrow y \text{ in $\omega_r$}) = 
  \E |Z|  - \E |Z| \mathbbm{1}(x\leftrightarrow y \text{ in $\omega_r$}) \succeq \E |Z| \succeq r^{2\alpha} \|x-y\|^{-d-\alpha}.
\end{equation}
Fix $r=\min\{c,1/4\} \cdot \min\{\|x-y\|,\xi^*(\beta)\}$ so that this lower bound holds. Applying the BK inequality to the configuration $\omega_r$ yields the upper bound
\begin{multline}
  \E \left[\binom{|Z|}{2}\mathbbm{1}(x\nleftrightarrow y \text{ in $\omega_r$})\right] \preceq \|x-y\|^{-2d-2\alpha} \E_{\beta,r} \left[|K_x \cap B_r(x)|^2 |K_y \cap B_r(y)|^2 \mathbbm{1}(x\nleftrightarrow y) \right]\\ \preceq 
  \|x-y\|^{-2d-2\alpha} r^{d+3\alpha} = \|x-y\|^{-d+\alpha} \left(\frac{\|x-y\|}{r}\right)^{-d-3\alpha} \leq \|x-y\|^{-d+\alpha} \left(\frac{\|x-y\|}{r}\right)^{-2\alpha},
  \label{eq:good_Z2_upper}
\end{multline}
where we used \eqref{eq:Sak_upper_restate} and \eqref{eq:M_r_upper_restate} in the first inequality on the second line and used that $r\leq \|x-y\|$ in the final inequality.
 Since
\[\E\left[|Z|^2\mathbbm{1}(x\nleftrightarrow y \text{ in $\omega_r$})\right] = 2\E\left[\binom{|Z|}{2}\mathbbm{1}(x\nleftrightarrow y \text{ in $\omega_r$})\right] + \E\left[|Z|\mathbbm{1}(x\nleftrightarrow y \text{ in $\omega_r$})\right] \]
it follows from \eqref{eq:good_Z_lower} and \eqref{eq:good_Z2_upper} that
\[
  \E\left[|Z|^2\mathbbm{1}(x\nleftrightarrow y \text{ in $\omega_r$})\right] \preceq \E\left[|Z|\mathbbm{1}(x\nleftrightarrow y \text{ in $\omega_r$})\right]
\]
and hence by Cauchy-Schwarz that
\begin{align*}
  \P_\beta(x\leftrightarrow y) \geq \P(|Z|>0, x\nleftrightarrow y \text{ in $\omega_r$}) &\geq \frac{(\E[|Z|\mathbbm{1}(x\nleftrightarrow y \text{ in $\omega_r$})])^2}{\E[|Z|^2\mathbbm{1}(x\nleftrightarrow y \text{ in $\omega_r$})]} 
  \\&\succeq \E\left[|Z|\mathbbm{1}(x\nleftrightarrow y \text{ in $\omega_r$})\right] \succeq \|x-y\|^{-d+\alpha} \left(\frac{\|x-y\|}{r}\right)^{-2\alpha}
\end{align*}
as claimed.
\end{proof}

\begin{proof}[Proof of \cref{thm:CL_Sak}]
The theorem follows immediately from \cref{lem:CL_Sak_upper,lem:two_point_lower}.
\end{proof}

\begin{proof}[Proof of \cref{thm:Fisher_relation}]
The theorem follows immediately from \cref{lem:slightly_subcritical_two_point_upper,lem:two_point_lower}.
\end{proof}

\subsection{The slightly subcritical volume distribution}

In this section we prove \cref{thm:subcritical_volume}.
Our first lemma establishes a version of the scaling hypothesis \eqref{eq:scaling_hypothesis_moments} for the measures $\P_{\beta_c,r}$ (rather than $\P_\beta$ with $\beta<\beta_c$ as claimed in \cref{thm:subcritical_volume}) as a corollary of various other results proven throughout the series.

\begin{lemma}
\label{lem:moments_scaling_hypothesis_cut_off} If \eqref{CL} holds then
\[
  \E_{\beta_c,r}|K|^p \asymp_p \left(\frac{\E_{\beta_c,r}|K|^2}{\E_{\beta_c,r}|K|}\right)^{p-1} \E_{\beta_c,r}|K|
\]
for each integer $p\geq 1$ and every $r\geq 1$.
\end{lemma}

\begin{proof}[Proof of \cref{lem:moments_scaling_hypothesis_cut_off}]
The case $d<3\alpha$ was established in \cref{thm:CL_volume_moments}. For $d \geq 3\alpha$ we have by \cref{cor:HD_CL} that \eqref{CL} holds if and only if $\alpha<2$, and in this case the claim follows from \cref{I-thm:hd_moments_main} when $d>3\alpha$ and from \cref{III-thm:critical_dim_moments_main} when $d=3\alpha$ (alternatively one can use the less precise results \cref{I-thm:critical_dim_moments_main_slowly_varying} and \cref{III-thm:critical_dim_hydro} in this case).
\end{proof}

Our next goal is to prove the upper bound of \cref{thm:subcritical_volume}.
Recall that we defined $\zeta^*(\beta)$ by
\[
  \zeta^*(\beta) := \frac{\E_{\beta_c,\xi^*(\beta)}|K|^2}{\E_{\beta_c,\xi^*(\beta)}|K|} 
\]
which under \eqref{CL} satisfies
\begin{equation}
  \zeta^*(\beta)  \asymp 
\begin{cases}
\xi^*(\beta)^{2\alpha} \\
(\log \xi^*(\beta))^{-1/2} \xi^*(\beta)^{2\alpha} \\
  \xi^*(\beta)^{(d+\alpha)/2} 
  \end{cases}
  \!\!\!\!\asymp
  \begin{cases}
\chi(\beta)^{2} &\qquad d>3\alpha \qquad \text{(LR HD)} \\
(\log \chi(\beta))^{-1/2} \chi(\beta)^{2} &\qquad d=3\alpha \qquad \text{(LR CD)} \\
  \chi(\beta)^{(d+\alpha)/2\alpha} &\qquad d<3\alpha \qquad \text{(LR LD)}
  \end{cases}
  \label{eq:zeta*_asymptotics}
\end{equation}
by \cref{I-thm:hd_moments_main} (in the case $d>3\alpha$), \cref{III-thm:critical_dim_moments_main} (in the case $d=3\alpha$), and \cref{lem:sCL_moments} (in the case $d<3\alpha$). (Here we are also making implicit use of \cref{cor:HD_CL}.)

\begin{lemma}
\label{lem:moments_scaling_hypothesis_upper}
 If \eqref{CL} holds then
\[
\E_{\beta}|K|^p \preceq_p 
 \zeta^*(\beta)^{p-1} \chi(\beta)
\]
for every integer $p\geq 1$ and $r\geq 1$.
\end{lemma}

\begin{proof}[Proof of \cref{lem:moments_scaling_hypothesis_upper}]
We prove the claim by induction on $p$, the case $p=1$ being trivial. Suppose that $p>1$ and that the claim has been proven for all smaller $p$.
We have by \eqref{eq:pth_moment_derivative} and the BK inequality that
\[
  \frac{\partial}{\partial r} \E_{\beta,r}|K|^p \preceq_p r^{-\alpha-1} \E_{\beta,r}|K| \E_{\beta,r}|K|^p + r^{-\alpha-1}\sum_{\ell=1}^{p-1} \E_{\beta,r}|K|^{\ell+1} \E_{\beta,r}|K|^{p-\ell},
\]
and applying the induction hypothesis yields that there exists a constant $C_p$ such that
\[
  \frac{\partial}{\partial r} \E_{\beta,r}|K|^p \leq C_p r^{-\alpha-1} \chi(\beta) \E_{\beta,r}|K|^p + C_p r^{-\alpha-1} \chi(\beta)^2 
  \zeta^*(\beta)^{p-1}.
\]
Introducing the integrating factor $e^{-\int_{\xi^*(\beta)}^r C_p \chi(\beta) s^{-\alpha-1} \dif s}$, it follows that
\[
  \frac{\partial}{\partial r} \left( e^{-\int_{\xi^*(\beta)}^r C_p \chi(\beta) s^{-\alpha-1} \dif s} \E_{\beta,r}|K|^p\right) \leq  C_p  e^{-\int_{\xi^*(\beta)}^r C_p \chi(\beta) s^{-\alpha-1} \dif s} r^{-\alpha-1} \chi(\beta)^2 
  \zeta^*(\beta)^{p-1}
\]
for every $r\geq \xi^*(\beta)$ and hence that
\begin{align*}
  \E_\beta |K|^p &\preceq_p  e^{\int_{\xi^*(\beta)}^r C_p \chi(\beta) s^{-\alpha-1} \dif s} \left(\E_{\beta,\xi^*(\beta)}|K|^p + \int_{\xi^*(\beta)}^\infty e^{-\int_{\xi^*(\beta)}^r C_p \chi(\beta) s^{-\alpha-1} \dif s} r^{-\alpha-1} \chi(\beta)^2 \zeta^*(\beta)^{p-1} \dif r \right)
  \\
  &\preceq_p \E_{\beta,\xi^*(\beta)}|K|^p + \zeta^*(\beta)^{p-1}\chi(\beta) \preceq_p \zeta^*(\beta)^{p-1}\chi(\beta),
\end{align*}
where we bounded $\E_{\beta,\xi^*(\beta)}|K|^p \leq \E_{\beta_c,\xi^*(\beta)}|K|^p \preceq \zeta^*(\beta)^{p-1}\chi(\beta)$ using \cref{lem:moments_scaling_hypothesis_cut_off} in the final inequality. This completes the induction step and hence the proof.
\end{proof}

\begin{proof}[Proof of \cref{thm:subcritical_volume}]
We begin with the claim that 
\begin{equation}
\label{eq:volume_scaling_hypothesis_proof}
\E_{\beta}|K|^p \asymp_p \zeta^*(\beta)^{p-1} \chi(\beta).
\end{equation}
Given \cref{lem:subcritical_susceptibility,lem:moments_scaling_hypothesis_upper}, it suffices to prove the lower bound 
 $\E_{\beta}|K|^2 \succeq \chi(\beta)\zeta^*(\beta)$,
with the corresponding lower bound on higher moments following from this bound by H\"older's inequality. We will prove this estimate using the formula for $\zeta^*(\beta)$ in terms of $\chi(\beta)$ given in \eqref{eq:zeta*_asymptotics} together with the critical volume tail asymptotics (assuming \eqref{CL})
\begin{equation}
\label{eq:volume_tail_all_dimensions}
  \P_{\beta_c}(|K| \geq n) \asymp \begin{cases}
n^{-1/2}&\qquad d>3\alpha \qquad \text{(LR HD)} \\
(\log n)^{1/4} n^{-1/2} &\qquad d=3\alpha \qquad \text{(LR CD)} \\
  n^{-(d-\alpha)/(d+\alpha)} &\qquad d<3\alpha \qquad \text{(LR LD)}
  \end{cases}
\end{equation}
as proven in \cref{I-thm:hd_moments_main} (for $d>3\alpha$), \cref{III-thm:critical_dim_moments_main} (for $d=3\alpha$), and \cref{thm:main_low_dim} ($d<3\alpha$).
Comparing this with \eqref{eq:zeta*_asymptotics} yields that
\begin{equation}
  \P_{\beta_c}(|K|\geq \zeta^*(\beta)) \asymp \frac{\chi(\beta)}{\zeta^*(\beta)}
\label{eq:volume_tail_zeta_formulation}
\end{equation}
for every $0\leq \beta<\beta_c$ in every case in which \eqref{CL} holds.
Write $\hat \E_\beta$ for the size-biased measure, so that $\hat \E_\beta|K| = \E_\beta|K|^2/\E_\beta|K|$. Applying Markov's inequality to the size-biased measure yields that
\[
  \E_\beta \left[|K| \mathbbm{1}(|K| \geq \lambda \hat \E_\beta|K| )\right]=\E_\beta|K| \hat\P_\beta (|K| \geq \lambda \hat \E_\beta|K| ) \leq \frac{1}{\lambda}\E_\beta|K|
\]
and hence that
\begin{equation}
\label{eq:biased_Markov}
  \E_\beta \left[|K| \mathbbm{1}(|K| \leq \lambda \hat \E_\beta|K| )\right]\geq \frac{\lambda-1}{\lambda}\E_\beta|K|.
\end{equation}
On the other hand, it follows from \eqref{eq:volume_tail_all_dimensions} that
\begin{align*}
  \E_\beta \left[|K| \mathbbm{1}(|K| \leq 2 \hat \E_\beta|K| )\right]
  &\leq
  \E_{\beta_c} \left[|K| \mathbbm{1}(|K| \leq 2 \hat \E_\beta|K| )\right] =\sum_{n=1}^{\lfloor 2\hat \E_{\beta}|K|\rfloor} \P_{\beta_c}(|K|\geq n)
   \\
  &\preceq 
  \begin{cases}
(\hat \E_\beta|K|)^{1/2} &\qquad d>3\alpha \qquad \text{(LR HD)} \\
(\log \hat \E_\beta|K|)^{1/4} (\hat \E_\beta|K|)^{1/2} &\qquad d=3\alpha \qquad \text{(LR CD)} \\
  (\hat \E_\beta|K|)^{2\alpha /(d+\alpha)} &\qquad d<3\alpha \qquad \text{(LR LD)}.
  \end{cases}
\end{align*}
Comparing this estimate with \eqref{eq:biased_Markov} and \eqref{eq:zeta*_asymptotics} we obtain that $\hat \E_\beta|K| \succeq \zeta^*(\beta)$ as claimed.

\medskip We now prove the claims concerning the subcritical volume tail. Using \eqref{eq:volume_scaling_hypothesis_proof} and \eqref{eq:volume_tail_zeta_formulation} we obtain the bound
\[
  \P_\beta(|K|\geq n) \leq n^{-p} \E_\beta|K|^p \preceq_p \frac{\chi(\beta)}{\zeta^*(\beta)} \left(\frac{n}{\zeta^*(\beta)}\right)^{-p} \preceq_p \P_{\beta_c}(|K|\geq \zeta^*(\beta))  \left(\frac{n}{\zeta^*(\beta)}\right)^{-p}
\]
for every $p\geq 1$, $\beta_c/2\leq \beta <\beta_c$, and $n\geq \zeta^*(\beta)$, which is easily seen to imply the claimed upper bound on $\P_\beta(|K|\geq n)$ since $p\geq 1$ is arbitrary and $\P_{\beta_c}(|K|\geq \zeta^*(\beta))$ is of the same order as a regularly varying function. For the lower bound, we can apply the Paley-Zygmund inequality to the size-biased measure to obtain that
\[
  \hat \P_\beta(|K|\geq \frac{1}{2} \hat \E_\beta|K|) \geq \frac{1}{4} \frac{(\hat \E_\beta |K|)^2}{\hat \E_\beta|K|^2} = 
  \frac{1}{4} \frac{(\E_\beta |K|^2)^2}{\E_\beta|K| \E_\beta|K|^3} \asymp 1,
\]
so that
\[
  \E_\beta\left[|K|\mathbbm{1}(|K|\geq \frac{1}{2} \hat \E_\beta|K|)\right] \succeq \chi(\beta)
\]
for every $\beta_c/2\leq \beta<\beta_c$. Combining this with \eqref{eq:biased_Markov} yields there exists a constant $C$ such that
\begin{multline}
  \P_{\beta}\left(\frac{1}{2}\hat \E_\beta |K| \leq |K| \leq C \hat \E_\beta |K|\right) \asymp \frac{1}{\hat \E_\beta |K|}\E_{\beta}\left[|K|\mathbbm{1}\Bigl(\frac{1}{2}\hat \E_\beta |K| \leq |K| \leq C \hat \E_\beta |K|\Bigr) \right]
  \\ \succeq \frac{\chi(\beta)}{\zeta^*(\beta)} \asymp \P_{\beta_c}\Bigl(|K|\geq \zeta^*(\beta)\Bigr) \asymp \P_{\beta_c}\Bigl(|K|\geq \frac{1}{2}\hat \E_\beta|K|\Bigr),
  \label{eq:subcritical_volume_debiasing}
\end{multline}
where we used that $\zeta^*(\beta)\asymp \frac{1}{2}\hat \E_\beta|K|$ and that $\P_{\beta_c}(|K|\geq n)$ of the same order as a regularly varying function in the final estimate. By the intermediate value theorem, for each $n\leq \frac{1}{2}\hat \E_{\beta}|K|$ there exists $\beta'\leq \beta$ such that $\frac{1}{2}\hat \E_{\beta'}|K|=n$, and applying \eqref{eq:subcritical_volume_debiasing} at this $\beta'$ yields the claimed estimate
 $\P_{\beta}(|K|\geq n) \geq \P_{\beta'}(|K|\geq n) \succeq \P_{\beta_c}(|K|\geq n)$ for every $\beta_c/2\leq \beta <\beta_c$ and $n\leq \frac{1}{2}\hat \E_{\beta}|K| \asymp \zeta^*(\beta)$.  
%
\end{proof}

\subsection*{Acknowledgements}
This work was supported by NSF grant DMS-1928930 and a Packard Fellowship for Science and Engineering. We thank Gordon Slade and Ed Perkins for helpful comments on an earlier draft.

\addcontentsline{toc}{section}{References}

 \setstretch{1}
 \footnotesize{
  \bibliographystyle{abbrv}
  \bibliography{unimodularthesis.bib}
  }

\end{document}